\documentclass[10pt,reqno]{amsart}

\usepackage{graphicx}
\usepackage{times}
\usepackage[colorlinks=true,linkcolor=blue,citecolor=blue]{hyperref}%
\usepackage{dsfont}

\usepackage{xcolor}
\usepackage{stmaryrd}
\usepackage{pdfrender,xcolor}
\usepackage{tikz}

\newtheorem{theorem}{Theorem}[section]
\newtheorem{lemma}[theorem]{Lemma}
\newtheorem{corollary}[theorem]{Corollary}
\newtheorem{prop}[theorem]{Proposition}
\theoremstyle{definition}
\newtheorem{definition}[theorem]{Definition}
\theoremstyle{remark}
\newtheorem{remark}[theorem]{Remark}
\numberwithin{equation}{section}

\usepackage[top=1.2in, bottom=1.2in, left=1.2in, right=1.2in]{geometry}

\usepackage[scr]{rsfso}
\usepackage[english]{babel}
\usepackage[utf8x]{inputenc}
\usepackage{fancyhdr}
\usepackage{amsmath}
\usepackage{amsthm}
\usepackage{amssymb}
\usepackage{eucal}
\usepackage{comment}
\usepackage{enumitem}
\setlist{leftmargin=*}
\usepackage[integrals]{wasysym}

\newcommand\nc{\newcommand}
\nc{\on}{\operatorname}
\nc{\E}{\mathbb{E}}
\nc{\R}{\mathbb R}
\nc{\C}{\mathbb C}
\nc{\Q}{\mathbb Q}
\nc{\Z}{\mathbb Z}
\nc{\N}{\mathbb N}
\nc{\F}{\mathbb F}
\nc{\T}{\mathbb{T}}
\nc{\wt}{\widetilde}
\nc{\ol}{\overline}
\nc{\short}[3]{0 \longrightarrow #1 \longrightarrow #2 \longrightarrow #3 \longrightarrow 0}
\nc{\pd}[2]{\frac{\partial #1}{\partial #2}}
\nc{\rnc}{\renewcommand}
\nc{\e}{\varepsilon}
\nc{\DMO}{\DeclareMathOperator}
\nc{\grad}{\nabla}
\nc{\fsp}{{}}
\nc{\fspp}{{}}

\newcommand{\abbr}[1]{{\sc\lowercase{#1}}}

\nc{\half}{\mathrm{half}}
\nc{\X}{\mathrm{X}}
\nc{\Y}{\mathrm{Y}}
\rnc{\t}{\mathrm{t}}
\nc{\x}{\mathrm{x}}
\nc{\y}{\mathrm{y}}
\nc{\s}{\mathrm{s}}
\nc{\z}{\mathrm{z}}
\nc{\w}{\mathrm{w}}
\rnc{\r}{\mathrm{r}}
\rnc{\a}{\mathrm{a}}
\rnc{\b}{\mathrm{b}}
\rnc{\k}{\mathrm{k}}
\nc{\open}{\mathrm{open}}
\rnc{\leq}{\leqslant}
\rnc{\geq}{\geqslant}
\rnc{\d}{\mathrm{d}}
\rnc{\O}{\mathrm{O}}
\rnc{\exp}{\mathbf{Exp}}
\newenvironment{nouppercase}{%
  \renewcommand{\uppercasenonmath}[1]{}}{}
\title{\fsp\Large {\abbr{KPZ} equation from open \abbr{ASEP} with general boundary asymmetry}}

\author{Kevin Yang}

\usepackage{setspace}
\begin{document}
\setstretch{1.0}
\fsp
\raggedbottom
\begin{nouppercase}
\maketitle
\end{nouppercase}
\begin{center}
\today
\end{center}

\begin{abstract}
\fspp We consider generalizations of open \abbr{ASEP} in the interval and half-space, where the speed of the reservoir dynamics can depend on the local particle configuration. We show that their height functions have a continuum limit given by the open \abbr{KPZ} equation. This removes the assumption of Liggett's condition in \cite{CS,P}, thus answering a question of Corwin \cite{C22} and Himwich \cite{H25}, and it removes the assumption of product invariant measures in \cite{GPS}. In the case of the interval, we also show convergence of the stationary measure for the height function increment process to that of the increment process for the open KPZ equation.
\end{abstract}

{\hypersetup{linkcolor=blue}
\setcounter{tocdepth}{1}
\tableofcontents}

\allowdisplaybreaks
\section{Introduction}
The \emph{open Kardar-Parisi-Zhang} (open \abbr{KPZ}) \emph{equation} is a conjecturally universal \abbr{SPDE} model \cite{C22} for fluctuation statistics in stochastic transport between reservoirs. It is a boundary-version of the \abbr{KPZ} equation introduced in \cite{KPZ}. In the bounded interval setting, the equation is given below, where {\small$\mathbf{A},\mathbf{B}\in\R$} and {\small$\lambda\neq0$} are fixed:
\begin{align}
\partial_{\t}\mathbf{h}_{\t,\X}&=\tfrac12\partial_{\X}^{2}\mathbf{h}_{\t,\X}-\tfrac{\lambda}{2}|\partial_{\X}\mathbf{h}_{\t,\X}|^{2}+\xi_{\t,\X},\quad(\t,\X)\in[0,\infty)\times[0,1],\label{eq:openkpz}\\
\partial_{\X}\mathbf{h}_{\t,0}&:=\mathbf{A}+\tfrac12 \quad\text{and}\quad \partial_{\X}\mathbf{h}_{\t,1}:=\mathbf{B}+\tfrac12. \label{eq:openkpzII}
\end{align}
Above, $\xi$ is the space-time white noise on {\small$[0,\infty)\times[0,1]$}, that is, a Gaussian field with (formal) covariance kernel {\small$\E\xi_{\t,\X}\xi_{\s,\Y}=\delta_{\t=\s}\delta_{\X=\Y}$}. The corresponding Cole-Hopf solution theory is given by {\small$\mathbf{h}:=-\lambda^{-1}\log\mathbf{Z}$}, where
\begin{align}
\partial_{\t}\mathbf{Z}_{\t,\X}&=\tfrac12\partial_{\X}^{2}\mathbf{Z}_{\t,\X}-\lambda\mathbf{Z}_{\t,\X}\xi_{\t,\X},\label{eq:openshe}\\
\partial_{\X}\mathbf{Z}_{\t,0}&=-\lambda\mathbf{A}\mathbf{Z}^{}_{\t,0}\quad\mathrm{and}\quad\partial_{\X}\mathbf{Z}_{\t,1}=-\lambda\mathbf{B}\mathbf{Z}^{}_{\t,1}.\label{eq:opensheII}
\end{align}
The discrepancy in boundary parameters (i.e. the additional {\small$1/2$} in \eqref{eq:openkpzII}) is due to a necessary renormalization procedure \cite{C22,GH19}. We note that the Cole-Hopf solution agrees with the regularity structures solution constructed in \cite{GH19}. To be precise, we interpret \eqref{eq:openshe}-\eqref{eq:opensheII} in its Duhamel form. In particular, we first let $\mathbf{H}$ be the Robin heat kernel, which solves
\begin{align*}
\partial_{\t}\mathbf{H}_{\s,\t,\X,\Y}&=\tfrac12\partial_{\X}^{2}\mathbf{H}_{\s,\t,\X,\Y}\quad\text{for } 0\leq\s\leq\t \ \text{and} \ \X,\Y\in[0,1],\\
\partial_{\X}\mathbf{H}_{\s,\t,0,\Y}&=-\lambda\mathbf{A}\cdot\mathbf{H}_{\s,\t,0,\Y}\quad\text{and}\quad \partial_{\X}\mathbf{H}_{\s,\t,1,\Y}=-\lambda\mathbf{B}\cdot\mathbf{H}_{\s,\t,1,\Y},\\
\mathbf{H}_{\s,\s,\X,\Y}&=\mathbf{1}_{\X=\Y} \quad\text{for all } \s\in[0,\infty) \ \text{and} \ \X,\Y\in[0,1].
\end{align*}
We then define {\small$\mathbf{Z}$} to be the unique adapted solution to the following integral equation:
\begin{align}
\mathbf{Z}_{\t,\X}={\textstyle\int_{[0,1]}}\mathbf{H}_{0,\t,\X,\Y}\mathbf{Z}_{0,\Y}\d\Y-\lambda{\textstyle\int_{0}^{\t}\int_{[0,1]}}\mathbf{H}_{\s,\t,\X,\Y}\mathbf{Z}_{\s,\Y}\xi_{\s,\Y}\d\Y\d\s.\nonumber
\end{align}
The open \abbr{KPZ} equation has witnessed tremendous attention as far as its statistical properties are concerned, but rigorous proof of its conjectured universality remains minimal. In \cite{CS,P,GPS}, it is derived as a continuum limit for the height function associated to a model known as open \abbr{ASEP}, which is the standard \abbr{ASEP} but on a discretization of the interval {\small$[0,1]$} equipped with interactions with reservoirs at the two boundary points. We explain this model in more detail shortly (see also Section \ref{section:model} for a precise formulation).

However, \cite{CS,P,GPS} all require a restrictive technical assumption on speeds of boundary dynamics which restricts to a codimension {\small$2$} subspace of possible open \abbr{ASEP} models. In \cite{CS,P}, the assumption is called ``Liggett's condition", in view of earlier work \cite{L75,L77} of Liggett. In \cite{GPS}, the assumption therein guarantees an explicit (and simple!) invariant measure for the open \abbr{ASEP} model (which is otherwise complicated and requires the ``matrix product ansatz" \cite{DEHP}). See also \cite{YPTRF}; it extends \cite{CS,P,GPS} to non-simple versions of open \abbr{ASEP} with an analogue of Liggett's condition.

\emph{In this paper, we prove that \eqref{eq:openkpz}-\eqref{eq:openkpzII}} (or equivalently, \eqref{eq:openshe}-\eqref{eq:opensheII}), \emph{is a continuum limit for the height function associated to the open \abbr{ASEP} model with general} (possibly speed-changed) \emph{asymmetries for the boundary dynamics}. Let us describe this particle system in words below. (Again, see Section \ref{section:model} for a precise formulation in terms of infinitesimal generators.)
\begin{itemize}
\item The state space of the model is {\small$\{\pm1\}^{\mathbb{K}_{N}}$}, where {\small$\mathbb{K}_{N}:=\{1,\ldots,N\}$} is a discrete interval. We denote configurations by {\small$\eta\in\{\pm1\}^{\mathbb{K}_{N}}$}. The interpretation is that {\small$\eta_{\x}=1$} if and only if there is a particle at {\small$\x$}.
\item The process consists of swapping ``spins" {\small$\eta_{\x}$} at neighboring points in {\small$\mathbb{K}_{N}$} and by flipping spins at the boundary points {\small$\{1,N\}$}. The swapping dynamics are as follows. Fix a pair of neighboring points {\small$\x,\x+1\in\mathbb{K}_{N}$}. Suppose that {\small$\eta_{\x}=1$} and {\small$\eta_{\x+1}=-1$}. We swap these spin values at speed {\small$1/2\cdot(N^{2}-\lambda N^{3/2})$}. If {\small$\eta_{\x}=-1$} and {\small$\eta_{\x+1}=1$}, then we swap these spin values at speed {\small$1/2\cdot(N^{2}+\lambda N^{3/2})$}.
\item We now describe the boundary dynamics. Fix \emph{any} local functions {\small$\alpha,\beta,\gamma,\delta:\{\pm1\}^{\mathbb{K}_{N}}\to\R$} such that:
\begin{itemize}
\item The quantities {\small$\alpha[\eta],\gamma[\eta]$} depend only on {\small$\eta_{\x}$} for {\small$\x\in\{1,\ldots,\mathrm{m}\}$} for some {\small$\mathrm{m}>0$} fixed.
\item The quantities {\small$\beta[\eta],\delta[\eta]$} depend only on {\small$\eta_{\x}$} for {\small$\x\in\{N-\mathrm{m},N\}$} for the same {\small$\mathrm{m}>0$}.
\end{itemize}
Suppose that {\small$\eta_{1}=-1$}. We flip it (i.e. send {\small$\eta_{1}\mapsto1$} as to create a particle at {\small$\{1\}$}) at speed {\small$1/4\cdot N^{2}+\alpha[\eta]\cdot N^{3/2}$}. Suppose that {\small$\eta_{1}=1$}. We flip it at speed {\small$1/4\cdot N^{2}+\gamma[\eta]\cdot N^{3/2}$}. Now, suppose that {\small$\eta_{N}=-1$}. We flip it at speed {\small$1/4\cdot N^{2}+\delta[\eta]\cdot N^{3/2}$}. Suppose that {\small$\eta_{N}=1$}. We flip it at speed {\small$1/4\cdot N^{2}+\beta[\eta]\cdot N^{3/2}$}.
\item All Poisson clocks above are jointly independent.
\item The process described by this dynamic is denoted by {\small$\eta_{\t}=(\eta_{\t,\x})_{\x\in\mathbb{K}_{N}}$}. The height function is defined precisely in Section \ref{section:model}; roughly speaking, it is an appropriately chosen spatial anti-derivative of {\small$\x\mapsto\eta_{\t,\x}$}.
\end{itemize}
The general question of large-scale behaviors in speed-change exclusion processes has received a great deal of attention over the past few decades \cite{FHU,K,LV} (see also \cite{KL} for a broad overview), motivating the model we study in this paper.

Our result builds on the small handful of models for which universality of open \abbr{KPZ} is known. An immediate consequence of our work is the derivation of open \abbr{KPZ} from open \abbr{ASEP} \emph{without} assuming Liggett's condition (which, in the above notation, amount to the technical relations {\small$\alpha=\gamma$} and {\small$\delta=\beta$} for constant {\small$\alpha,\beta,\gamma,\delta$}). This answers a question of Corwin \cite{C22} and Conjecture 1.2 in \cite{H25}. See Theorem \ref{theorem:main} for a precise statement.

At a technical level, we show that the height function associated to the interacting particle system of interest has an exponential which converges to \eqref{eq:openshe}-\eqref{eq:opensheII}. Because we consider general speeds for the asymmetric part of the boundary-reservoirs dynamics, the exponential transform does \emph{not} exactly linearize the associated height function. In particular, the equation for the evolution of the exponentiated height function is a lattice version of \eqref{eq:openshe}-\eqref{eq:opensheII} but with additional error terms. Controlling with these error terms is the main difficulty, because we \emph{do not} have access to simple invariant measures in general.

The technique that we use to address the aforementioned difficulty comes from \cite{Y24}, which allows us to understand the large-{\small$N$} contribution of general local functions even after general perturbations of the dynamics that can destroy any knowledge of invariant measures. We remark, however, that the work \cite{Y24} focuses on models on a torus (thus no boundary interactions) and for perturbations of order {\small$N$} (as opposed to order {\small$N^{3/2}$} like in this paper). Thus, we must adapt these methods carefully in order to apply them here. In any case, the technical basis for the method is in analyzing the dynamics of the particle system through its Kolmogorov equations, and using the resulting \abbr{PDE} estimates to prove a ``Kipnis-Varadhan" inequality (i.e. \abbr{CLT}-type estimate) for additive functionals.

As an immediate consequence of our result and a coupling idea from \cite{CK} is the convergence of the stationary measure for the increment process of the height function to that for the increment process of \eqref{eq:openkpz}-\eqref{eq:openkpzII}. Such a result extends work of \cite{CK,H25}, the latter of which focuses on arbitrary constant parameters {\small$\alpha,\beta,\gamma,\delta$}, to any local functions {\small$\alpha[\eta],\beta[\eta],\gamma[\eta],\delta[\eta]$}. See Corollary \ref{corollary:stat} for a precise statement.
\subsection{The half-line setting}
In \cite{CS,P}, the authors address the half-space {\small$[0,\infty)$} instead of the compact interval {\small$[0,1]$}. Precisely, in these papers, it is shown that a half-space version of the open \abbr{ASEP} converges to the half-space \abbr{KPZ} equation, which is given by \eqref{eq:openkpz}-\eqref{eq:openkpzII} but on {\small$[0,\infty)\times[0,\infty)$} and without the boundary condition involving the parameter {\small$\mathbf{B}$}. In particular, at the level of the Cole-Hopf map, the limiting \abbr{SPDE} in the half-space is given in its differential version as follows (in which {\small$\xi^{\half}$} is a space-time white noise on {\small$[0,\infty)\times[0,\infty)$}):
\begin{align}
\partial_{\t}\mathbf{Z}^{\half}_{\t,\x}&=\tfrac12\partial_{\x}^{2}\mathbf{Z}^{\half}_{\t,\x}-\lambda\mathbf{Z}^{\half}_{\t,\x}\xi^{\half}_{\t,\x},\label{eq:openshehalf}\\
\partial_{\x}\mathbf{Z}^{\half}_{\t,0}&=-\lambda\mathbf{A}\mathbf{Z}^{\half}_{\t,0}.\label{eq:openshehalfII}
\end{align}
Again, to be precise, the above is interpreted in its Duhamel form. So, we first let {\small$\mathbf{H}^{\half}$} solve
\begin{align*}
\partial_{\t}\mathbf{H}^{\half}_{\s,\t,\X,\Y}&=\tfrac12\partial_{\X}^{2}\mathbf{H}^{\half}_{\s,\t,\X,\Y}\quad\text{for } 0\leq\s\leq\t \ \text{and} \ \X,\Y\in[0,\infty),\\
\partial_{\X}\mathbf{H}^{\half}_{\s,\t,0,\Y}&=-\lambda\mathbf{A}\cdot\mathbf{H}^{\half}_{\s,\t,0,\Y},\\
\mathbf{H}^{\half}_{\s,\s,\X,\Y}&=\mathbf{1}_{\X=\Y} \quad\text{for all } \s\in[0,\infty) \ \text{and} \ \X,\Y\in[0,\infty).
\end{align*}
We then define {\small$\mathbf{Z}^{\half}$} to be the unique adapted solution to the following integral equation:
\begin{align}
\mathbf{Z}^{\half}_{\t,\X}={\textstyle\int_{[0,\infty)}}\mathbf{H}^{\half}_{0,\t,\X,\Y}\mathbf{Z}^{\half}_{0,\Y}\d\Y-\lambda{\textstyle\int_{0}^{\t}\int_{[0,\infty)}}\mathbf{H}^{\half}_{\s,\t,\X,\Y}\mathbf{Z}^{\half}_{\s,\Y}\xi^{\half}_{\s,\Y}\d\Y\d\s.\nonumber
\end{align}
However, even in the half-space, \cite{CS,P} requires Liggett's condition at the left boundary. \emph{In this paper, we show that a half-space version of the class of particle systems described above has the half-space \abbr{KPZ} equation as a continuum limit}. Thus, we remove Liggett's condition in \cite{CS,P} in the half-space case as well.
\subsection{Acknowledgements}
The author was supported by the NSF under Grant No. DMS-2203075.  We thank Ivan Corwin for bringing up and discussing the problem, the reference \cite{H25}, and the question of convergence of stationary measures. We also thank Zongrui Yang for pointing out the reference \cite{DEHP}.
%
%
%
\section{The models and results}\label{section:model}
\subsection{The case of the interval}
Recall {\small$\mathbb{K}_{N}:=\{1,\ldots,N\}$}, and that the process of interest has state space {\small$\{\pm1\}^{\mathbb{K}_{N}}$}. We introduce the particle systems of interest through their infinitesimal generator, a linear operator on functions {\small$\mathfrak{f}:\{\pm1\}^{\mathbb{K}_{N}}\to\R$}. This operator is given by
\begin{align}
\mathscr{L}_{N}:=\mathscr{L}_{N,\mathrm{S}}+\mathscr{L}_{N,\mathrm{A}}+\mathscr{L}_{N,\mathrm{left}}+\mathscr{L}_{N,\mathrm{right}},\label{eq:generatorIa}
\end{align}
in which the operators on the \abbr{RHS} of \eqref{eq:generatorIa} are given as follows. First, for any {\small$\x\in\mathbb{K}_{N}$} such that {\small$\x+1\in\mathbb{K}_{N}$} as well, we let {\small$\mathscr{L}_{\x}$} be the infinitesimal generator for a symmetric simple exclusion process on {\small$\{\x,\x+1\}$} of speed {\small$1$}. More precisely, for any {\small$\mathfrak{f}:\{\pm1\}^{\mathbb{K}_{N}}\to\R$}, we set {\small$\mathscr{L}_{\x}\mathfrak{f}[\eta]:=\mathfrak{f}[\eta^{\x,\x+1}]-\mathfrak{f}[\eta]$}, where {\small$\eta^{\x,\y}$} is the configuration below obtained by swapping {\small$\eta_{\x},\eta_{\y}$}:
\begin{align}
\eta^{\x,\y}_{\z}&=\mathbf{1}_{\z\neq\x,\y}\eta_{\z}+\mathbf{1}_{\z=\x}\eta_{\y}+\mathbf{1}_{\z=\y}\eta_{\x}.\nonumber
\end{align}
We now define
\begin{align}
\mathscr{L}^{}_{N,\mathrm{S}}:=\tfrac12N^{2}\sum_{\x,\x+1\in\mathbb{K}_{N}}\mathscr{L}_{\x}\quad\mathrm{and}\quad\mathscr{L}^{}_{N,\mathrm{A}}:=\tfrac{\lambda}{2}N^{\frac32}\sum_{\x,\x+1\in\mathbb{K}_{N}}\left(\mathbf{1}_{\substack{\eta^{}_{\x}=-1\\\eta^{}_{\x+1}=1}}-\mathbf{1}_{\substack{\eta^{}_{\x}=1\\\eta^{}_{\x+1}=-1}}\right)\mathscr{L}_{\x}\label{eq:generatorIb}
\end{align}
Now, as in the introduction, fix local functions {\small$\alpha,\beta,\gamma,\delta:\{\pm1\}^{\mathbb{K}_{N}}\to\R$} such that:
\begin{itemize}
\item The quantities {\small$\alpha[\eta],\gamma[\eta]$} depend only on {\small$\eta_{\x}$} for {\small$\x\in\{1,\ldots,\mathrm{m}\}$} for some {\small$\mathrm{m}>0$} fixed.
\item The quantities {\small$\beta[\eta],\delta[\eta]$} depend only on {\small$\eta_{\x}$} for {\small$\x\in\{N-\mathrm{m},N\}$} for the same {\small$\mathrm{m}>0$}.
\end{itemize}
We then define
\begin{align}
\mathscr{L}^{}_{N,\mathrm{left}}&:=\left(\tfrac14N^{2}+N^{\frac32}\alpha[\eta^{}]\mathbf{1}_{\eta^{}_{1}=-1}+N^{\frac32}\gamma[\eta^{}]\mathbf{1}_{\eta^{}_{1}=1}\right)\mathscr{S}_{1},\label{eq:generatorIc}\\
\mathscr{L}^{}_{N,\mathrm{right}}&:=\left(\tfrac14N^{2}+N^{\frac32}\delta[\eta^{}]\mathbf{1}_{\eta^{}_{N}=-1}+N^{\frac32}\beta[\eta^{}]\mathbf{1}_{\eta^{}_{N}=1}\right)\mathscr{S}_{N},\label{eq:generatorId}
\end{align}
where $\mathscr{S}_{\x}$ is the generator for speed $1$ dynamics that flips the sign of {\small$\eta^{}_{\x}$}. In particular, for any $\mathfrak{f}:\{\pm1\}^{\mathbb{K}_{N}}\to\R$, we define $\mathscr{S}_{\x}\mathfrak{f}[\eta^{}]=\mathfrak{f}[\eta^{\x}]-\mathfrak{f}[\eta^{}]$, where $\eta^{\x}_{\z}=\eta^{}_{\z}$ if $\z\neq\x$ and $\eta^{\x}_{\x}=-\eta^{}_{\x}$.

We will denote by {\small$\eta_{\t}\in\{\pm1\}^{\mathbb{K}_{N}}$} the state (or value) of the configuration of spins at time {\small$\t\geq0$}.

We now define precisely the associated height function to this particle system. We let {\small$\mathbf{h}^{N}_{\t,0}$} be given by {\small$2N^{-1/2}$} times the number of particles that have been removed at site {\small$1$} by time {\small$\t$} minus the numbers of particles that have been created at site {\small$1$} by time {\small$\t$}. (This deletion and creation specifically comes from the reservoir interaction at site {\small$1$}, i.e. coming from the generator \eqref{eq:generatorIc}.) Now, for any {\small$\x\in\{1,\ldots,N\}$}, we define
\begin{align}
\mathbf{h}^{N}_{\t,\x}&=\mathbf{h}^{N}_{\t,0}+N^{-\frac12}\sum_{\y=1,\ldots,\x}\eta_{\t,\y}.\label{eq:hf}
\end{align}
We emphasize that {\small$\mathbf{h}^{N}$} is a function on {\small$[0,\infty)\times\{0,\ldots,N\}$}, not {\small$[0,\infty)\times\mathbb{K}_{N}=[0,\infty)\times\{1,\ldots,N\}$}. 

Now, let us define a microscopic Cole-Hopf (or Gartner) transform. It is an analogue of the exponential map taking \eqref{eq:openkpz}-\eqref{eq:openkpzII} to \eqref{eq:openshe}-\eqref{eq:opensheII}, and it is given by
\begin{align}
\mathbf{Z}^{N}_{\t,\x}&:=\exp\Big\{-\lambda \mathbf{h}^{N}_{\t,\x}+(\tfrac{\lambda^{2}}{2}N+\tfrac{\lambda^{4}}{24})\t\Big\}.\label{eq:ch}
\end{align}
In the case that {\small$\lambda=1$}, we note that the renormalization constant {\small$N/2+1/24$} is exactly the one from \cite{BG} for \abbr{ASEP} on the line. (For general {\small$\lambda\neq0$}, it also agrees with (1.27) in \cite{DT}.)

Now, before we can state the main result, we need to introduce a last piece of notation for a family of probability measures on the state space {\small$\{\pm1\}^{\mathbb{K}_{N}}$}. For any {\small$\sigma\in[-1,1]$}, we let {\small$\mathbb{P}^{\sigma}$} be a product measure on {\small$\{\pm1\}^{\mathbb{K}_{N}}$} such that {\small$\E^{\sigma}\eta_{\x}=\sigma$} for all {\small$\x\in\mathbb{K}_{N}$}, where {\small$\E^{\sigma}$} denotes expectation with respect to {\small$\mathbb{P}^{\sigma}$}. These probability measures are \emph{not} invariant measures for {\small$\t\mapsto\eta_{\t}$} in general. However, as the main result (Theorem \ref{theorem:main}) of this paper indicates, these measures are still used to compute the boundary parameters {\small$\mathbf{A},\mathbf{B}$} in limiting \abbr{SPDE} \eqref{eq:openshe}-\eqref{eq:opensheII}.

We will now state the first main result of this paper as follows.

\begin{theorem}\label{theorem:main}
\fsp Suppose that for any {\small$p\geq1$}, there is a constant {\small$\mathrm{C}_{p}<\infty$} such that for all {\small$\x,\y\in\llbracket0,N\rrbracket$}, we have 
\begin{align}
\E|\mathbf{Z}^{N}_{0,\x}|^{2p}\leq\mathrm{C}_{p} \quad\text{and}\quad \E|\mathbf{Z}^{N}_{0,\x}-\mathbf{Z}^{N}_{0,\y}|^{2p}\leq\mathrm{C}_{p}N^{-p}|\x-\y|^{p}.\label{eq:mainap}
\end{align}
Suppose also that {\small$\X\mapsto\mathbf{Z}^{N}_{0,N\X}$}, after extending it from $\{0,\frac1N,\ldots,1\}$ to $[0,1]$ via linear interpolation, converges to a possibly random {\small$\mathbf{Z}^{}_{0,\cdot}\in\mathscr{C}^{}([0,1])$}; the convergence is uniformly on {\small$[0,1]$} in probability. Then the function
\begin{align*}
(\t,\X)\mapsto\mathbf{Z}^{N}_{\t,N\X},
\end{align*}
if we linearly interpolate it from {\small$\{0,\frac1N,\ldots,1\}$} to {\small$[0,1]$}, converges in law in the Skorokhod space {\small$\mathscr{D}([0,1],\mathscr{C}([0,1]))$} to the solution of \eqref{eq:openshe}-\eqref{eq:opensheII} with initial data {\small$\mathbf{Z}^{}_{0,\cdot}$} and with boundary parameters 
\begin{align}
\mathbf{A}&:=\tfrac32\lambda +2\E^{0}\left\{\alpha[\eta^{}]-\gamma[\eta^{}]\right\}-2\E^{0}\left\{\eta^{}_{1}\left[\alpha[\eta^{}]-\gamma[\eta^{}]\right]\right\}\label{eq:robina}\\
\mathbf{B}&:=-\tfrac32\lambda +2\E^{0}\left\{\delta[\eta^{}]+\beta[\eta^{}]\right\}-2\E^{0}\left\{\eta^{}_{N}\left[\delta[\eta^{}]+\beta[\eta^{}]\right]\right\}.\label{eq:robinb}
\end{align}
\end{theorem}
We now briefly comment on the structure of \eqref{eq:robina}-\eqref{eq:robinb}. The expectations in \eqref{eq:robina} are ultimately necessary in order to capture the leading-order behavior of the function {\small$2(\alpha[\eta_{\t}]-\gamma[\eta_{\t}])-2\eta_{\t,1}(\alpha[\eta_{\t}]+\gamma[\eta_{\t}])$} coming from the dynamics of {\small$\mathbf{Z}^{N}$} arising from the open reservoir dynamics in the particle system. Thus, Theorem \ref{theorem:main} is implicitly stating that this function homogenizes into a \emph{non-invariant} measure expectation; this already highlights technical challenges in proving Theorem \ref{theorem:main}. The expectations in \eqref{eq:robinb} arise similarly from open reservoir dynamics at {\small$N$}.

In the case where {\small$\alpha,\beta,\gamma,\delta$} are constant (i.e. independent of {\small$\eta_{\t}$}, then the function {\small$2(\alpha[\eta_{\t}]-\gamma[\eta_{\t}])-2\eta_{\t,1}(\alpha[\eta_{\t}]+\gamma[\eta_{\t}])$} is nontrivial only because of the second term {\small$\eta_{\t,1}(\alpha[\eta_{\t}]+\gamma[\eta_{\t}])$}. However, if we assume that {\small$\alpha=\gamma$}, then this term vanishes, and the term of interest  {\small$2(\alpha[\eta_{\t}]-\gamma[\eta_{\t}])-2\eta_{\t,1}(\alpha[\eta_{\t}]+\gamma[\eta_{\t}])$} becomes trivial to analyze. The assumption {\small$\alpha=\gamma$} is exactly Liggett's condition in \cite{CS,P} (at the left boundary). Similarly, \cite{CS,P} require {\small$\delta=\beta$}.

Finally, it is an elementary (and fairly quick, thus omitted here) calculation to verify that {\small$\mathbf{A},\mathbf{B}$} agree with the parameters for the limiting \abbr{SPDE}s in \cite{CS,P} in the case of constant {\small$\alpha,\beta,\gamma,\delta$} satisfying {\small$\alpha=\gamma$} and {\small$\delta=\beta$}.
\subsubsection{Convergence of stationary measures}
The process {\small$\t\mapsto(\mathbf{h}^{N}_{\t,\x})_{\x\in\mathbb{K}_{N}}$} does not admit an invariant probability measure. However, the \emph{increment process} {\small$\t\mapsto(\mathbf{h}^{N}_{\t,\x}-\mathbf{h}^{N}_{\t,0})_{\x\in\mathbb{K}_{N}}$} admits a unique invariant probability measure; it is given by the formula \eqref{eq:hf} upon sampling the {\small$\eta_{\t,\cdot}$}-variables therein according to its (unique) invariant measure. We let {\small$\boldsymbol{\pi}^{N,\alpha,\beta,\gamma,\delta}$} denote said invariant measure of the increment process. We also let {\small$\boldsymbol{\pi}^{\mathbf{A},\mathbf{B}}$} denote the invariant measure of the increment process {\small$\t\mapsto(\mathbf{h}_{\t,\X}-\mathbf{h}_{\t,0})_{\X\in[0,1]}$}, where {\small$\mathbf{h}$} solves \eqref{eq:openkpz}-\eqref{eq:openkpzII} with {\small$\lambda=-1$}. (We restrict to {\small$\lambda=-1$} because the construction and uniqueness of {\small$\boldsymbol{\pi}^{\mathbf{A},\mathbf{B}}$} in \cite{CK,KM,P23} is done for this choice of {\small$\lambda$}.)

The next result, which eventually follows from the process-level convergence in Theorem \ref{theorem:main} and a coupling trick used in \cite{CK}, states convergence of stationary measures for the increment process. Before we state it, we need to introduce some notation for a space-rescaling operator which transforms {\small$\mathbb{K}_{N}$} into {\small$\{0,\frac{1}{N},\ldots,1\}$} and extends functions on the latter to functions on {\small$[0,1]$} by linear interpolation. To be precise, for any function {\small$\mathsf{f}:\mathbb{K}_{N}\to\R$}, we set {\small$\Gamma^{N}\mathsf{f}:[0,1]\to\R$} by first defining {\small$\Gamma^{N}\mathsf{f}_{\X}:=\mathsf{f}_{N\X}$} for {\small$\X\in\{0,\frac{1}{N},\ldots,1\}$} and then extending from {\small$\{0,\frac{1}{N},\ldots,1\}$} to {\small$[0,1]$} by linear interpolation.
\begin{corollary}\label{corollary:stat}
\fsp Suppose that {\small$\lambda=-1$}, and let {\small$\Gamma^{N}_{\ast}\boldsymbol{\pi}^{N,\alpha,\beta,\gamma,\delta}$} be the pushforward of {\small$\boldsymbol{\pi}^{N,\alpha,\beta,\gamma,\delta}$} under {\small$\Gamma^{N}$}. We have {\small$\Gamma^{N}_{\ast}\boldsymbol{\pi}^{N,\alpha,\beta,\gamma,\delta}\to\pi^{\mathbf{A},\mathbf{B}}$} weakly as {\small$N\to\infty$} as probability measures on {\small$\mathscr{C}([0,1])$}, where {\small$\mathbf{A},\mathbf{B}$} are from \eqref{eq:robina}-\eqref{eq:robinb}.
\end{corollary}
\subsection{The case of the half-space}
We will now consider a particle system evolving on the lattice {\small$\Z_{>0}$} of positive integers. This process, which we denote by {\small$\t\mapsto\eta^{\mathrm{half}}_{\t}$} in order to avoid confusing notation with the case of the interval, has state space {\small$\{\pm1\}^{\Z_{>0}}$}, and its infinitesimal generator is given by
\begin{align}
\mathscr{L}^{\mathrm{half}}:=\mathscr{L}_{N,\mathrm{S}}^{\mathrm{half}}+\mathscr{L}_{N,\mathrm{A}}^{\mathrm{half}}+\mathscr{L}_{N,\mathrm{left}}.\label{eq:generatorhalf}
\end{align}
The operator {\small$\mathscr{L}_{N,\mathrm{S}}^{\mathrm{half}}$} is given as in \eqref{eq:generatorIb}, except we replace {\small$\eta$} by {\small$\eta^{\mathrm{half}}\in\{\pm1\}^{\Z_{>0}}$} and we replace {\small$\mathbb{K}_{N}$} by {\small$\Z_{>0}$}. The operator  {\small$\mathscr{L}_{N,\mathrm{A}}^{\mathrm{half}}$} is given as in \eqref{eq:generatorIb} with the same replacements. The operator {\small$\mathscr{L}_{N,\mathrm{left}}$} is given by \eqref{eq:generatorIc} (except, again, we use the {\small$\eta^{\mathrm{half}}$} notation instead of {\small$\eta$}).

Now, similar to \eqref{eq:hf}, we let {\small$\mathbf{h}^{N,\mathrm{half}}_{\t,0}$} equal {\small$2N^{-1/2}$} times the number of particles that have been removed at site {\small$1$} by time {\small$\t$} minus the number that have been created at site {\small$1$} by time {\small$\t$}. Next, for any {\small$\x\in\Z_{>0}$}, we define
\begin{align}
\mathbf{h}^{N,\half}_{\t,\x}=\mathbf{h}^{N,\half}_{\t,0}+N^{-\frac12}\sum_{\y=1,\ldots,\x}\eta^{\half}_{\t,\y},\label{eq:hfhalf}
\end{align}
so that {\small$\mathbf{h}^{N,\half}$} is a function on {\small$[0,\infty)\times\Z_{\geq0}$}, not {\small$[0,\infty)\times\Z_{>0}$}. Similar to \eqref{eq:ch}, we also define
\begin{align}
\mathbf{Z}^{N,\half}_{\t,\x}:=\exp\Big\{-\lambda \mathbf{h}^{N,\half}_{\t,\x}+(\tfrac{\lambda^{2}}{2}N+\tfrac{\lambda^{4}}{24})\t\Big\}.\label{eq:chhalf}
\end{align}
%
\begin{theorem}\label{theorem:mainhalf}
\fsp Suppose that for any {\small$p\geq1$}, there exists {\small$\mathrm{C}_{p},\kappa_{p}<\infty$} such that for all {\small$\x,\y\in\Z_{\geq0}$}, we have 
\begin{align}
\E|\mathbf{Z}^{N,\half}_{0,\x}|^{2p}\leq\mathrm{C}_{p}\exp(\tfrac{\kappa_{p}|\x|}{N})\quad\text{and}\quad\E|\mathbf{Z}^{N,\half}_{0,\x}-\mathbf{Z}^{N,\half}_{0,\y}|^{2p}\leq\mathrm{C}_{p}\exp(\tfrac{\kappa_{p}(|\x|+|\y|)}{N})\cdot N^{-p}|\x-\y|^{p}.\nonumber
\end{align}
Suppose also that {\small$\X\mapsto\mathbf{Z}^{N,\half}_{0,N\X}$}, after extending it from {\small$N^{-1}\cdot\Z_{\geq0}$} to {\small$[0,\infty)$} via linear interpolation, converges to a possibly random {\small$\mathbf{Z}^{}_{0,\cdot}\in\mathscr{C}^{}([0,1])$} (uniformly in probability). Then the function
\begin{align*}
(\t,\X)\mapsto\mathbf{Z}^{N,\half}_{\t,N\X},
\end{align*}
if we linearly interpolate it from {\small$N^{-1}\cdot\Z_{\geq0}$} to {\small$[0,\infty)$}, converges in law in the Skorokhod space {\small$\mathscr{D}([0,1],\mathscr{C}([0,\infty)))$} to the solution of \eqref{eq:openshehalf}-\eqref{eq:openshehalfII} with initial data {\small$\mathbf{Z}^{\half}_{0,\cdot}$} and with boundary parameter {\small$\mathbf{A}$} given by \eqref{eq:robina}.
\end{theorem}
\subsection{Outline for the rest of the paper}
The majority of the paper focuses on the case of the interval, since this situation is ``harder from the perspective of homogenization of local functions" given the competing boundary dynamics. So, in Section \ref{section:mshe}, we derive the evolution equation for the Cole-Hopf map {\small$\mathbf{Z}^{N}$}. In Section \ref{section:bgstrat}, we establish moment estimates to show tightness of {\small$\mathbf{Z}^{N}$}, and we use a martingale problem to reduce the proof of Theorem \ref{theorem:main} to a key stochastic estimate (Proposition \ref{prop:stoch}). We prepare important ingredients to prove Proposition \ref{prop:stoch} in Section \ref{section:stochestimates}. The proof of Proposition \ref{prop:stoch} is the focus of Section \ref{section:stochproof}. Next, in Section \ref{section:corollary}, we prove Corollary \ref{corollary:stat}. Finally, in Section \ref{section:half}, we go through the exact same procedure (with minor adjustments) to prove the half-space result (Theorem \ref{theorem:mainhalf}).
\subsection{Notation}
\begin{enumerate}
\item We use big-Oh notation, so that {\small$\a=\mathrm{O}(\b)$} means {\small$|\a|\leq\mathrm{C}|\b|$} for some constant $\mathrm{C}>0$. The parameters that $\mathrm{C}$ depends on will be written as subscripts in the $\mathrm{O}$-notation. We write {\small$\a\lesssim\b$} to mean {\small$\a=\mathrm{O}(\b)$}, and we write {\small$\a=\mathrm{o}(\b)$} if $|\a|/|\b|$ vanishes in the large-$N$ limit.
\item For any {\small$\a,\b\in\R$}, we set {\small$\llbracket\a,\b\rrbracket:=[\a,\b]\cap\Z$}.
\end{enumerate}
%
%
%
\section{Evolution equation for {\small$\mathbf{Z}^{N}$}}\label{section:mshe}
In order to state the dynamics of {\small$\mathbf{Z}^{N}$}, we must introduce notation for a discrete approximation to the Robin Laplacian in \eqref{eq:openshe}-\eqref{eq:opensheII} and its associated heat kernel. 
\begin{definition}\label{definition:robinheat}
\fsp Consider any function {\small$\varphi:\llbracket0,N\rrbracket\to\R$}. We define the discrete Robin Laplacian {\small$\Delta_{\mathbf{A},\mathbf{B}}$} via
\begin{align}
\Delta_{\mathbf{A},\mathbf{B}}\varphi_{\x}:=\begin{cases}\varphi_{\x+1}+\varphi_{\x-1}-2\varphi_{\x}&\x\in\{1,\ldots,N-1\}\\\varphi_{1}-\varphi_{0}+\tfrac{\lambda\mathbf{A}}{N}\varphi_{0}&\x=0\\\varphi_{N-1}-\varphi_{N}-\tfrac{\lambda\mathbf{B}}{N}\varphi_{N}&\x=N\end{cases}\label{eq:robinheatI}
\end{align}
(In words, extend {\small$\varphi$} to {\small$\{-1,\ldots,N+1\}$} via {\small$\varphi_{0}-\varphi_{-1}=-\lambda N^{-1}\mathbf{A}\varphi_{0}$} and {\small$\varphi_{N+1}-\varphi_{N}=-\lambda N^{-1}\mathbf{B}\varphi_{N}$}, and take the ``usual" discrete Laplacian.) Now, let {\small$\mathbf{H}^{N}_{\s,\t,\x,\y}$} be a function of {\small$(\s,\t,\x,\y)\in[0,\infty)^{2}\times\mathbb{K}_{N}^{2}$} with {\small$\s\leq\t$} such that 
\begin{align}
\partial_{\t}\mathbf{H}^{N}_{\s,\t,\x,\y}&=\tfrac12N^{2}\Delta_{\mathbf{A},\mathbf{B}}\mathbf{H}^{N}_{\s,\t,\x,\y}\quad\text{and}\quad\mathbf{H}^{N}_{\s,\s,\x,\y}=\mathbf{1}_{\x=\y}.\label{eq:robinheatII}
\end{align}
We clarify that the Robin Laplacian acts on {\small$\x$}, and the differential equation is for {\small$\s<\t$}.
\end{definition}
We will break up the evolution equation for {\small$\mathbf{Z}^{N}$} into three parts -- the bulk of {\small$\mathbb{K}_{N}$}, the left boundary point {\small$\{0\}$}, and the right boundary point {\small$\{N\}$}. We start with the bulk.
\begin{lemma}\label{lemma:mshebulk}
\fsp We claim that the following equation holds:
\begin{align}
\d\mathbf{Z}^{N}_{\t,\x}&=\tfrac12N^{2}\Delta_{\mathbf{A},\mathbf{B}}\mathbf{Z}^{N}_{\t,\x}\d\t+\mathbf{Z}^{N}_{\t,\x}\d\mathscr{Q}_{\t,\x},\quad(\t,\x)\in[0,\infty)\times\{1,\ldots,N-1\}.\label{eq:mshebulkI}
\end{align}
Above, {\small$\mathscr{Q}$} is the following compensated Poisson process, in which we use notation explained afterwards:
\begin{align}
\d\mathscr{Q}_{\t,\x}&=\Big\{\exp(2\lambda N^{-\frac12})-1\Big\}\mathbf{1}_{\eta_{\t,\x}=1}\mathbf{1}_{\eta_{\t,\x+1}=-1}\Big(\d\mathscr{Q}^{\to}_{\t,\x}-(\tfrac12N^{2}-\tfrac{\lambda}{2} N^{\frac32})\d\t\Big)\label{eq:mshebulkIIa}\\
&+\Big\{\exp(-2\lambda N^{-\frac12})-1\Big\}\mathbf{1}_{\eta_{\t,\x}=-1}\mathbf{1}_{\eta_{\t,\x+1}=1}\Big(\d\mathscr{Q}^{\leftarrow}_{\t,\x}-(\tfrac12N^{2}+\tfrac{\lambda}{2}N^{\frac32})\d\t\Big).\label{eq:mshebulkIIb}
\end{align}
The {\small$\t\mapsto\mathscr{Q}^{\to}_{\t,\x}$} and {\small$\t\mapsto\mathscr{Q}^{\leftarrow}_{\t,\x}$} are independent Poisson processes of speed {\small$(N^{2}\mp \lambda N^{3/2})/2$}, respectively. Poisson processes at different {\small$\x$}-points are all independent of each other as well.
\end{lemma}
\begin{remark}\label{remark:sdeintegrate}
\fsp Strictly speaking, the equation \eqref{eq:mshebulkI} is short-hand for the following integrated version:
\begin{align*}
\mathbf{Z}^{N}_{\t,\x}&=\mathbf{Z}^{N}_{0,\x}+\int_{0}^{\t}\tfrac12N^{2}\Delta_{\mathbf{A},\mathbf{B}}\mathbf{Z}^{N}_{\s,\x}\d\s+\int_{0}^{\t}\mathbf{Z}^{N}_{\s,\x}\d\mathscr{Q}_{\s,\x},
\end{align*}
where the last term is a sum over jump times of {\small$\mathscr{Q}$}. Since the above display is more complicated to write and will not be directly useful, we stick with the differential short-hand as in \eqref{eq:mshebulkI}. We will use a similar differential short-hand for similar equations appearing below as well.
\end{remark}
\begin{proof}
Since we look at {\small$\x\neq0,N$}, the equation for the dynamics of {\small$\mathbf{Z}^{N}$} is the same as in the full line \abbr{ASEP}. Thus, we can cite Proposition 2.2 in \cite{DT} (with {\small$m=1$} and {\small$r_{1}=1$} therein). 
\end{proof}
We now present the evolution equation for {\small$\mathbf{Z}^{N}$} at the left boundary point {\small$\{0\}$}.
\begin{lemma}\label{lemma:msheleft}
\fsp With notation explained afterwards, we have 
\begin{align}
\d\mathbf{Z}^{N}_{\t,0}&=\tfrac12N^{2}\Delta_{\mathbf{A},\mathbf{B}}\mathbf{Z}^{N}_{\t,0}\d\t+\mathbf{Z}^{N}_{\t,0}\d\mathscr{Q}_{\t,0}+\lambda N\mathfrak{f}_{\mathrm{left}}[\eta_{\t}]\mathbf{Z}^{N}_{\t,0}\d\t+N^{\frac12}\mathfrak{b}_{\mathrm{left}}[\eta_{\t}]\mathbf{Z}^{N}_{\t,0}\d\t.\label{eq:msheleftI}
\end{align}
%
\begin{itemize}
\item The function {\small$\mathfrak{f}_{\mathrm{left}}:\{\pm1\}^{\mathbb{K}_{N}}\to\R$} is given by {\small$\mathfrak{f}_{\mathrm{left}}[\eta]:=3\lambda/4+\alpha[\eta]-\gamma[\eta]-\eta_{1}(\alpha[\eta]+\gamma[\eta])-\mathbf{A}/2$}.
\item The function {\small$\mathfrak{b}_{\mathrm{left}}:\{\pm1\}^{\mathbb{K}_{N}}\to\R$} is bounded uniformly in {\small$N,\eta$}.
\item The process {\small$\t\mapsto\mathscr{Q}_{\t,0}$} is the following compensated Poisson process:
\begin{align}
\d\mathscr{Q}_{\t,0}&=\Big\{\exp(2\lambda N^{-\frac12})-1\Big\}\mathbf{1}_{\eta_{\t,1}=-1}\Big\{\d\mathscr{Q}^{+}_{\t,0}-\Big(\tfrac14N^{2}+N^{\frac32}\alpha[\eta_{\t}]\Big)\d\t\Big\}\label{eq:msheleftIIa}\\
&+\Big\{\exp(-2\lambda N^{-\frac12})-1\Big\}\mathbf{1}_{\eta_{\t,1}=1}\Big\{\d\mathscr{Q}^{-}_{\t,0}-\Big(\tfrac14N^{2}+N^{\frac32}\gamma[\eta_{\t}]\Big)\d\t\Big\}.\label{eq:msheleftIIb}
\end{align}
Above, {\small$\t\mapsto\mathscr{Q}^{\pm}_{\t,0}$} are independent Poisson processes of speed {\small$1/4\cdot N^{2}+N^{3/2}\alpha[\eta_{\t}]$} for the {\small$+$}-superscript and {\small$1/4\cdot N^{2}+N^{3/2}\gamma[\eta_{\t}]$} for the {\small$-$}-superscript. They are independent of the Poisson processes in Lemma \ref{lemma:mshebulk}.
\end{itemize}
\end{lemma}
Let us briefly clarify the structure of \eqref{eq:msheleftI} before we show Lemma \ref{lemma:msheleft}. The first two terms on the \abbr{RHS} of \eqref{eq:msheleftI} are a lattice version of open \abbr{SHE}. The order {\small$N$} term is not small deterministically. We eventually integrate \eqref{eq:msheleftI} in space against smooth test functions, and because integration in space involves an additional factor of {\small$N^{-1}$}, the third term on the \abbr{RHS} of \eqref{eq:msheleftI} is in general order {\small$1$}. However, \emph{after averaging in space-time}, the {\small$\mathfrak{f}_{\mathrm{left}}[\eta_{\t}]$} function exhibits rapid fluctuations and substantial cancellation. Indeed, given \eqref{eq:robina}, {\small$\mathfrak{f}_{\mathrm{left}}$} is mean-zero with respect to the product measure {\small$\mathbb{P}^{0}$} (hence ``fluctuating"). \emph{As we alluded to in the introduction, though, the measure {\small$\mathbb{P}^{0}$} is not an invariant measure for {\small$\t\mapsto\eta_{\t}$}}, so proving this cancellation property for {\small$\mathfrak{f}_{\mathrm{left}}$} makes up the technical majority of this paper. (The fact that being mean-zero with respect to {\small$\mathbb{P}^{0}$} is enough on its own to imply fluctuation-cancellation despite it not being an invariant measure is also perhaps surprising.) Next, we note that the bounded function {\small$\mathfrak{b}_{\mathrm{left}}$} will be controlled by more deterministic means, as indicated by the fact that we have not asserted any probabilistic information about this function in Lemma \ref{lemma:msheleft}. Indeed, by integrating in space-time, we obtain (as before) a factor of {\small$N^{-1}$} which beats the {\small$N^{1/2}$}-scaling for the last term in \eqref{eq:msheleftI}.
\begin{proof}
Before we start, in this proof, we will write {\small$\a\sim\b$} to mean {\small$\a-\b=\mathrm{O}(N^{\frac12})\mathbf{Z}^{N}_{\t,0}$}. Recall the formulas \eqref{eq:hf} and \eqref{eq:ch}. We have 
\begin{align}
\tfrac12N^{2}\Delta_{\mathbf{A},\mathbf{B}}\mathbf{Z}^{N}_{\t,0}&=\tfrac12N^{2}\Big\{\exp(-\lambda N^{-\frac12}\eta_{\t,1})-1\Big\}\mathbf{Z}^{N}_{\t,0}+\tfrac{\lambda}{2}N\mathbf{A}\mathbf{Z}^{N}_{\t,0}\nonumber\\
&\sim-\tfrac12\lambda N^{\frac32}\eta_{\t,1}\mathbf{Z}^{N}_{\t,0}+\tfrac14\lambda^{2}N\mathbf{Z}^{N}_{\t,0}+\tfrac{\lambda}{2}N\mathbf{A}\mathbf{Z}^{N}_{\t,0},\label{eq:msheleftI1}
\end{align}
where the second line follows from Taylor expanding the exponential and using {\small$\eta_{\t,1}\in\{\pm1\}^{}$} (so that {\small$\eta_{\t,1}^{2}=1$}). We now compute {\small$\d\mathbf{Z}^{N}_{\t,0}$} as follows. By \eqref{eq:hf} and \eqref{eq:ch}, the only way for {\small$\mathbf{Z}^{N}_{\t,0}$} to change value is if {\small$\eta_{\t,1}$} flips. If {\small$\eta_{\t,1}$} flips from {\small$-1$} to {\small$+1$}, then {\small$\mathbf{h}^{N}_{\t,0}$} goes up by {\small$2\lambda N^{-1/2}$}. If it flips from {\small$+1$} to {\small$-1$}, then {\small$\mathbf{h}^{N}_{\t,0}$} goes down by {\small$2\lambda N^{-1/2}$}. The former happens when the {\small$\mathscr{Q}^{+}$} clock rings and {\small$\eta_{\t,1}=-1$}. The latter happens when {\small$\mathscr{Q}^{-}$} rings and {\small$\eta_{\t,1}=1$}. Using this and {\small$\mathbf{1}_{\eta_{\t,1}=\pm1}=(1\pm\eta_{\t,1})/2$}, we have 
\begin{align}
\d\mathbf{Z}^{N}_{\t,0}&=\Big\{\exp(2\lambda N^{-\frac12})-1\Big\}\mathbf{Z}^{N}_{\t,0}\tfrac{1-\eta_{\t,1}}{2}\d\mathscr{Q}^{+}_{\t,0}+\Big\{\exp(-2\lambda N^{-\frac12})-1\Big\}\mathbf{Z}^{N}_{\t,0}\tfrac{1+\eta_{\t,1}}{2}\d\mathscr{Q}^{-}_{\t,0}+(\tfrac{\lambda^{2}}{2}N+\tfrac{\lambda^{4}}{24})\mathbf{Z}^{N}_{\t,0}\d\t\nonumber\\
&=(\tfrac14N^{2}+N^{\frac32}\alpha[\eta_{\t}])\tfrac{1-\eta_{\t,1}}{2}\Big\{\exp(2\lambda N^{-\frac12})-1\Big\}\mathbf{Z}^{N}_{\t,0}\d\t\nonumber\\
&+(\tfrac14N^{2}+N^{\frac32}\gamma[\eta_{\t}])\tfrac{1+\eta_{\t,1}}{2}\Big\{\exp(-2\lambda N^{-\frac12})-1\Big\}\mathbf{Z}^{N}_{\t,0}\d\t\nonumber\\
&+\mathbf{Z}^{N}_{\t,0}\d\mathscr{Q}_{\t,0}+(\tfrac{\lambda^{2}}{2}N+\tfrac{\lambda^{4}}{24})\mathbf{Z}^{N}_{\t,0}\d\t,\label{eq:msheleftI2}
\end{align}
where the last identity follows from the definition \eqref{eq:msheleftIIa}-\eqref{eq:msheleftIIb}. We now take the first and second terms on the far \abbr{RHS} of \eqref{eq:msheleftI2} and Taylor expand the exponentials therein. This gives
\begin{align}
&(\tfrac14N^{2}+N^{\frac32}\alpha[\eta_{\t}])\tfrac{1-\eta_{\t,1}}{2}\Big\{\exp(2\lambda N^{-\frac12})-1\Big\}\mathbf{Z}^{N}_{\t,0}\d\t+(\tfrac14N^{2}+N^{\frac32}\gamma[\eta_{\t}])\tfrac{1+\eta_{\t,1}}{2}\Big\{\exp(-2\lambda N^{-\frac12})-1\Big\}\mathbf{Z}^{N}_{\t,0}\d\t\nonumber\\
&\sim(\tfrac14N^{2}+N^{\frac32}\alpha[\eta_{\t}])\tfrac{1-\eta_{\t,1}}{2}(2\lambda N^{-\frac12}+2\lambda^{2}N^{-1})\mathbf{Z}^{N}_{\t,0}\d\t+(\tfrac14N^{2}+N^{\frac32}\gamma[\eta_{\t}])\tfrac{1+\eta_{\t,1}}{2}(-2\lambda N^{-\frac12}+2\lambda^{2}N^{-1})\mathbf{Z}^{N}_{\t,0}\d\t\nonumber\\
&\sim-\tfrac12\lambda N^{\frac32}\eta_{\t,1}\mathbf{Z}^{N}_{\t,0}\d\t+N\Big(\tfrac12\lambda^{2}+\lambda \alpha[\eta_{\t}]-\lambda \gamma[\eta_{\t}]-\lambda \eta_{\t,1}(\alpha[\eta_{\t}]+\gamma[\eta_{\t}])\Big)\mathbf{Z}^{N}_{\t,0}\d\t.\nonumber
\end{align}
If we now plug the previous display into \eqref{eq:msheleftI2} and note that the {\small$\lambda^{4}/24$} coefficient in \eqref{eq:msheleftI2} can be dropped if we change {\small$=$} to {\small$\sim$}, we deduce
\begin{align}
\d\mathbf{Z}^{N}_{\t,0}&=-\tfrac12\lambda N^{\frac32}\eta_{\t,1}\mathbf{Z}^{N}_{\t,0}\d\t+N\Big(\lambda^{2}+\lambda \alpha[\eta_{\t}]-\lambda \gamma[\eta_{\t}]-\lambda \eta_{\t,1}(\alpha[\eta_{\t}]+\lambda \gamma[\eta_{\t}])\Big)\mathbf{Z}^{N}_{\t,0}\d\t+\mathbf{Z}^{N}_{\t,0}\d\mathscr{Q}_{\t,0}.\nonumber
\end{align}
Combining the previous display with \eqref{eq:msheleftI1} shows that {\small$\d\mathbf{Z}^{N}_{\t,0}\sim\frac12N^{2}\Delta_{\mathbf{A},\mathbf{B}}\mathbf{Z}^{N}_{\t,0}\d\t+\mathbf{Z}^{N}_{\t,0}\d\mathscr{Q}_{\t,0}+\lambda N\mathfrak{f}_{\mathrm{left}}[\eta_{\t}]\mathbf{Z}^{N}_{\t,0}\d\t$}. Because {\small$\sim$} means equality up to an error term of the form {\small$\mathrm{O}(N^{-1/2})\mathbf{Z}^{N}_{\t,0}$}, the claim \eqref{eq:msheleftI} follows.
\end{proof}
We now give an analogue of Lemma \ref{lemma:msheleft} but for the right boundary {\small$\{N\}$}. Every clarifying remark we gave for Lemma \ref{lemma:msheleft} applies to Lemma \ref{lemma:msheright} below (after a mirror-image reflection on {\small$\{0,\ldots,N\}$} that swaps {\small$0\leftrightarrow N$}).
\begin{lemma}\label{lemma:msheright}
\fsp With notation explained afterwards, we have 
\begin{align}
\d\mathbf{Z}^{N}_{\t,N}&=\tfrac12N^{2}\Delta_{\mathbf{A},\mathbf{B}}\mathbf{Z}^{N}_{\t,N}\d\t+\mathbf{Z}^{N}_{\t,N}\d\mathscr{Q}_{\t,N}+\lambda N\mathfrak{f}_{\mathrm{right}}[\eta_{\t}]\mathbf{Z}^{N}_{\t,N}\d\t+N^{\frac12}\mathfrak{b}_{\mathrm{right}}[\eta_{\t}]\mathbf{Z}^{N}_{\t,N}\d\t.\label{eq:msherightI}
\end{align}
%
\begin{itemize}
\item The function {\small$\mathfrak{f}_{\mathrm{right}}:\{\pm1\}^{\mathbb{K}_{N}}\to\R$} is given by {\small$\mathfrak{f}_{\mathrm{left}}[\eta]:=3\lambda /4+\delta[\eta]-\beta[\eta]-\eta_{N}(\delta[\eta]+\beta[\eta])-\mathbf{B}/2$}.
\item The function {\small$\mathfrak{b}_{\mathrm{right}}:\{\pm1\}^{\mathbb{K}_{N}}\to\R$} is bounded uniformly in {\small$N,\eta$}.
\item The process {\small$\t\mapsto\mathscr{Q}_{\t,N}$} is the following compensated Poisson process:
\begin{align}
\d\mathscr{Q}_{\t,N}&=\Big\{\exp(2\lambda N^{-\frac12})-1\Big\}\mathbf{1}_{\eta_{\t,N}=-1}\Big\{\d\mathscr{Q}^{+}_{\t,N}-\Big(\tfrac14N^{2}+N^{\frac32}\delta[\eta_{\t}]\Big)\d\t\Big\}\label{eq:msherightIIa}\\
&+\Big\{\exp(-2\lambda N^{-\frac12})-1\Big\}\mathbf{1}_{\eta_{\t,N}=1}\Big\{\d\mathscr{Q}^{-}_{\t,N}-\Big(\tfrac14N^{2}+N^{\frac32}\beta[\eta_{\t}]\Big)\d\t\Big\}.\label{eq:msherightIIb}
\end{align}
Above, {\small$\t\mapsto\mathscr{Q}^{\pm}_{\t,N}$} are independent Poisson processes of speed {\small$1/4\cdot N^{2}+N^{3/2}\delta[\eta_{\t}]$} for the {\small$+$}-superscript and {\small$1/4\cdot N^{2}+N^{3/2}\beta[\eta_{\t}]$} for the {\small$-$}-superscript. They are independent of Poisson processes in Lemmas \ref{lemma:mshebulk} and \ref{lemma:msheleft}.
\end{itemize}
\end{lemma}
\begin{proof}
The argument is identical to the proof of Lemma \ref{lemma:msheleft}.
\end{proof}
Let us now collect the evolution equations for {\small$\mathbf{Z}^{N}$} in the bulk and at the boundaries into one equation. We will also reorganize the resulting differential equation into its Duhamel (space-time integrated) form.
\begin{corollary}\label{corollary:mshe}
\fsp Retain the notation in Lemmas \ref{lemma:mshebulk}, \ref{lemma:msheleft}, and \ref{lemma:msheright}. For any {\small$\t\in[0,\infty)$} and {\small$\x\in\{0,\ldots,N\}$}, we have 
\begin{align}
\d\mathbf{Z}^{N}_{\t,\x}=\tfrac12N^{2}\Delta_{\mathbf{A},\mathbf{B}}\mathbf{Z}^{N}_{\t,\x}\d\t+\mathbf{Z}^{N}_{\t,\x}\d\mathscr{Q}_{\t,\x}&+\lambda \mathbf{1}_{\x=0}N\mathfrak{f}_{\mathrm{left}}[\eta_{\t}]\mathbf{Z}^{N}_{\t,\x}\d\t+\mathbf{1}_{\x=0}N^{\frac12}\mathfrak{b}_{\mathrm{left}}[\eta_{\t}]\mathbf{Z}^{N}_{\t,\x}\d\t\label{eq:msheIa}\\
&+\lambda \mathbf{1}_{\x=N}N\mathfrak{f}_{\mathrm{right}}[\eta_{\t}]\mathbf{Z}^{N}_{\t,\x}\d\t+\mathbf{1}_{\x=N}N^{\frac12}\mathfrak{b}_{\mathrm{right}}[\eta_{\t}]\mathbf{Z}^{N}_{\t,\x}\d\t.\label{eq:msheIb}
\end{align}
By the Duhamel formula, we therefore deduce the following integral equation:
\begin{align}
\mathbf{Z}^{N}_{\t,\x}&=\sum_{\y\in\llbracket0,N\rrbracket}\mathbf{H}^{N}_{0,\t,\x,\y}\mathbf{Z}^{N}_{0,\y}+\int_{0}^{\t}\sum_{\y\in\llbracket0,N\rrbracket}\mathbf{H}^{N}_{\s,\t,\x,\y}\mathbf{Z}^{N}_{\s,\y}\d\mathscr{Q}_{\s,\y}\label{eq:msheIIa}\\
&+\lambda \int_{0}^{\t}\mathbf{H}^{N}_{\s,\t,\x,0}\cdot N\mathfrak{f}_{\mathrm{left}}[\eta_{\s}]\mathbf{Z}^{N}_{\s,0}\d\s+\lambda \int_{0}^{\t}\mathbf{H}^{N}_{\s,\t,\x,N}\cdot N\mathfrak{f}_{\mathrm{right}}[\eta_{\s}]\mathbf{Z}^{N}_{\s,N}\d\s\label{eq:msheIIb}\\
&+\int_{0}^{\t}\mathbf{H}^{N}_{\s,\t,\x,0}\cdot N^{\frac12}\mathfrak{b}_{\mathrm{left}}[\eta_{\s}]\mathbf{Z}^{N}_{\s,0}\d\s+\int_{0}^{\t}\mathbf{H}^{N}_{\s,\t,\x,N}\cdot N^{\frac12}\mathfrak{b}_{\mathrm{right}}[\eta_{\s}]\mathbf{Z}^{N}_{\s,N}\d\s.\label{eq:msheIIc}
\end{align}
\end{corollary}
%
%
%
\section{Tightness and identification of limit points via martingale problem}\label{section:bgstrat}
\subsection{Tightness}
We will use the Duhamel equation \eqref{eq:msheIIa}-\eqref{eq:msheIIc} to establish estimates for the size and regularity of {\small$\mathbf{Z}^{N}$} in terms of those of its initial data. The strategy is the same as \cite{CS,P}. The additional terms \eqref{eq:msheIIb}-\eqref{eq:msheIIc}, which are not in \cite{CS,P}, do not introduce issues on this matter (their deterministic sizes, in some sense, are {\small$\mathrm{O}(1)$} as discussed after Lemma \ref{lemma:msheleft}).
\begin{lemma}\label{lemma:moments}
\fsp Fix any {\small$p\geq1$}. We have the following for all {\small$\s,\t\in[0,1]$} with {\small$\s\neq\t$} and all {\small$\x,\y\in\llbracket0,N\rrbracket$} with {\small$\x\neq\y$}:
\begin{align}
\E|\mathbf{Z}^{N}_{\t,\x}|^{2p}&\lesssim_{p}1,\label{eq:momentsI}\\
\E|\mathbf{Z}^{N}_{\t,\x}-\mathbf{Z}^{N}_{\t,\y}|^{2p}&\lesssim_{p}N^{-p}|\x-\y|^{p},\label{eq:momentsII}\\
\E|\mathbf{Z}^{N}_{\t,\x}-\mathbf{Z}^{N}_{\s,\x}|^{2p}&\lesssim_{p}|\t-\s|^{\frac12p}.\label{eq:momentsIII}
\end{align}
\end{lemma}
\begin{proof}
First, it will be convenient to use the following notation to collect \eqref{eq:msheIIb}-\eqref{eq:msheIIc} into one object:
\begin{align}
\mathfrak{w}[\eta]&:=\underbrace{\mathbf{1}_{\x=0}(N\mathfrak{f}_{\mathrm{left}}[\eta]+N^{\frac12}\mathfrak{b}_{\mathrm{left}}[\eta])}_{N\cdot\mathfrak{w}_{\mathrm{left}}[\eta]}+\underbrace{\mathbf{1}_{\x=N}(N\mathfrak{f}_{\mathrm{right}}[\eta]+N^{\frac12}\mathfrak{b}_{\mathrm{right}}[\eta])}_{N\cdot\mathfrak{w}_{\mathrm{right}}[\eta]}.\label{eq:momentsI0}
\end{align}
For any {\small$p\geq1$}, we use the notation {\small$\|\cdot\|_{p}:=(\E|\cdot|^{p})^{1/p}$}. We claim that the following holds:
\begin{align}
\|\int_{0}^{\t}\mathbf{H}^{N}_{\s,\t,\x,0}N\cdot\mathfrak{w}_{\mathrm{left}}[\eta_{\s}]\mathbf{Z}^{N}_{\s,0}\d\s\|_{2p}^{2}&\lesssim(\int_{0}^{\t}\mathbf{H}^{N}_{\s,\t,\x,0}N\cdot\|\mathfrak{w}_{\mathrm{left}}[\eta_{\s}]\mathbf{Z}^{N}_{\s,0}\|_{2p}\d\s)^{2}\nonumber\\
&\lesssim(\int_{0}^{\t}|\t-\s|^{-\frac12}\|\mathbf{Z}^{N}_{\s,0}\|_{2p}\d\s)^{2}\lesssim\int_{0}^{\t}|\t-\s|^{-\frac12}\|\mathbf{Z}^{N}_{\s,0}\|_{2p}^{2}\d\s.\label{eq:moments0}
\end{align}
The first bound follows by the triangle inequality. The second line follows by {\small$|\mathfrak{m}_{\mathrm{left}}|\lesssim1$} and by {Proposition \ref{prop:hke}}, and then by Cauchy-Schwarz for the time-integral. By the same token, we also have the following estimate:
\begin{align}
\|\int_{0}^{\t}\mathbf{H}^{N}_{\s,\t,\x,N}N\cdot\mathfrak{w}_{\mathrm{right}}[\eta_{\s}]\mathbf{Z}^{N}_{\s,N}\d\s\|_{2p}^{2}\lesssim\int_{0}^{\t}|\t-\s|^{-\frac12}\|\mathbf{Z}^{N}_{\s,N}\|_{2p}^{2}\d\s.\label{eq:moments0b}
\end{align}
We now estimate space-regularity of the objects in \eqref{eq:moments0}-\eqref{eq:moments0b}. We claim that
\begin{align}
&\|\int_{0}^{\t}(\mathbf{H}^{N}_{\s,\t,\x,0}-\mathbf{H}^{N}_{\s,\t,\y,0})\cdot N\cdot\mathfrak{m}_{\mathrm{left}}[\eta_{\s}]\mathbf{Z}^{N}_{\s,0}\d\s\|_{2p}^{2}\lesssim(\int_{0}^{\t}|\mathbf{H}^{N}_{\s,\t,\x,0}-\mathbf{H}^{N}_{\s,\t,\y,0}|\cdot N\cdot\|\mathbf{Z}^{N}_{\s,0}\|_{2p}\d\s)^{2}\nonumber\\
&\lesssim(\int_{0}^{\t}|\t-\s|^{-\frac34}N^{-\frac12}|\x-\y|^{\frac12}\|\mathbf{Z}^{N}_{\s,0}\|_{2p}\d\s)^{2}\lesssim\int_{0}^{\t}|\t-\s|^{-\frac34}\|\mathbf{Z}^{N}_{\s,0}\|_{2p}^{2}\d\s\cdot N^{-1}|\x-\y|.\label{eq:moments0c}
\end{align}
The reasoning is the same as what gave us \eqref{eq:moments0}, except we use gradient estimates for {\small$\mathbf{H}^{N}$} in {Proposition \ref{prop:hke}}. By the same token, we also have the following estimate:
\begin{align}
\|\int_{0}^{\t}(\mathbf{H}^{N}_{\s,\t,\x,N}-\mathbf{H}^{N}_{\s,\t,\y,N})\cdot N\cdot\mathfrak{m}_{\mathrm{right}}[\eta_{\s}]\mathbf{Z}^{N}_{\s,N}\d\s\|_{2p}\lesssim\int_{0}^{\t}|\t-\s|^{-\frac34}\|\mathbf{Z}^{N}_{\s,N}\|_{2p}^{2}\d\s\cdot N^{-1}|\x-\y|^{}.\label{eq:moments0d}
\end{align}
Next, we establish time-regularity. To this end, fix {\small$\t_{1},\t_{2}\in[0,1]$}, and without loss of generality, we assume that {\small$\t_{1}\leq\t_{2}$}. We compute
\begin{align}
&\int_{0}^{\t_{2}}\mathbf{H}^{N}_{\s,\t_{2},\x,0}\cdot N\mathfrak{m}_{\mathrm{left}}[\eta_{\s}]\mathbf{Z}^{N}_{\s,0}\d\s-\int_{0}^{\t_{1}}\mathbf{H}^{N}_{\s,\t_{1},\x,0}\cdot N\mathfrak{m}_{\mathrm{left}}[\eta_{\s}]\mathbf{Z}^{N}_{\s,0}\d\s\nonumber\\
&=\int_{\t_{1}}^{\t_{2}}\mathbf{H}^{N}_{\s,\t_{2},\x,0}\cdot N\mathfrak{m}_{\mathrm{left}}[\eta_{\s}]\mathbf{Z}^{N}_{\s,0}\d\s+\int_{0}^{\t_{1}}(\mathbf{H}^{N}_{\s,\t_{2},\x,0}-\mathbf{H}^{N}_{\s,\t_{1},\x,0})\cdot N\mathfrak{m}_{\mathrm{left}}[\eta_{\s}]\mathbf{Z}^{N}_{\s,0}\d\s.\nonumber
\end{align}
By {Proposition \ref{prop:hke}}, we know that {\small$|\mathbf{H}^{N}_{\s,\t_{2},\x,0}-\mathbf{H}^{N}_{\s,\t_{1},\x,0}|\lesssim N^{-1}|\t_{1}-\s|^{-3/4}\cdot |\t_{2}-\t_{1}|^{1/4}$}. Thus, we can control the first term in the second line using the argument in \eqref{eq:moments0} (in which we get another scaling factor of {\small$|\t_{2}-\t_{1}|$} because the integration is over a length {\small$|\t_{2}-\t_{1}|$} interval). We can also control the second term in the last line above using the argument that gave \eqref{eq:moments0c} (except {\small$N^{-1/2}|\x-\y|^{-1/2}$} gets traded for {\small$|\t_{2}-\t_{1}|^{1/4}$}; see Proposition \ref{prop:hke}). Ultimately, we obtain
\begin{align}
&\|\int_{0}^{\t_{2}}\mathbf{H}^{N}_{\s,\t_{2},\x,0}\cdot N\mathfrak{m}_{\mathrm{left}}[\eta_{\s}]\mathbf{Z}^{N}_{\s,0}\d\s-\int_{0}^{\t_{1}}\mathbf{H}^{N}_{\s,\t_{1},\x,0}\cdot N\mathfrak{m}_{\mathrm{left}}[\eta_{\s}]\mathbf{Z}^{N}_{\s,0}\d\s\|_{2p}^{2}\nonumber\\
&+\|\int_{0}^{\t_{2}}\mathbf{H}^{N}_{\s,\t_{2},\x,N}\cdot N\mathfrak{m}_{\mathrm{right}}[\eta_{\s}]\mathbf{Z}^{N}_{\s,N}\d\s-\int_{0}^{\t_{1}}\mathbf{H}^{N}_{\s,\t_{1},\x,N}\cdot N\mathfrak{m}_{\mathrm{right}}[\eta_{\s}]\mathbf{Z}^{N}_{\s,N}\d\s\|_{2p}^{2}\nonumber\\
&\lesssim|\t_{2}-\t_{1}|^{\frac12}\cdot\sup_{\t\in[0,1]}\sup_{\x\in\llbracket0,N\rrbracket}\|\mathbf{Z}^{N}_{\t,\x}\|_{2p}^{2}.\label{eq:moments0e}
\end{align}
The rest of the proof now follows by what is done in \cite{P} (since the only other terms in \eqref{eq:msheIIa}-\eqref{eq:msheIIc} are ones that appear already in \cite{P}). In particular, by (32) in \cite{P} and \eqref{eq:moments0}-\eqref{eq:moments0b}, we have 
\begin{align}
\|\mathbf{Z}^{N}_{\t,\x}\|_{2p}^{2}&\lesssim_{p}1+\int_{0}^{\t}|\t-\s|^{-\frac12}\sup_{\x\in\llbracket0,N\rrbracket}\|\mathbf{Z}^{N}_{\s,\x}\|_{2p}^{2}\d\s.
\end{align}
We can replace the \abbr{LHS} by a supremum over {\small$\x\in\llbracket0,N\rrbracket$}. We can then apply Gronwall's inequality to the resulting estimate and obtain
\begin{align*}
\sup_{\t\in[0,1]}\sup_{\x\in\llbracket0,N\rrbracket}\|\mathbf{Z}^{N}_{\t,\x}\|_{2p}^{2}\lesssim\sup_{\x\in\llbracket0,N\rrbracket}\|\mathbf{Z}^{N}_{0,\x}\|_{2p}^{2}\cdot\sup_{\t\in[0,1]}\exp(\int_{0}^{\t}|\t-\s|^{-\frac12}\d\s)\lesssim1,
\end{align*}
where the last bound follows from our assumptions on the initial data. This establishes the first estimate \eqref{eq:momentsI}. The proofs of \eqref{eq:momentsII}-\eqref{eq:momentsIII} follow by using \eqref{eq:momentsI} to estimate space and time regularity of the \abbr{RHS} of \eqref{eq:msheIIa} exactly as written in the proof of Proposition 5.4 in \cite{P}, as well as using \eqref{eq:moments0c}, \eqref{eq:moments0d}, and \eqref{eq:moments0e} to estimate space and time regularity of \eqref{eq:msheIIb}-\eqref{eq:msheIIc}. We do not reproduce the identical details here.
\end{proof}
\begin{corollary}\label{corollary:tightness}
\fsp Retain the setting of Theorem \ref{theorem:main}. The process {\small$(\t,\X)\mapsto\mathbf{Z}^{N}_{\t,N\X}$}, if we extend it from {\small$\{0,\frac{1}{N},\ldots,1\}$} to {\small$[0,1]$} via linear interpolation, is tight in the large-{\small$N$} limit in the Skorokhod space {\small$\mathscr{D}([0,1],\mathscr{C}([0,1]))$}. Moreover, any limit point of this tight sequence must live in the space of continuous paths {\small$\mathscr{C}([0,1],\mathscr{C}([0,1]))$}.
\end{corollary}
\begin{proof}
For tightness, it suffices to combine Lemma \ref{lemma:moments} and a standard Kolmogorov continuity argument; see the proof of Proposition 4.7 in \cite{CS}. Continuity of subsequential limits follows because the jumps in {\small$\mathbf{Z}^{N}$} have size of {\small$\mathrm{O}(N^{-1/2})\to0$} (see Corollary \ref{corollary:mshe}), and any element in a Skorokhod space with no jumps is continuous.
\end{proof}
\subsection{Identification of limit points}
By Corollary \ref{corollary:tightness}, to prove Theorem \ref{theorem:main}, it suffices to identify subsequential limits of {\small$\mathbf{Z}^{N}$} as solutions to the open \abbr{SHE} (at least in law). To this end, we employ the following characterization of the law of \eqref{eq:openshe}-\eqref{eq:opensheII} in terms of a martingale problem. The following follows by Proposition 5.9 in \cite{CS}.
\begin{prop}\label{prop:mgproblem}
\fsp Suppose that a probability measure {\small$\mathbb{Q}$} on the space {\small$\mathscr{C}([0,1],\mathscr{C}([0,1]))$} of continuous {\small$\mathscr{C}([0,1])$}-valued paths satisfies the following properties. 
\begin{itemize}
\item First, define the following suitable space of Robin test functions:
\begin{align*}
\mathscr{C}^{\infty}_{\mathbf{A},\mathbf{B}}:=\Big\{\varphi\in\mathscr{C}^{\infty}_{\mathrm{c}}(\R): \varphi'_{0}=-\lambda \mathbf{A}\varphi_{0} \ \ \text{and} \ \ \varphi'_{1}=-\lambda \mathbf{B}\varphi_{1}\Big\}.
\end{align*}
Second, let {\small$\mathscr{Z}$} be a random element in {\small$\mathscr{C}([0,1],\mathscr{C}([0,1]))$} distributed according to {\small$\mathbb{Q}$}.
\item We assume that {\small$\E|\mathscr{Z}_{\t,\x}|^{2}\leq\mathrm{C}$} by a constant {\small$\mathrm{C}<\infty$} that is independent of {\small$(\t,\x)\in[0,1]\times[0,1]$}.
\item We assume that the following processes are {\small$\mathbb{Q}$}-martingales in {\small$\t\in[0,1]$} for any deterministic {\small$\varphi\in\mathscr{C}_{\mathbf{A},\mathbf{B}}^{\infty}$}:
\begin{align*}
\mathscr{N}_{\t}&:=\int_{[0,1]}\mathscr{Z}_{\t,\x}\varphi_{\x}\d\x-\int_{[0,1]}\mathscr{Z}_{0,\x}\varphi_{\x}\d\x-\int_{0}^{\t}\int_{[0,1]}\mathscr{Z}_{\s,\x}\cdot\tfrac12\partial_{\x}^{2}\varphi_{\x}\d\x\d\s,\\
\mathscr{Q}_{\t}&:=\mathscr{N}_{\t}^{2}-\lambda^{2}\int_{0}^{\t}\int_{[0,1]}\mathscr{Z}_{\s,\x}^{2}\varphi_{\x}^{2}\d\x\d\s.
\end{align*}
\end{itemize}
Then {\small$\mathbb{Q}$} is equal to the law of \eqref{eq:openshe}-\eqref{eq:opensheII} with initial data {\small$\mathscr{Z}_{0,\cdot}$}.
\end{prop}
In view of Proposition \ref{prop:mgproblem}, our main goal will be to establish that any subsequential limit point constructed in Corollary \ref{corollary:tightness} satisfies the martingale problem spelled out in Proposition \ref{prop:mgproblem} (with the correct initial data, though this follows automatically by construction). We first remark that the second moment condition {\small$\E|\mathscr{Z}_{\t,\x}|^{2}\lesssim1$} from Proposition \ref{prop:mgproblem} holds for any subsequential limit from Corollary \ref{corollary:tightness} by \eqref{eq:momentsI} and Fatou's lemma. Therefore, we will focus on the martingale properties in Proposition \ref{prop:mgproblem}.

We first fix any {\small$\varphi\in\mathscr{C}_{\mathbf{A},\mathbf{B}}^{\infty}$}. Next, we introduce the following pairing notation (like in \cite{CS}) for a Riemann sum approximation to spatial integration, in which {\small$\psi:\llbracket0,N\rrbracket\to\R$} is arbitrary:
\begin{align}
(\psi,\varphi)_{N}:=\tfrac{1}{N}\sum_{\x\in\llbracket0,N\rrbracket}\psi_{\x}\varphi_{N^{-1}\x}.
\end{align}
We clarify that the {\small$\varphi$} function is rescaled by {\small$N^{-1}$} in space in order to make it a function on {\small$\llbracket0,N\rrbracket$}. Now, we use \eqref{eq:msheIa}-\eqref{eq:msheIb}. This gives the following equation:
\begin{align}
(\mathbf{Z}^{N}_{\t,\cdot},\varphi)_{N}-(\mathbf{Z}^{N}_{0,\cdot},\varphi)_{N}&=\int_{0}^{\t}(\tfrac12N^{2}\Delta_{\mathbf{A},\mathbf{B}}\mathbf{Z}^{N}_{\s,\cdot},\varphi)_{N}\d\s+\int_{0}^{\t}(\mathbf{Z}^{N}_{\s,\cdot}\d\mathscr{Q}_{\s,\cdot},\varphi)_{N}\label{eq:identify1a}\\
&+\lambda \int_{0}^{\t}\mathfrak{f}_{\mathrm{left}}[\eta_{\s}]\mathbf{Z}^{N}_{\s,0}\cdot\varphi_{0}\d\s+\lambda \int_{0}^{\t}\mathfrak{f}_{\mathrm{right}}[\eta_{\s}]\mathbf{Z}^{N}_{\s,N}\cdot\varphi_{1}\d\s\label{eq:identify1b}\\
&+\int_{0}^{\t}N^{-\frac12}\mathfrak{b}_{\mathrm{left}}[\eta_{\s}]\mathbf{Z}^{N}_{\s,0}\cdot\varphi_{0}\d\s+\int_{0}^{\t}N^{-\frac12}\mathfrak{b}_{\mathrm{right}}[\eta_{\s}]\mathbf{Z}^{N}_{\s,N}\cdot\varphi_{1}\d\s.\label{eq:identify1c}
\end{align}
Now, we make two observations. The first is that the last term in \eqref{eq:identify1a} is a martingale in the {\small$\t$}-variable, since {\small$\mathbf{Z}^{N}$} is an adapted process, and {\small$\mathscr{Q}_{\t}$} is a compensated Poisson process. The second is that by (5.24) in \cite{CS}, we have 
\begin{align}
\int_{0}^{\t}(\tfrac12N^{2}\Delta_{\mathbf{A},\mathbf{B}}\mathbf{Z}^{N}_{\s,\cdot},\varphi)_{N}\d\s=\int_{0}^{\t}(\mathbf{Z}^{N}_{\s,\cdot},\tfrac12\varphi'')_{N}\d\s+\mathbf{R}_{0,\t}^{N}[\varphi],\label{eq:identify2}
\end{align}
where {\small$\mathbf{R}_{0,\t}^{N}[\varphi]$} vanishes in probability as {\small$N\to\infty$}, and {\small$\varphi''$} is the continuum second-derivative of {\small$\varphi$}. (This bound \eqref{eq:identify2} is the statement that the continuum Robin Laplacian is self-adjoint with respect to Lebesgue measure on {\small$[0,1]$}, plus error corrections which vanish in the large-{\small$N$} limit to take into account the difference between discrete and continuum space.) Ultimately, we deduce that 
\begin{align}
\mathscr{N}^{N}_{\t}&=(\mathbf{Z}^{N}_{\t,\cdot},\varphi)_{N}-(\mathbf{Z}^{N}_{0,\cdot},\varphi)_{N}-\int_{0}^{\t}(\mathbf{Z}^{N}_{\s,\cdot},\tfrac12\varphi'')_{N}\d\s+\sum_{\k=0,\ldots,4}\mathbf{R}_{\k,\t}^{N}[\varphi],\label{eq:identify3a}
\end{align}
is a martingale in {\small$\t\in[0,1]$}, where {\small$\mathbf{R}^{N}_{0,\t}[\varphi]$} is from \eqref{eq:identify2}, and 
\begin{align}
\mathbf{R}_{1,\t}^{N}[\varphi]&:=\lambda \int_{0}^{\t}\mathfrak{f}_{\mathrm{left}}[\eta_{\s}]\mathbf{Z}^{N}_{\s,0}\cdot\varphi_{0}\d\s,\label{eq:identify3b}\\
\mathbf{R}_{2,\t}^{N}[\varphi]&:=\lambda \int_{0}^{\t}\mathfrak{f}_{\mathrm{right}}[\eta_{\s}]\mathbf{Z}^{N}_{\s,N}\cdot\varphi_{1}\d\s,\label{eq:identify3c}\\
\mathbf{R}_{3,\t}^{N}[\varphi]&:=\int_{0}^{\t}N^{-\frac12}\mathfrak{b}_{\mathrm{left}}[\eta_{\s}]\mathbf{Z}^{N}_{\s,0}\cdot\varphi_{0}\d\s,\label{eq:identify3d}\\
\mathbf{R}_{4,\t}^{N}[\varphi]&:=\int_{0}^{\t}N^{-\frac12}\mathfrak{b}_{\mathrm{right}}[\eta_{\s}]\mathbf{Z}^{N}_{\s,N}\cdot\varphi_{1}\d\s.\label{eq:identify3e}
\end{align}
We will now compute the predictable bracket of {\small$\mathscr{N}^{N}_{\t}$}, as this will generate a martingale consisting of a compensation of {\small$(\mathscr{N}^{N}_{\t})^{2}$}. In particular, we are left to compute the following object:
\begin{align}
\Big[\int_{0}^{\t}(\mathbf{Z}^{N}_{\s,\cdot}\d\mathscr{Q}_{\s,\cdot},\varphi)_{N},\int_{0}^{\t}(\mathbf{Z}^{N}_{\s,\cdot}\d\mathscr{Q}_{\s,\cdot},\varphi)_{N}\Big]&=\int_{0}^{\t}\tfrac{1}{N^{2}}\sum_{\x,\y\in\llbracket0,N\rrbracket}\mathbf{Z}^{N}_{\s,\x}\mathbf{Z}^{N}_{\s,\y}\varphi_{N^{-1}\x}\varphi_{N^{-1}\y}\d[\mathscr{Q}^{N}_{\s,\x},\mathscr{Q}^{N}_{\s,\y}].\label{eq:identity4}
\end{align}
We appeal to the formulas in Lemmas \ref{lemma:mshebulk}, \ref{lemma:msheleft}, and \ref{lemma:msheright} for the compensated Poisson process {\small$\mathscr{Q}$}. Using these, we make the following observations.
\begin{enumerate}
\item First, fix {\small$\t\in[0,1]$} and {\small$\x\in\llbracket1,N-1\rrbracket$}. By \eqref{eq:mshebulkIIa}-\eqref{eq:mshebulkIIb}, and because {\small$\mathscr{Q}^{\to},\mathscr{Q}^{\leftarrow}$} are independent, we have
\begin{align*}
\d[\mathscr{Q}_{\t,\x},\mathscr{Q}_{\t,\x}]&=(\tfrac12N^{2}-\tfrac12\lambda N^{\frac32})\cdot\Big\{\exp(2\lambda N^{-\frac12})-1\Big\}^{2}\mathbf{1}_{\eta_{\t,\x}=1}\mathbf{1}_{\eta_{\t,\x+1}=-1}\d\t\\
&+(\tfrac12N^{2}+\tfrac12\lambda N^{\frac32})\cdot\Big\{\exp(-2\lambda N^{-\frac12})-1\Big\}^{2}\mathbf{1}_{\eta_{\t,\x}=-1}\mathbf{1}_{\eta_{\t,\x+1}=1}\d\t\\
&=2\lambda^{2}N\cdot\Big(\mathbf{1}_{\eta_{\t,\x}=1}\mathbf{1}_{\eta_{\t,\x+1}=-1}+\mathbf{1}_{\eta_{\t,\x}=-1}\mathbf{1}_{\eta_{\t,\x+1}=1}\Big)+\mathrm{O}(N^{\frac12}).
\end{align*}
\item If we follow point (1) above and then use \eqref{eq:msheleftIIa}-\eqref{eq:msheleftIIb} and \eqref{eq:msherightIIa}-\eqref{eq:msherightIIb} we know that at the boundary points {\small$\x\in\{0,N\}$}, we have 
\begin{align}
\d[\mathscr{Q}_{\t,\x},\mathscr{Q}_{\t,\x}]&=\mathrm{O}(N)\d\t.\nonumber
\end{align}
\item Finally, because {\small$\t\mapsto\mathscr{Q}_{\t,\x}$} and {\small$\t\mapsto\mathscr{Q}_{\t,\y}$} are independent for {\small$\x\neq\y$}, we have {\small$\d[\mathscr{Q}_{\t,\x},\mathscr{Q}_{\t,\y}]=0$} if {\small$\x\neq\y$}.
\end{enumerate}
By points (1)-(4) above, we can now rewrite \eqref{eq:identity4} as
\begin{align}
&\Big[\int_{0}^{\t}(\mathbf{Z}^{N}_{\s,\cdot}\d\mathscr{Q}_{\s,\cdot},\varphi)_{N},\int_{0}^{\t}(\mathbf{Z}^{N}_{\s,\cdot}\d\mathscr{Q}_{\s,\cdot},\varphi)_{N}\Big]\nonumber\\
&=\int_{0}^{\t}\tfrac{1}{N}\sum_{\x=1,\ldots,N-1}2\lambda^{2}\Big(\mathbf{1}_{\eta_{\s,\x}=1}\mathbf{1}_{\eta_{\s,\x+1}=-1}+\mathbf{1}_{\eta_{\s,\x}=-1}\mathbf{1}_{\eta_{\s,\x+1}=1}\Big)|\mathbf{Z}^{N}_{\s,\x}|^{2}|\varphi_{N^{-1}\x}|^{2}\d\s\nonumber\\
&+\int_{0}^{\t}\tfrac{1}{N}\sum_{\x=1,\ldots,N-1}\mathrm{O}(N^{-\frac12})|\mathbf{Z}^{N}_{\s,\x}|^{2}|\varphi_{N^{-1}\x}|^{2}\d\s+\int_{0}^{\t}\mathrm{O}(N^{-1})|\mathbf{Z}^{N}_{\s,0}|^{2}|\varphi_{0}|^{2}\d\s+\int_{0}^{\t}\mathrm{O}(N^{-1})|\mathbf{Z}^{N}_{\s,N}|^{2}|\varphi_{1}|^{2}\d\s.\nonumber
\end{align}
Since the {\small$\mathbf{Z}^{N}$} process is tight (after rescaling space), the last line vanishes uniformly in {\small$\t\in[0,1]$} in the large-{\small$N$} limit. Moreover, since {\small$\eta_{\t,\y}\in\{\pm1\}$}, we can rewrite the indicator functions as 
\begin{align*}
\mathbf{1}_{\eta_{\s,\x}=1}\mathbf{1}_{\eta_{\s,\x+1}=-1}+\mathbf{1}_{\eta_{\s,\x}=-1}\mathbf{1}_{\eta_{\s,\x+1}=1}&=\tfrac{1+\eta_{\s,\x}}{2}\tfrac{1-\eta_{\s,\x+1}}{2}+\tfrac{1-\eta_{\s,\x}}{2}\tfrac{1+\eta_{\s,\x+1}}{2}=\tfrac12-\tfrac12\eta_{\s,\x}\eta_{\s,\x+1}.
\end{align*}
Combining the previous two displays gives
\begin{align}
\Big[\int_{0}^{\t}(\mathbf{Z}^{N}_{\s,\cdot}\d\mathscr{Q}_{\s,\cdot},\varphi)_{N},\int_{0}^{\t}(\mathbf{Z}^{N}_{\s,\cdot}\d\mathscr{Q}_{\s,\cdot},\varphi)_{N}\Big]&=\lambda^{2}\int_{0}^{\t}(|\mathbf{Z}^{N}_{\s,\cdot}|^{2},|\varphi|^{2})_{N}\d\s+\mathbf{R}_{5,\t}^{N}[\varphi]+\mathbf{R}_{6,\t}^{N}[\varphi],\nonumber
\end{align}
where {\small$\mathbf{R}^{N}_{6,\t}[\varphi]\to0$} uniformly in {\small$\t\in[0,1]$}, and where
\begin{align}
\mathbf{R}^{N}_{5,\t}[\varphi]&:=-\tfrac{\lambda^{2}}{2}\int_{0}^{\t}\tfrac{1}{N}\sum_{\x=1,\ldots,N-1}\eta_{\s,\x}\eta_{\s,\x+1}|\mathbf{Z}^{N}_{\s,\x}|^{2}|\varphi_{N^{-1}\x}|^{2}\d\s.\nonumber
\end{align}
Ultimately, we deduce that the following process in {\small$\t\in[0,1]$} is a martingale (where {\small$\mathscr{N}^{N}$} is from \eqref{eq:identify3a}):
\begin{align}
(\mathscr{N}^{N}_{\t})^{2}-\lambda^{2}\int_{0}^{\t}(|\mathbf{Z}^{N}_{\s,\cdot}|^{2},|\varphi|^{2})_{N}\d\s-\mathbf{R}^{N}_{5,\t}[\varphi]-\mathbf{R}^{N}_{6,\t}[\varphi].\label{eq:identity5}
\end{align}
Now, let {\small$\mathbf{Z}$} be any subsequential limit of {\small$(\t,\X)\mapsto\mathbf{Z}^{N}_{\t,N\X}$}. By a standard approximation of integrals via Riemann sums, we get the following convergence as {\small$N\to\infty$} (uniformly over {\small$\s\in[0,1]$} in probability) for any {\small$\wt{\varphi}\in\mathscr{C}^{\infty}(\R)$}:
\begin{align*}
(\mathbf{Z}^{N}_{\s,\cdot},\wt{\varphi})_{N}\to\int_{[0,1]}\mathbf{Z}_{\s,\X}\wt{\varphi}_{\X}\d\X \quad\text{and}\quad (|\mathbf{Z}^{N}_{\s,\cdot}|^{2},\wt{\varphi})_{N}\to\int_{[0,1]}|\mathbf{Z}_{\s,\X}|^{2}\wt{\varphi}_{\X}\d\X.
\end{align*}
Combining the previous display with \eqref{eq:identify3a}, \eqref{eq:identity5}, and Proposition \ref{prop:mgproblem} implies that to finish the proof of Theorem \ref{theorem:main}, it suffices to show the following stochastic estimate.
\begin{prop}\label{prop:stoch}
\fsp First, we have the following vanishings in probability in the large-{\small$N$} limit:
\begin{align}
\lim_{N\to\infty}\sup_{\k=0,\ldots,4}\sup_{\t\in[0,1]}|\mathbf{R}^{N}_{\k,\t}[\varphi]|=0.\label{eq:stochI}
\end{align}
Second, for any {\small$\x\in\llbracket0,N\rrbracket$}, we let {\small$\mathfrak{a}_{\x}:\{\pm1\}^{\mathbb{K}_{N}}\to\R$} be a function satisfying the following.
\begin{enumerate}
\item We have the deterministic bound {\small$\sup_{\x,\eta}|\mathfrak{a}_{\x}[\eta]|\lesssim1$}.
\item There exists {\small$\mathfrak{m}\lesssim1$} independent of {\small$N$} so that {\small$\mathfrak{a}_{\x}$} is ``{\small$\mathfrak{m}$}-local" in the following sense. For any {\small$\x\in\llbracket0,N\rrbracket$} and {\small$\eta\in\{\pm1\}^{\mathbb{K}_{N}}$}, the quantity {\small$\mathfrak{a}_{\x}[\eta]$} depends only on {\small$\eta_{\w}$} for {\small$\w\in\mathbb{K}_{N}$} such that {\small$|\x-\w|\leq\mathfrak{m}$}.
\item Finally, we assume that {\small$\E^{0}\mathfrak{a}_{\x}=0$} for all {\small$\x\in\llbracket0,N\rrbracket$}.
\end{enumerate}
Then, we have the following vanishing in probability in the large-{\small$N$} limit, in which {\small$\wt{\varphi}\in\mathscr{C}^{\infty}(\R)$} is arbitrary:
\begin{align}
\lim_{N\to\infty}\sup_{\t\in[0,1]}\Big|\int_{0}^{\t}\tfrac{1}{N}\sum_{\x\in\llbracket0,N\rrbracket}\mathfrak{a}_{\x}[\eta_{\s}]|\mathbf{Z}^{N}_{\s,\x}|^{2}\cdot\wt{\varphi}_{N^{-1}\x}\d\s\Big| = 0.\label{eq:stochII}
\end{align}
\end{prop}
%
%
%
\section{Preliminary stochastic estimates}\label{section:stochestimates}
In this section, we gather various estimates that will be used in the proof of Proposition \ref{prop:stoch}. First, we record various estimates that are standard in hydrodynamic limits (e.g. entropy production, \abbr{LSI}, and local equilibrium). Then, we move onto the more technical (and much less standard) Kipnis-Varadhan-type estimate with respect to a \emph{non-invariant} measure. 
\subsection{Tools from hydrodynamic limits}
The goal of this subsection is to prove a number of estimates regarding how close the law of the dynamics is to the \emph{non-invariant} probability measure {\small$\mathbb{P}^{0}$} on local scales near the boundary {\small$\{1,N\}$}, and to a family of non-invariant measures obtained by conditioning {\small$\mathbb{P}^{0}$} on the particle density (again at local scales) in the bulk. We make this precise when it becomes more relevant.

The first step is a classical estimate for the growth of relative entropy with respect to the non-invariant measure {\small$\mathbb{P}^{0}$}. We start by introducing the main functionals of interest.
\begin{definition}\label{definition:lyapunov}
\fsp Let {\small$\mathfrak{P}$} be any probability density on {\small$\{\pm1\}^{\mathbb{K}_{N}}$} with respect to {\small$\mathbb{P}^{0}$}. We let the relative entropy of {\small$\mathfrak{P}$} with respect to {\small$\mathbb{P}^{0}$} be given by
\begin{align}
\mathrm{H}^{0}[\mathfrak{P}]:=\E^{0}\mathfrak{P}\log\mathfrak{P}.\label{eq:entropy}
\end{align}
Next, we let {\small$\mathfrak{D}^{0}[\mathfrak{P}]$} be the following ``Fisher information" of {\small$\mathfrak{P}$} with respect to {\small$\mathbb{P}^{0}$}, in which {\small$\mathscr{L}_{\x}$} is the generator for a symmetric simple exclusion process of speed {\small$1$} on the bond {\small$\{\x,\x+1\}$} (see the notation after \eqref{eq:generatorIa}), and in which {\small$\mathscr{S}_{\x}$} is the generator for speed {\small$1$} spin-flipping at {\small$\x$} (see the notation after \eqref{eq:generatorIc}-\eqref{eq:generatorId}):
\begin{align}
\mathfrak{D}^{0}[\mathfrak{P}]:=\Big\{\sum_{\x,\x+1\in\mathbb{K}_{N}}\E^{0}[(\mathscr{L}_{\x}\sqrt{\mathfrak{P}})^{2}]\Big\}+\E^{0}|\mathscr{S}_{1}\sqrt{\mathfrak{P}}|^{2}+\E^{0}|\mathscr{S}_{N}\sqrt{\mathfrak{P}}|^{2}.\label{eq:fi}
\end{align}
\end{definition}
\begin{lemma}\label{lemma:entropyproduction}
\fsp Fix any deterministic {\small$\mathrm{T}>0$}. For any probability density {\small$\mathfrak{P}_{0}$} with respect to {\small$\mathbb{P}^{0}$} and any deterministic {\small$\t\geq0$}, let {\small$\mathfrak{P}_{\t}$} be the probability density for the law of {\small$\eta_{\t}\in\{\pm1\}^{\mathbb{K}_{N}}$} with respect to {\small$\mathbb{P}^{0}$}, assuming that {\small$\eta_{0}\sim\mathfrak{P}_{0}\d\mathbb{P}^{0}$}. We have 
\begin{align}
\int_{0}^{\t}\mathfrak{D}^{0}[\mathfrak{P}_{\s}]\d\s&\lesssim N^{-2}\mathrm{H}[\mathfrak{P}_{0}]+N^{-\frac12}\mathrm{T}\lesssim N^{-1}+N^{-\frac12}\mathrm{T}.\label{eq:entropyproductionI}
\end{align}
\end{lemma}
\begin{proof}
We first claim that the following calculation holds:
\begin{align}
\tfrac{\d}{\d\t}\mathrm{H}[\mathfrak{P}_{\t}]&=\tfrac{\d}{\d\t}\E^{0}\mathfrak{P}_{\t}[\eta]\log\mathfrak{P}_{\t}[\eta]=\E^{0}\partial_{\t}\mathfrak{P}_{\t}[\eta]\cdot\log\mathfrak{P}_{\t}[\eta]+\E^{0}\mathfrak{P}_{\t}[\eta]\cdot\partial_{\t}\log\mathfrak{P}_{\t}[\eta]\nonumber\\
&=\E^{0}\mathfrak{P}_{\t}\mathscr{L}_{N}\log\mathfrak{P}_{\t}[\eta]+\E^{0}\partial_{\t}\mathfrak{P}_{\t}[\eta]=\E^{0}\mathfrak{P}_{\t}\mathscr{L}_{N}\log\mathfrak{P}_{\t}[\eta].\label{eq:entropyproductionI1}
\end{align}
The first line follows from the Leibniz rule. The second line follows by the Kolmogorov equation {\small$\partial_{\t}\mathfrak{P}_{\t}=\mathscr{L}_{N}^{\star}\mathfrak{P}_{\t}$}, where {\small$\star$} means adjoint with respect to the reference measure {\small$\mathbb{P}^{0}$} for the density {\small$\mathfrak{P}_{\t}$}. We also use that {\small$\E^{0}\mathfrak{P}_{\t}[\eta]=1$}, so its time-derivative vanishes. Now, recall {\small$\mathscr{L}_{N}$} from \eqref{eq:generatorIa}. We have 
\begin{align}
\E^{0}\mathfrak{P}_{\t}\mathscr{L}_{N}\log\mathfrak{P}_{\t}[\eta]&=\E^{0}\mathfrak{P}_{\t}[\eta](\mathscr{L}_{N,\mathrm{S}}+\mathscr{L}_{N,\mathrm{A}})\log\mathfrak{P}_{\t}[\eta]\nonumber\\
&+\E^{0}\mathfrak{P}_{\t}[\eta]\mathscr{L}_{N,\mathrm{left}}\log\mathfrak{P}_{\t}[\eta]+\E^{0}\mathfrak{P}_{\t}[\eta]\mathscr{L}_{N,\mathrm{right}}\log\mathfrak{P}_{\t}[\eta].\label{eq:entropyproductionI2}
\end{align}
We now estimate the first term on the \abbr{RHS} of the previous display. By \eqref{eq:generatorIb}, we have the representation
\begin{align}
\mathscr{L}_{N,\mathrm{S}}+\mathscr{L}_{N,\mathrm{A}}&=\sum_{\x,\x+1\in\mathbb{K}_{N}}\mathfrak{s}_{\x}[\eta]\mathscr{L}_{\x}\quad\text{where}\quad \mathfrak{s}_{\x}[\eta]=\tfrac12N^{2}+\tfrac12\lambda N^{\frac32}\left(\mathbf{1}_{\substack{\eta^{}_{\x}=-1\\\eta^{}_{\x+1}=1}}-\mathbf{1}_{\substack{\eta^{}_{\x}=1\\\eta^{}_{\x+1}=-1}}\right),
\end{align}
and where {\small$\mathscr{L}_{\x}$} is the generator for a speed {\small$1$} symmetric simple exclusion on {\small$\{\x,\x+1\}$}. Now, we use the inequality {\small$\log(\a)-\log(\b)\leq2\a^{-1/2}(\sqrt{\b}-\sqrt{\a})$} for {\small$\a,\b>0$} (see the proof of Theorem 9.2 in Appendix 1 of \cite{KL}). We also note that the difference of indicator functions in the above display is equal to {\small$(\eta_{\x+1}-\eta_{\x})/2$}, which can be verified directly since {\small$\eta_{\x},\eta_{\x+1}\in\{\pm1\}$}. All of this ultimately implies that 
\begin{align}
\E^{0}\mathfrak{P}_{\t}[\eta](\mathscr{L}_{N,\mathrm{S}}+\mathscr{L}_{N,\mathrm{A}})\log\mathfrak{P}_{\t}[\eta]&\leq2\E^{0}\sqrt{\mathfrak{P}_{\t}[\eta]}\cdot\sum_{\x,\x+1\in\mathbb{K}_{N}}\mathfrak{s}_{\x}[\eta]\mathscr{L}_{\x}\sqrt{\mathfrak{P}_{\t}[\eta]}\nonumber\\
&=N^{2}\sum_{\x,\x+1\in\mathbb{K}_{N}}\E^{0}\sqrt{\mathfrak{P}_{\t}[\eta]}\mathscr{L}_{\x}\sqrt{\mathfrak{P}_{\t}[\eta]}\nonumber\\
&+\lambda N^{\frac32}\sum_{\x,\x+1\in\mathbb{K}_{N}}\E^{0}\sqrt{\mathfrak{P}_{\t}[\eta]}(\eta_{\x+1}-\eta_{\x})\cdot\mathscr{L}_{\x}\sqrt{\mathfrak{P}_{\t}[\eta]}\nonumber\\
&=-\tfrac12N^{2}\sum_{\x,\x+1\in\mathbb{K}_{N}}\E^{0}|\mathscr{L}_{\x}\sqrt{\mathfrak{P}_{\t}[\eta]}|^{2}\nonumber\\
&+\lambda N^{\frac32}\sum_{\x,\x+1\in\mathbb{K}_{N}}\E^{0}\sqrt{\mathfrak{P}_{\t}[\eta]}(\eta_{\x+1}-\eta_{\x})\cdot\mathscr{L}_{\x}\sqrt{\mathfrak{P}_{\t}[\eta]}.\label{eq:entropyproductionI3}
\end{align}
We clarify that the last identity follows since {\small$\E^{0}\sqrt{\mathfrak{P}_{\t}[\eta]}\mathscr{L}_{\x}\sqrt{\mathfrak{P}_{\t}[\eta]}=-1/2\cdot\E^{0}|\mathscr{L}_{\x}\sqrt{\mathfrak{P}_{\t}[\eta]}|^{2}$} since {\small$\mathbb{P}^{0}$} is invariant under swapping any two spins. Moreover, we know that {\small$\mathscr{L}_{\x}$} is self-adjoint with respect to {\small$\mathbb{P}^{0}$}. Also, an elementary calculation yields the Leibniz-type rule {\small$\mathscr{L}_{\x}(\mathfrak{f}[\eta]\mathfrak{g}[\eta])=\mathscr{L}_{\x}\mathfrak{f}[\eta]\cdot\mathfrak{g}[\eta^{\x,\x+1}]+\mathfrak{f}[\eta]\cdot\mathscr{L}_{\x}\mathfrak{g}[\eta]$}, and another direct calculation shows that {\small$\mathscr{L}_{\x}(\eta_{\x+1}-\eta_{\x})=2(\eta_{\x}-\eta_{\x+1})$}. Therefore, we have 
\begin{align}
&\lambda N^{\frac32}\sum_{\x,\x+1\in\mathbb{K}_{N}}\E^{0}\sqrt{\mathfrak{P}_{\t}[\eta]}(\eta_{\x+1}-\eta_{\x})\cdot\mathscr{L}_{\x}\sqrt{\mathfrak{P}_{\t}[\eta]}\nonumber\\
&=\lambda N^{\frac32}\sum_{\x,\x+1\in\mathbb{K}_{N}}\E^{0}\mathscr{L}_{\x}\Big(\sqrt{\mathfrak{P}_{\t}[\eta]}(\eta_{\x+1}-\eta_{\x})\Big)\sqrt{\mathfrak{P}_{\t}[\eta]}\nonumber\\
&=-\lambda N^{\frac32}\sum_{\x,\x+1\in\mathbb{K}_{N}}\E^{0}(\mathscr{L}_{\x}\sqrt{\mathfrak{P}_{\t}[\eta]})\cdot(\eta_{\x+1}-\eta_{\x})\sqrt{\mathfrak{P}_{\t}[\eta]}+2\lambda N^{\frac32}\E^{0}\mathfrak{P}_{\t}[\eta]\sum_{\x,\x+1\in\mathbb{K}_{N}}(\eta_{\x}-\eta_{\x+1})\nonumber\\
&=-\lambda N^{\frac32}\sum_{\x,\x+1\in\mathbb{K}_{N}}\E^{0}(\mathscr{L}_{\x}\sqrt{\mathfrak{P}_{\t}[\eta]})\cdot(\eta_{\x+1}-\eta_{\x})\sqrt{\mathfrak{P}_{\t}[\eta]}+2\lambda N^{\frac32}\E^{0}\mathfrak{P}_{\t}[\eta]\eta_{1}-2\lambda N^{\frac32}\E^{0}\mathfrak{P}_{\t}[\eta]\eta_{N}.\nonumber
\end{align}
We now move the first term on the far \abbr{RHS} of the previous display to the far \abbr{LHS}. This and the identities {\small$|\eta_{\x}|=1$} and {\small$\E^{0}\mathfrak{P}_{\t}[\eta]=1$} imply that 
\begin{align}
N^{\frac32}\sum_{\x,\x+1\in\mathbb{K}_{N}}\E^{0}\sqrt{\mathfrak{P}_{\t}[\eta]}(\eta_{\x+1}-\eta_{\x})\cdot\mathscr{L}_{\x}\sqrt{\mathfrak{P}_{\t}[\eta]}&\leq 2|\lambda|N^{\frac32}.\nonumber
\end{align}
If we combine the previous display with \eqref{eq:entropyproductionI3}, then we obtain
\begin{align}
\E^{0}\mathfrak{P}_{\t}[\eta](\mathscr{L}_{N,\mathrm{S}}+\mathscr{L}_{N,\mathrm{A}})\log\mathfrak{P}_{\t}[\eta]\leq -\tfrac12N^{2}\sum_{\x,\x+1\in\mathbb{K}_{N}}\E^{0}|\mathscr{L}_{\x}\sqrt{\mathfrak{P}_{\t}[\eta]}|^{2}+2|\lambda| N^{\frac32}.\label{eq:entropyproductionI4}
\end{align}
We will now control the second line of \eqref{eq:entropyproductionI2}, starting with the first term therein. We claim that
\begin{align}
\E^{0}\mathfrak{P}_{\t}[\eta]\mathscr{L}_{N,\mathrm{left}}\log\mathfrak{P}_{\t}[\eta]&\leq2\E^{0}\sqrt{\mathfrak{P}_{\t}[\eta]}\cdot(\tfrac14N^{2}+\alpha[\eta]N^{\frac32})\mathscr{S}_{1}\sqrt{\mathfrak{P}_{\t}[\eta]}\nonumber\\
&=\tfrac12N^{2}\E^{0}\sqrt{\mathfrak{P}_{\t}[\eta]}\mathscr{S}_{1}\sqrt{\mathfrak{P}_{\t}[\eta]}+2N^{\frac32}\E^{0}\sqrt{\mathfrak{P}_{\t}[\eta]}\cdot\alpha[\eta]\cdot\mathscr{S}_{1}\sqrt{\mathfrak{P}_{|t}[\eta]}\nonumber\\
&\leq-\tfrac14N^{2}\E^{0}|\mathscr{S}_{1}\sqrt{\mathfrak{P}_{\t}[\eta]}|^{2}+2\|\alpha\|_{\mathrm{L}^{\infty}(\{\pm1\}^{\mathbb{K}_{N}})}N^{\frac32}(\E^{0}|\mathfrak{P}_{\t}|)^{\frac12}(\E^{0}|\mathscr{S}_{1}\sqrt{\mathfrak{P}_{\t}[\eta]}|^{2})^{\frac12}\nonumber\\
&\leq-\tfrac14N^{2}\E^{0}|\mathscr{S}_{1}\sqrt{\mathfrak{P}_{\t}[\eta]}|^{2}+\mathrm{O}(N^{\frac32}).\nonumber
\end{align}
The first line follows by the same {\small$\log(\a)-\log(\b)\leq2\a^{-1/2}(\sqrt{\b}-\sqrt{\a})$} argument as what gave us the first line in \eqref{eq:entropyproductionI3} (combined with non-negativity of {\small$1/4\cdot N^{2}+\alpha[\eta]N^{3/2}$} for {\small$N$} large and the formula \eqref{eq:generatorIc}). The second line follows by first breaking open the parentheses on the \abbr{RHS} of the first line, and then noting that {\small$\mathscr{S}_{1}$} is self-adjoint with respect to {\small$\E^{0}$}, which implies the identity {\small$\E^{0}\mathfrak{f}\mathscr{S}_{1}\mathfrak{f}=-1/2\cdot\E^{0}|\mathscr{S}_{1}\mathfrak{f}|^{2}$}. The third line follows by Schwarz. The last line follows since {\small$\alpha$} is uniformly bounded, since {\small$\E^{0}|\mathscr{S}_{1}\sqrt{\mathfrak{P}_{\t}[\eta]}|^{2}\lesssim\E^{0}\mathfrak{P}_{\t}[\eta]$} by Schwarz and the fact that {\small$\mathscr{S}_{1}$} keeps {\small$\mathbb{P}^{0}$} invariant, and since {\small$\E^{0}\mathfrak{P}_{\t}=1$}. By the same token, we also have 
\begin{align}
\E^{0}\mathfrak{P}_{\t}[\eta]\mathscr{L}_{N,\mathrm{right}}\log\mathfrak{P}_{\t}[\eta]&\leq-\tfrac14N^{2}\E^{0}|\mathscr{S}_{N}\sqrt{\mathfrak{P}_{\t}[\eta]}|^{2}+\mathrm{O}(N^{\frac32}).\nonumber
\end{align}
We now combine the previous two displays with \eqref{eq:entropyproductionI4} and \eqref{eq:entropyproductionI1}-\eqref{eq:entropyproductionI2} to get
\begin{align}
\tfrac{\d}{\d\t}\mathrm{H}[\mathfrak{P}_{\t}]&\leq-\tfrac14N^{2}\mathfrak{D}^{0}[\mathfrak{P}_{\t}]+\mathrm{O}(N^{\frac32}).\nonumber
\end{align}
By integrating this estimate and noting that {\small$\mathrm{H}[\mathfrak{P}_{\t}]\geq0$} (which is a standard convexity argument as in Appendix 1.8 in \cite{KL}), we then arrive at the first estimate in \eqref{eq:entropyproductionI}. The second estimate in \eqref{eq:entropyproductionI} follows since {\small$\mathrm{H}[\mathfrak{P}_{0}]\lesssim N$}, regardless of what the probability density {\small$\mathfrak{P}_{0}$} is. This finishes the proof.
\end{proof}
Lemma \ref{lemma:entropyproduction} shows that the (time-integrated) Dirichlet form is small. We now record an estimate which allows us to measure how close the law of the particle system is to {\small$\mathbb{P}^{0}$} in terms of the Dirichlet form. We clarify that there are essentially two types of inequalities (one for the law of the system near the boundary and one for the law of the system in the bulk of the interval {\small$\mathbb{K}_{N}$}). However, both types of inequalities deteriorate as the length-scale on which we compare the law to {\small$\mathbb{P}^{0}$} grows. This is essentially necessary, since the law of the particle system at the global scale cannot be close to a non-invariant measure (locally uniformly in time).

Before we can state the estimate, we must introduce the following family of probability measures on smaller state spaces {\small$\{\pm1\}^{\mathbb{L}}$} obtained by conditioning {\small$\mathbb{P}^{0}$} on the particle density on {\small$\mathbb{L}$}, where {\small$\mathbb{L}\subseteq\mathbb{K}_{N}$} is a discrete interval. These measures are necessary since in the bulk of the system, the dynamics conserves the total particle density, so {\small$\mathbb{P}^{0}$} is not ergodic for the bulk dynamics.
\begin{definition}\label{defintion:canonical}
\fsp Consider any {\small$\sigma\in[-1,1]$} and any interval {\small$\mathbb{L}\subseteq\mathbb{K}_{N}$}. We define {\small$\mathbb{P}^{\sigma,\mathbb{L}}$} to be the measure on {\small$\{\pm1\}^{\mathbb{L}}$} obtained by conditioning {\small$\mathbb{P}^{0}$} on the following hyperplane:
\begin{align}
\mathbb{H}^{\sigma,\mathbb{L}}:=\Big\{\sigma\in\{\pm1\}^{\mathbb{L}}:|\mathbb{L}|^{-1}\sum_{\x\in\mathbb{L}}\eta_{\x}=\sigma\Big\}.\label{eq:canonicalI}
\end{align}
(We will only ever consider {\small$\sigma,\mathbb{L}$} for which {\small$\mathbb{P}^{0}(\mathbb{H}^{\sigma,\mathbb{L}})\neq0$}, so that the conditioning is not degenerate.) 

Now, let {\small$\mathfrak{P}$} be a probability density on {\small$\{\pm1\}^{\mathbb{K}_{N}}$} with respect to {\small$\mathbb{P}^{0}$}. We let {\small$\Pi^{\sigma,\mathbb{L}}\mathfrak{P}$} be the density for the law of the pushforward of {\small$\mathfrak{P}\d\mathbb{P}^{0}$} under the projection map {\small$\{\pm1\}^{\mathbb{K}_{N}}\to\{\pm1\}^{\mathbb{L}}$} then conditioned on {\small$\mathbb{H}^{\sigma,\mathbb{L}}$}. We also define the following relative entropy and Fisher information with respect to {\small$\mathbb{P}^{\sigma,\mathbb{L}}$}:
\begin{align}
\mathrm{H}^{\sigma,\mathbb{L}}[\Pi^{\sigma,\mathbb{L}}\mathfrak{P}]&:=\E^{\sigma,\mathbb{L}}\Pi^{\sigma,\mathbb{L}}\mathfrak{P}\cdot\log\Pi^{\sigma,\mathbb{L}}\mathfrak{P},\label{eq:canonicalII}\\
\mathfrak{D}^{\sigma,\mathbb{L}}[\Pi^{\sigma,\mathbb{L}}\mathfrak{P}]&:=\sum_{\x,\x+1\in\mathbb{L}}\E^{\sigma,\mathbb{L}}|\mathscr{L}_{\x}\sqrt{\Pi^{\sigma,\mathbb{L}}\mathfrak{P}}|^{2}.\label{eq:canonicalIII}
\end{align}
We clarify that there is no spin-flip action on the \abbr{RHS} of \eqref{eq:canonicalIII} (compared to \eqref{eq:fi}) because this dynamic does not keep {\small$\mathbb{P}^{\sigma,\mathbb{L}}$} measures invariant (since it does not conserve the particle density on {\small$\mathbb{L}$}).
\end{definition}
We now use Lemma \ref{lemma:entropyproduction} with a standard one-block estimate of hydrodynamic limit theory (see \cite{GPV} for the original source of this method). This ultimately lets us compare the law of the particle system (at local scales) to {\small$\mathbb{P}^{0}$} and {\small$\mathbb{P}^{\sigma,\mathbb{L}_{\mathrm{bulk}}}$} measures.
\begin{lemma}\label{lemma:localeq}
\fsp First, for any {\small$\x\in\llbracket0,N\rrbracket$}, we let {\small$\mathfrak{f}_{\x}:\{\pm1\}^{\mathbb{K}_{N}}\to\R$} be a function satisfying the following.
\begin{enumerate}
\item We have the deterministic bound {\small$\sup_{\x,\eta}|\mathfrak{f}_{\x}[\eta]|\lesssim1$}.
\item There exists {\small$\mathfrak{m}_{N}>0$} depending possibly on {\small$N$} so that for any {\small$\x\in\llbracket0,N\rrbracket$} and {\small$\eta\in\{\pm1\}^{\mathbb{K}_{N}}$}, the quantity {\small$\mathfrak{f}_{\x}[\eta]$} depends only on {\small$\eta_{\w}$} for {\small$\w\in\mathbb{L}_{\x}$}, where {\small$\mathbb{L}_{\x}\subseteq\mathbb{K}_{N}$} is an interval containing {\small$\x$} of size at most {\small$\mathfrak{m}_{N}$}.
\end{enumerate}
Then, we have the following estimate for any deterministic {\small$\mathfrak{t}>0$}:
\begin{align}
\int_{0}^{\mathfrak{t}}\tfrac{1}{N}\sum_{\x\in\llbracket0,N\rrbracket}\E|\mathfrak{f}_{\x}[\eta_{\s}]|\d\s\lesssim_{\mathfrak{t}}\sup_{\x\in\llbracket0,N\rrbracket}\sup_{\sigma\in[-1,1]}\E^{\sigma,\mathbb{L}_{\x}}|\mathfrak{f}_{\x}|+N^{-\frac32}\mathfrak{m}_{N}^{3}.\label{eq:localeqI}
\end{align}
Now, suppose {\small$\mathfrak{d}:\{\pm1\}^{\mathbb{K}_{N}}\to\R$} is such that {\small$\sup_{\eta}|\mathfrak{d}[\eta]|\lesssim1$}, and that {\small$\mathfrak{d}[\eta]$} depends only on {\small$\eta_{\w}$} for {\small$\w\in\mathbb{L}_{\mathfrak{d}}$}, where {\small$\mathbb{L}_{\mathfrak{d}}\subseteq\mathbb{K}_{N}$} is an interval that contains either {\small$1$} or {\small$N$}. Then, for any deterministic {\small$\mathfrak{t}>0$}, we have 
\begin{align}
\int_{0}^{\mathfrak{t}}\E|\mathfrak{d}[\eta_{\s}]|\d\s\lesssim_{\mathfrak{t}}\E^{0}|\mathfrak{d}|+N^{-\frac12}|\mathbb{L}_{\mathfrak{d}}|^{2}.\label{eq:localeqII}
\end{align}
\end{lemma}
\begin{proof}
We first prove \eqref{eq:localeqI}. In what follows, we will let {\small$\Pi^{\mathbb{L}_{\x}}$} denote pushforward (acting on probability measures and densities with respect to {\small$\mathbb{P}^{0}$}) under the projection {\small$\{\pm1\}^{\mathbb{K}_{N}}\to\{\pm1\}^{\mathbb{L}_{\x}}$}. For any {\small$\x$}, we have the following calculation in which {\small$\mathfrak{P}_{\s}[\eta]$} is the density for the law of {\small$\eta_{\s}$} with respect to {\small$\mathbb{P}^{0}$}, since {\small$\mathfrak{f}_{\x}[\eta]$} depends only on {\small$\eta_{\w}$} for {\small$\w\in\mathbb{L}_{\x}$}:
\begin{align}
\E|\mathfrak{f}_{\x}[\eta_{\s}]|&=\E^{0}\mathfrak{P}_{\s}[\eta]\cdot|\mathfrak{f}_{\x}[\eta]|=\E^{0}\Pi^{\mathbb{L}_{\x}}\mathfrak{P}_{\s}[\eta]\cdot|\mathfrak{f}_{\x}[\eta]|\nonumber
\end{align}
Now, we decompose the state space according to the {\small$\eta$}-density on {\small$\mathbb{L}_{\x}$}, so that 
\begin{align}
\E^{0}\Pi^{\mathbb{L}_{\x}}\mathfrak{P}_{\s}[\eta]\cdot|\mathfrak{f}_{\x}[\eta]|&={\textstyle\sum_{\sigma}}\E^{0}\{(\Pi^{\mathbb{L}_{\x}}\mathfrak{P}_{\s}[\eta]\cdot\mathbf{1}_{\mathbb{H}^{\sigma,\mathbb{L}_{\x}}})\cdot|\mathfrak{f}_{\x}[\eta]|\}\nonumber\\
&={\textstyle\sum_{\sigma}}\mathbb{P}^{0}[\mathbb{H}^{\sigma,\mathbb{L}_{\x}}]\cdot\E^{\sigma,\mathbb{L}_{\x}}\{(\Pi^{\mathbb{L}_{\x}}\mathfrak{P}_{\s}[\eta]\cdot\mathbf{1}_{\mathbb{H}^{\sigma,\mathbb{L}_{\x}}})\cdot|\mathfrak{f}_{\x}[\eta]|\},\nonumber
\end{align}
where the sum is over a finite set of {\small$\sigma$} (since the possible {\small$\eta$}-densities on a finite set form a finite set), and where the last line follows by the definition of conditional probability. Now, we want to replace {\small$\Pi^{\mathbb{L}_{\x}}\mathfrak{P}_{\s}[\eta]\cdot\mathbf{1}_{\mathbb{H}^{\sigma,\mathbb{L}_{\x}}}$} by {\small$\Pi^{\sigma,\mathbb{L}_{\x}}\mathfrak{P}_{\s}[\eta]$}. Again, by definition of conditional probability, we can do this if we multiply by {\small$\mathbb{P}_{\s}[\mathbb{H}^{\sigma,\mathbb{L}_{\x}}]\mathbb{P}^{0}[\mathbb{H}^{\sigma,\mathbb{L}_{\x}}]^{-1}$}, where {\small$\d\mathbb{P}_{\s}:=\mathfrak{P}_{\s}\d\mathbb{P}^{0}$}. Therefore, we have 
\begin{align}
{\textstyle\sum_{\sigma}}\mathbb{P}^{0}[\mathbb{H}^{\sigma,\mathbb{L}_{\x}}]\cdot\E^{\sigma,\mathbb{L}_{\x}}\{(\Pi^{\mathbb{L}_{\x}}\mathfrak{P}_{\s}[\eta]\cdot\mathbf{1}_{\mathbb{H}^{\sigma,\mathbb{L}_{\x}}})\cdot|\mathfrak{f}_{\x}[\eta]|\}&={\textstyle\sum_{\sigma}}\mathbb{P}_{\s}[\mathbb{H}^{\sigma,\mathbb{L}_{\x}}]\cdot\E^{\sigma,\mathbb{L}_{\x}}\{\Pi^{\sigma,\mathbb{L}_{\x}}\mathfrak{P}_{\s}[\eta]\cdot|\mathfrak{f}_{\x}[\eta]|\}.\nonumber
\end{align}
Now, we apply the entropy inequality (see Appendix 1.8 in \cite{KL}). This implies that 
\begin{align}
\E^{\sigma,\mathbb{L}_{\x}}\{\Pi^{\sigma,\mathbb{L}_{\x}}\mathfrak{P}_{\s}[\eta]\cdot|\mathfrak{f}_{\x}[\eta]|\}&\lesssim\mathrm{H}^{\sigma,\mathbb{L}_{\x}}[\Pi^{\sigma,\mathbb{L}_{\x}}\mathfrak{P}_{\s}]+\log\E^{\sigma,\mathbb{L}_{\x}}\exp(|\mathfrak{f}_{\x}[\eta]|)\nonumber\\
&\lesssim|\mathbb{L}_{\x}|^{2}\mathfrak{D}^{\sigma,\mathbb{L}_{\x}}[\Pi^{\sigma,\mathbb{L}_{\x}}\mathfrak{P}_{\s}]+\E^{\sigma,\mathbb{L}_{\x}}|\mathfrak{f}_{\x}[\eta]|,\label{eq:localeqI1}
\end{align}
where the last line follows by the log-Sobolev inequality in \cite{Yau} and the inequalities {\small$\exp(|\mathfrak{f}_{\x}[\eta]|)\leq1+\mathrm{O}(|\mathfrak{f}_{\x}[\eta]|)$} (since {\small$\mathfrak{f}_{\x}=\mathrm{O}(1)$}) and {\small$\log(1+\mathrm{C}\cdot\E^{\sigma,\mathbb{L}_{\x}}|\mathfrak{f}_{\x}|)\lesssim\E^{\sigma,\mathbb{L}_{\x}}|\mathfrak{f}_{\x}[\eta]|$} for any positive {\small$\mathrm{C}=\mathrm{O}(1)$}. We combine the previous four displays to obtain
\begin{align*}
\int_{0}^{\t}\tfrac{1}{N}\sum_{\x\in\llbracket0,N\rrbracket}\E|\mathfrak{f}_{\x}[\eta_{\s}]|\d\s&\lesssim_{\t}\sup_{\x\in\llbracket0,N\rrbracket}\sup_{\sigma\in[-1,1]}\E^{\sigma,\mathbb{L}_{\x}}|\mathfrak{f}_{\x}|\d\s+\mathfrak{m}_{N}^{2}\int_{0}^{\t}\tfrac{1}{N}\sum_{\x\in\llbracket0,N\rrbracket}{\textstyle\sum_{\sigma}}\mathbb{P}_{\s}[\mathbb{H}^{\sigma,\mathbb{L}_{\x}}]\cdot\mathfrak{D}^{\sigma,\mathbb{L}_{\x}}[\Pi^{\sigma,\mathbb{L}_{\x}}\mathfrak{P}_{\s}]\d\s.
\end{align*}
Now, we again recall that {\small$\mathbb{P}_{\s}[\mathbb{H}^{\sigma,\mathbb{L}_{\x}}]\Pi^{\sigma,\mathbb{L}_{\x}}\mathfrak{P}_{\s}=\mathbb{P}^{0}[\mathbb{H}^{\sigma,\mathbb{L}_{\x}}]\Pi^{\mathbb{L}_{\x}}\mathfrak{P}_{\s}\cdot\mathbf{1}_{\mathbb{H}^{\sigma,\mathbb{L}_{\x}}}$} by the definition of conditional probability. Next, we recall the definition \eqref{eq:canonicalIII}, and we claim that 
\begin{align*}
{\textstyle\sum_{\sigma}}\mathbb{P}_{\s}[\mathbb{H}^{\sigma,\mathbb{L}_{\x}}]\mathfrak{D}^{\sigma,\mathbb{L}_{\x}}[\Pi^{\mathbb{L}_{\x}}\mathfrak{P}_{\s}]&={\textstyle\sum_{\sigma}}\mathbb{P}_{\s}[\mathbb{H}^{\sigma,\mathbb{L}_{\x}}]\cdot\sum_{\w,\w+1\in\mathbb{L}_{\x}}\E^{\sigma,\mathbb{L}_{\x}}|\mathscr{L}_{\w}\sqrt{\Pi^{\sigma,\mathbb{L}_{\x}}\mathfrak{P}_{\s}}|^{2}\\
&={\textstyle\sum_{\sigma}}\mathbb{P}^{0}[\mathbb{H}^{\sigma,\mathbb{L}_{\x}}]\cdot\sum_{\w,\w+1\in\mathbb{L}_{\x}}\E^{\sigma,\mathbb{L}_{\x}}|\mathscr{L}_{\w}\sqrt{\mathbf{1}_{\mathbb{H}^{\sigma,\mathbb{L}_{\x}}}\Pi^{\mathbb{L}_{\x}}\mathfrak{P}_{\s}}|^{2}\\
&={\textstyle\sum_{\sigma}}\mathbb{P}^{0}[\mathbb{H}^{\sigma,\mathbb{L}_{\x}}]\cdot\sum_{\w,\w+1\in\mathbb{L}_{\x}}\E^{\sigma,\mathbb{L}_{\x}}\mathbf{1}_{\mathbb{H}^{\sigma,\mathbb{L}_{\x}}}|\mathscr{L}_{\w}\sqrt{\Pi^{\mathbb{L}_{\x}}\mathfrak{P}_{\s}}|^{2}\\
&=\sum_{\w,\w+1\in\mathbb{L}_{\x}}\E^{0}|\mathscr{L}_{\w}\sqrt{\Pi^{\mathbb{L}_{\x}}\mathfrak{P}_{\s}}|^{2}\leq\sum_{\w,\w+1\in\mathbb{L}_{\x}}\E^{0}|\mathscr{L}_{\w}\sqrt{\mathfrak{P}_{\s}}|^{2}.
\end{align*}
We have already justified the first two lines in the previous paragraph. The third line follows by the Leibniz rule {\small$\mathscr{L}_{\w}(\mathfrak{f}[\eta]\mathfrak{g}[\eta])=\mathfrak{f}[\eta]\mathscr{L}_{\w}\mathfrak{g}[\eta]+\mathfrak{g}[\eta^{\w,\w+1}]\mathscr{L}_{\w}\mathfrak{f}[\eta]$} for {\small$\mathfrak{f}=\mathbf{1}_{\mathbb{H}^{\sigma,\mathbb{L}_{\x}}}$}, and noting that {\small$\mathfrak{f}$} is unchanged if one swaps spins at points in {\small$\mathbb{L}_{\x}$}. The last line follows by the definition of {\small$\mathbb{P}^{\sigma,\mathbb{L}_{\x}}$} and conditional probability, as well as convexity of the Fisher information (see Appendix 1.10 in \cite{KL}). Thus, by the previous two displays, we have 
\begin{align*}
\int_{0}^{\t}\tfrac{1}{N}\sum_{\x\in\llbracket0,N\rrbracket}\E|\mathfrak{f}_{\x}[\eta_{\s}]|\d\s&\lesssim_{\t}\sup_{\x\in\llbracket0,N\rrbracket}\sup_{\sigma\in[-1,1]}\E^{\sigma,\mathbb{L}_{\x}}|\mathfrak{f}_{\x}|\d\s+\mathfrak{m}_{N}^{2}\int_{0}^{\t}\tfrac{1}{N}\sum_{\x\in\llbracket0,N\rrbracket}\sum_{\w,\w+1\in\mathbb{L}_{\x}}\E^{0}|\mathscr{L}_{\w}\sqrt{\mathfrak{P}_{\s}}|^{2}\d\s\\
&\lesssim\sup_{\x\in\llbracket0,N\rrbracket}\sup_{\sigma\in[-1,1]}\E^{\sigma,\mathbb{L}_{\x}}|\mathfrak{f}_{\x}|\d\s+\mathfrak{m}_{N}^{3}\int_{0}^{\t}\tfrac{1}{N}\sum_{\x\in\llbracket0,N\rrbracket}\E^{0}|\mathscr{L}_{\x}\sqrt{\mathfrak{P}_{\s}}|^{2}\d\s.
\end{align*}
The last line above follows because each bond {\small$\x,\x+1$} on the \abbr{RHS} of the first line is summed over {\small$\lesssim\mathfrak{m}_{N}$}-many times. By definition \eqref{eq:fi}, the last term in the last line is bounded above by the far \abbr{LHS} of \eqref{eq:entropyproductionI} times {\small$N^{-1}\mathfrak{m}_{N}^{3}$}, so the desired estimate \eqref{eq:localeqI} follows by the previous display and \eqref{eq:entropyproductionI}.

It now remains to get \eqref{eq:localeqII}. By assumption in Lemma \ref{lemma:localeq}, the function {\small$\mathfrak{d}[\eta]$} depends only on {\small$\eta_{\w}$} for {\small$\w\in\mathbb{L}_{\mathfrak{d}}$}. Thus, we have {\small$\E|\mathfrak{d}[\eta_{\s}]|=\E^{0}\mathfrak{P}_{\s}[\eta]\cdot|\mathfrak{d}[\eta]|=\E^{0}\Pi^{\mathbb{L}_{\mathfrak{d}}}\mathfrak{P}_{\s}\cdot|\mathfrak{d}[\eta]|$}. Then, we again apply the entropy inequality in order to obtain the estimate
\begin{align}
\E^{0}\{\Pi^{\mathbb{L}_{\mathfrak{d}}}\mathfrak{P}_{\s}\cdot|\mathfrak{d}[\eta]|\}&\lesssim\mathrm{H}^{0}[\Pi^{\mathbb{L}_{\mathfrak{d}}}\mathfrak{P}_{\s}]+\log\E^{0}\exp(|\mathfrak{d}[\eta]|)\lesssim\mathrm{H}^{0}[\Pi^{\mathbb{L}_{\mathfrak{d}}}\mathfrak{P}_{\s}]+\E^{0}|\mathfrak{d}[\eta]|,\nonumber
\end{align}
where the last estimate follows because {\small$|\mathfrak{d}|\lesssim1$} deterministically (see the proof of \eqref{eq:localeqI1}). In view of the previous display, in order to deduce the desired estimate \eqref{eq:localeqII}, it now suffices to show that 
\begin{align}
\int_{0}^{\t}\mathrm{H}^{0}[\Pi^{\mathbb{L}_{\mathfrak{d}}}\mathfrak{P}_{\s}]\d\s\lesssim|\mathbb{L}_{\mathfrak{d}}|^{2}\cdot\int_{0}^{\t}\mathfrak{D}^{0}[\Pi^{\mathbb{L}_{\mathfrak{d}}}\mathfrak{P}_{\s}]\d\s\lesssim N^{-\frac12}|\mathbb{L}_{\mathfrak{d}}|^{2}.\label{eq:localeqII1}
\end{align}
The last inequality in \eqref{eq:localeqII1} follows first by removing {\small$\Pi^{\mathbb{L}_{\mathfrak{d}}}$} (which we can do by convexity of the {\small$\mathfrak{D}^{0}$} operator as in Appendix 1.10 of \cite{KL}) and then by Lemma \ref{lemma:entropyproduction}. Thus, we are left to show the first inequality. First, we claim 
\begin{align*}
\mathrm{H}^{0}[\Pi^{\mathbb{L}_{\mathfrak{d}}}\mathfrak{P}_{\s}]\lesssim\sum_{\w\in\mathbb{L}_{\mathfrak{d}}}\E^{0}|\mathscr{S}_{\w}\sqrt{\Pi^{\mathbb{L}_{\mathfrak{d}}}\mathfrak{P}_{\s}}|^{2}.
\end{align*}
This follows because by the classical log-Sobolev inequality for unbiased spin-flipping with respect to {\small$\mathbb{P}^{0}$} {(see Section 2.2 of \cite{L06})}. We now control each expectation on the \abbr{RHS} of the previous display. Suppose {\small$\mathbb{L}_{\mathfrak{d}}$} contains the boundary point {\small$1$}; if it instead contains {\small$N$}, the same argument applies (just replace the point {\small$1$} by {\small$N$}). Note that flipping {\small$\eta_{\w}$} is the same as swapping {\small$\eta_{\w}\leftrightarrow\eta_{1}$}, flipping {\small$\eta_{1}$}, and then swapping {\small$\eta_{\w}\leftrightarrow\eta_{1}$} back. Thus, we have 
\begin{align}
\E^{0}|\mathscr{S}_{\w}\sqrt{\Pi^{\mathbb{L}_{\mathfrak{d}}}\mathfrak{P}_{\s}}|^{2}&=\E^{0}|\mathscr{F}_{\w,1}\mathscr{S}_{1}\mathscr{F}_{\w,1}\sqrt{\Pi^{\mathbb{L}_{\mathfrak{d}}}\mathfrak{P}_{\s}}|^{2},\nonumber
\end{align}
where we have introduced the swapping operator {\small$\mathscr{F}_{\w,1}\mathfrak{f}[\eta]:=\mathfrak{f}[\eta^{\w,1}]$}. Because {\small$\mathbb{P}^{0}$} is invariant under {\small$\mathscr{F}_{\w,1}$}, we can replace {\small$\mathscr{F}_{\w,1}\mathscr{S}_{1}\mathscr{F}_{\w,1}$} above by {\small$\mathscr{F}_{\w,1}\mathscr{S}_{1}$} (because we can either sample {\small$\eta$} or {\small$\eta^{\w,1}$} without changing the expectation). Now, gathering this, we have the estimate
\begin{align*}
\E^{0}|\mathscr{F}_{\w,1}\mathscr{S}_{1}\mathscr{F}_{\w,1}\sqrt{\Pi^{\mathbb{L}_{\mathfrak{d}}}\mathfrak{P}_{\s}}|^{2}&=\E^{0}|\mathscr{F}_{\w,1}\mathscr{S}_{1}\sqrt{\Pi^{\mathbb{L}_{\mathfrak{d}}}\mathfrak{P}_{\s}}|^{2}\lesssim\E^{0}|\mathscr{S}_{1}\sqrt{\Pi^{\mathbb{L}_{\mathfrak{d}}}\mathfrak{P}_{\s}}|^{2}+\E^{0}|(\mathscr{F}_{\w,1}-\mathrm{Id})\mathscr{S}_{1}\sqrt{\Pi^{\mathbb{L}_{\mathfrak{d}}}\mathfrak{P}_{\s}}|^{2}\\
&\lesssim\E^{0}|\mathscr{S}_{1}\sqrt{\Pi^{\mathbb{L}_{\mathfrak{d}}}\mathfrak{P}_{\s}}|^{2}+\E^{0}|(\mathscr{F}_{\w,1}-\mathrm{Id})\sqrt{\Pi^{\mathbb{L}_{\mathfrak{d}}}\mathfrak{P}_{\s}}|^{2},
\end{align*}
where the last line follows by now using invariance of {\small$\mathbb{P}^{0}$} under {\small$\mathscr{S}_{1}$}. Now, {\small$\mathscr{F}_{\w,1}-\mathrm{Id}$} is the generator corresponding to swapping spins at {\small$\{1,\w\}$}. The Dirichlet energy for such a move can be estimated in terms of sums of Dirichlet energies along bonds {\small$\{\x_{\i},\x_{\i+1}\}$}, where {\small$(\x_{0},\ldots,\x_{\mathrm{L}})$} is a path of nearest-neighbor points that connects {\small$1$} to {\small$\w$}. Precisely, by the ``moving particle lemma" (see Lemma 5.2 in {\cite{Q06}}), we have
\begin{align*}
\E^{0}|(\mathscr{F}_{\w,1}-\mathrm{Id})\sqrt{\Pi^{\mathbb{L}_{\mathfrak{d}}}\mathfrak{P}_{\s}}|^{2}\lesssim|\mathbb{L}_{\mathfrak{d}}|\cdot\sum_{\x,\x+1\in\mathbb{L}_{\mathfrak{d}}}\E^{0}|\mathscr{L}_{\x}\sqrt{\Pi^{\mathbb{L}_{\mathfrak{d}}}\mathfrak{P}_{\s}}|^{2}.
\end{align*}
If we now combine the previous four displays, we obtain
\begin{align}
\mathrm{H}^{0}[\Pi^{\mathbb{L}_{\mathfrak{d}}}\mathfrak{P}_{\s}]&\lesssim|\mathbb{L}_{\mathfrak{d}}|^{2}\E^{0}|\mathscr{S}_{1}\sqrt{\Pi^{\mathbb{L}_{\mathfrak{d}}}\mathfrak{P}_{\s}}|^{2}+|\mathbb{L}_{\mathfrak{d}}|^{2}\cdot\sum_{\x,\x+1\in\mathbb{L}_{\mathfrak{d}}}\E^{0}|\mathscr{L}_{\x}\sqrt{\Pi^{\mathbb{L}_{\mathfrak{d}}}\mathfrak{P}_{\s}}|^{2}.\label{eq:localeqII2}
\end{align}
By definition \eqref{eq:fi}, the \abbr{RHS} of the previous display is {\small$\lesssim|\mathbb{L}_{\mathfrak{d}}|^{2}\mathfrak{D}^{0}[\mathfrak{P}_{\s}]$}, so the first bound in \eqref{eq:localeqII1} follows, and the proof is complete. 
\end{proof}
\subsection{Kipnis-Varadhan-type inequality}
We now record an estimate on time-averages of fluctuations supported near the boundary {\small$\{1,N\}$}. In a nutshell, it says that functions that are mean-zero and supported on {\small$\eta_{\w}$} for {\small$\w$} close to the boundary exhibit cancellations after time-integration; here, mean-zero refers to the probability measure {\small$\mathbb{P}^{0}$}, even though {\small$\mathbb{P}^{0}$} is not invariant! Rather, we will shortly show that {\small$\mathbb{P}^{0}$} is ``almost invariant" in some sense, and that this will be enough to imply cancellations after time-integration.

The estimate at hand will not be with respect to the {\small$\t\mapsto\eta_{\t}$} process, but rather a localization of it.
\begin{definition}\label{definition:localization}
\fsp Fix an interval {\small$\mathbb{L}\subseteq\mathbb{K}_{N}$} containing {\small$1$}, and suppose that {\small$|\mathbb{L}|\lesssim1$}. We let {\small$\mathbb{L}_{\mathrm{fat}}:=\llbracket1,\max\mathbb{L}+N^{\kappa}\rrbracket$} be a ``fattening" of {\small$\mathbb{L}$}, in which {\small$\kappa>0$} is a fixed, small constant. We also let {\small$\t\mapsto\eta^{\mathbb{L}_{\mathrm{fat}}}_{\t}\in\{\pm1\}^{\mathbb{L}_{\mathrm{fat}}}$} be a Markov process with infinitesimal generator given by the following restriction of the generator \eqref{eq:generatorIa} to {\small$\mathbb{L}_{\mathrm{fat}}$}:
\begin{align}
\mathscr{L}_{\mathbb{L}_{\mathrm{fat}}}&:=\mathscr{L}_{N,\mathrm{left}}+\mathscr{L}_{\mathbb{L}_{\mathrm{fat}},\mathrm{S}}+\mathscr{L}_{\mathbb{L}_{\mathrm{fat}},\mathrm{A}},\label{eq:localizationIa}\\
\mathscr{L}_{\mathbb{L}_{\mathrm{fat}},\mathrm{S}}&:=\tfrac12N^{2}\sum_{\x,\x+1\in\mathbb{L}_{\mathrm{fat}}}\mathscr{L}_{\x},\label{eq:localizationIb}\\
\mathscr{L}_{\mathbb{L}_{\mathrm{fat}},\mathrm{A}}&:=\tfrac12\lambda N^{\frac32}\sum_{\x,\x+1\in\mathbb{L}_{\mathrm{fat}}}\Big(\mathbf{1}_{\substack{\eta^{\mathbb{L}_{\mathrm{fat}}}_{\x}=-1}}\mathbf{1}_{{\eta^{\mathbb{L}_{\mathrm{fat}}}_{\x+1}=1}}-\mathbf{1}_{\substack{\eta^{\mathbb{L}_{\mathrm{fat}}}_{\x}=1}}\mathbf{1}_{{\eta^{\mathbb{L}_{\mathrm{fat}}}_{\x+1}=-1}}\Big)\mathscr{L}_{\x}.\label{eq:localizationIc}
\end{align}
The operator {\small$\mathscr{L}_{N,\mathrm{left}}$} is given by \eqref{eq:generatorIc}, except we replace {\small$\eta$} by {\small$\eta^{\mathbb{L}_{\mathrm{fat}}}$}. (The resulting coefficients {\small$\alpha[\eta^{\mathbb{L}_{\mathrm{fat}}}],\gamma[\eta^{\mathbb{L}_{\mathrm{fat}}}]$} are well-defined, since {\small$\mathbb{L}_{\mathrm{fat}}$} is an interval containing {\small$1$} of length at least {\small$1+N^{\kappa}$}, and {\small$\alpha[\eta],\gamma[\eta]$} are functions of only {\small$\eta_{\w}$} for {\small$|\w|\lesssim1$}, and thus of the image of {\small$\eta$} under the projection {\small$\{\pm1\}^{\mathbb{K}_{N}}\to\{\pm1\}^{\mathbb{L}_{\mathrm{fat}}}$}.)

Now, suppose that {\small$\mathbb{L}$} contains {\small$N$} and that {\small$|\mathbb{L}|\lesssim1$}. In this case, we define {\small$\mathbb{L}_{\mathrm{fat}}:=\llbracket\min\mathbb{L}-N^{\kappa},N\rrbracket$}. We also let {\small$\t\mapsto\eta^{\mathbb{L}_{\mathrm{fat}}}\in\{\pm1\}^{\mathbb{L}_{\mathrm{fat}}}$} be the Markov process with infinitesimal generator given by
\begin{align}
\mathscr{L}_{\mathbb{L}_{\mathrm{fat}}}&:=\mathscr{L}_{N,\mathrm{right}}+\mathscr{L}_{\mathbb{L}_{\mathrm{fat}},\mathrm{S}}+\mathscr{L}_{\mathbb{L}_{\mathrm{fat}},\mathrm{A}},\label{eq:localizationIIa}\\
\mathscr{L}_{\mathbb{L}_{\mathrm{fat}},\mathrm{S}}&:=\tfrac12N^{2}\sum_{\x,\x+1\in\mathbb{L}_{\mathrm{fat}}}\mathscr{L}_{\x},\label{eq:localizationIIb}\\
\mathscr{L}_{\mathbb{L}_{\mathrm{fat}},\mathrm{A}}&:=\tfrac12\lambda N^{\frac32}\sum_{\x,\x+1\in\mathbb{L}_{\mathrm{fat}}}\Big(\mathbf{1}_{\substack{\eta^{\mathbb{L}_{\mathrm{fat}}}_{\x}=-1}}\mathbf{1}_{{\eta^{\mathbb{L}_{\mathrm{fat}}}_{\x+1}=1}}-\mathbf{1}_{\substack{\eta^{\mathbb{L}_{\mathrm{fat}}}_{\x}=1}}\mathbf{1}_{{\eta^{\mathbb{L}_{\mathrm{fat}}}_{\x+1}=-1}}\Big)\mathscr{L}_{\x},\label{eq:localizationIIc}
\end{align}
where {\small$\mathscr{L}_{N,\mathrm{right}}$} is given by \eqref{eq:generatorId} (but replacing {\small$\eta$} therein by {\small$\eta^{\mathbb{L}_{\mathrm{fat}}}$}).
\end{definition}
\begin{prop}\label{prop:kv}
\fsp Suppose that {\small$\mathfrak{d}:\{\pm1\}^{\mathbb{K}_{N}}\to\R$} is such that {\small$\sup_{\eta}|\mathfrak{d}[\eta]|\lesssim1$}, and that {\small$\mathfrak{d}[\eta]$} depends only on {\small$\eta_{\w}$} for {\small$\w\in\mathbb{L}_{}$}, in which {\small$\mathbb{L}_{}\subseteq\mathbb{K}_{N}$} is an interval that contains either {\small$1$} or {\small$N$}. Suppose further that {\small$\E^{0}\mathfrak{d}=0$}, and that the interval {\small$\mathbb{L}_{}$} has size {\small$|\mathbb{L}_{}|\lesssim1$}. Then, for any small {\small$\rho>0$} (independent of {\small$N$}), and for {\small$\tau=N^{-2+\rho}$}, we have 
\begin{align}
\E^{0}|\tau^{-1}{\textstyle\int_{0}^{\tau}}\mathfrak{d}[\eta^{\mathbb{L}_{\mathrm{flat}}}_{\s}]\d\s|^{2}\lesssim_{} N^{-2}\tau^{-1}+N^{-\frac13}=N^{-\rho}+N^{-\frac13}.\label{eq:kvI}
\end{align}
We clarify that {\small$\E^{0}$} on the \abbr{LHS} above means expectation with respect to the law of {\small$\t\mapsto\eta^{\mathbb{L}_{\mathrm{fat}}}_{\t}$} assuming that {\small$\eta^{\mathbb{L}_{\mathrm{fat}}}_{0}\sim\mathbb{P}^{0}$}, i.e. that {\small$\eta^{\mathbb{L}_{\mathrm{fat}}}_{0,\w}$} are i.i.d. mean-zero random variables for {\small$\w\in\mathbb{L}_{\mathrm{fat}}$}.
\end{prop}
Again, we clarify that {\small$\mathfrak{d}[\eta]$} depends only on {\small$\eta$} through its image under the projection map {\small$\{\pm1\}^{\mathbb{K}_{N}}\to\{\pm1\}^{\mathbb{L}}$}, and thus {\small$\mathfrak{d}[\eta]$} is naturally a function of the image of {\small$\eta$} under the projection map {\small$\{\pm1\}^{\mathbb{K}_{N}}\to\{\pm1\}^{\mathbb{L}_{\mathrm{fat}}}$} as well. So, the quantity {\small$\mathfrak{d}[\eta^{\mathbb{L}_{\mathrm{fat}}}_{\s}]$} is canonically defined for any {\small$\s\geq0$}.

In view of the condition that {\small$\mathbb{L}_{\mathrm{fat}}$} has cardinality which is {\small$\mathrm{O}(1)$}, it must either contain {\small$1$} or {\small$N$} but not both. For simplicity, we will consider the case where {\small$\mathbb{L}_{\mathrm{fat}}$} contains the point {\small$1$}; the case where it contains {\small$N$} follows from the same argument (after a reflection on {\small$\mathbb{K}_{N}$} which swaps {\small$1\leftrightarrow N$}).

The key ingredient to the proof of Proposition \ref{prop:kv} is an estimate on the semigroup corresponding to {\small$\s\mapsto\eta^{\mathbb{L}_{\mathrm{fat}}}_{\s}$}. In particular, we will compare it to the semigroup generated by the following piece of the generator \eqref{eq:localizationIa}, which contains the speed {\small$N^{2}$}-terms:
\begin{align}
\mathscr{L}_{\mathbb{L}_{\mathrm{fat}},\mathrm{symm}}:=\tfrac12N^{2}\sum_{\x,\x+1\in\mathbb{L}_{\mathrm{fat}}}\mathscr{L}_{\x}+\tfrac14N^{2}\mathscr{S}_{1}.\label{eq:generatorsymm}
\end{align}
Because swapping spins and flipping spins (without any asymmetry or bias towards any particular sign) keeps {\small$\mathbb{P}^{0}$} invariant, it is an elementary calculation to show that {\small$\mathbb{P}^{0}$} \emph{is} an invariant measure for \eqref{eq:generatorsymm}. This is why we want to perturb \eqref{eq:localizationIa} about \eqref{eq:generatorsymm}. Finally, before we state the comparison, we introduce notation for the semigroups. For any {\small$\s\geq0$}, we set
\begin{align}
\mathbf{T}_{\s}:=\exp(\s\mathscr{L}_{\mathbb{L}_{\mathrm{fat}}})\quad\text{and}\quad\mathbf{T}_{\s}^{\mathrm{symm}}:=\exp(\s\mathscr{L}_{\mathbb{L}_{\mathrm{symm}},\mathrm{fat}}).\label{eq:semigroups}
\end{align}
%
\begin{lemma}\label{lemma:semigroups}
\fsp Fix any {\small$\s\in[0,\tau]$} with {\small$\tau=N^{-2+\rho}$} from Proposition \ref{prop:kv} and {\small$\rho>0$} small. We have 
\begin{align}
\|\mathbf{T}_{\s}-\mathbf{T}_{\s}^{\mathrm{symm}}\|_{\mathrm{L}^{\infty}(\{\pm1\}^{\mathbb{L}_{\mathrm{fat}}})\to\mathrm{L}^{\infty}(\{\pm1\}^{\mathbb{L}_{\mathrm{fat}}})}\lesssim N^{-\frac13}.\label{eq:semigroupsI}
\end{align}
\end{lemma}
\begin{proof}
Consider any {\small$\mathfrak{f}:\{\pm1\}^{\mathbb{L}_{\mathrm{fat}}}\to\R$}. We claim that 
\begin{align}
\mathbf{T}_{\s}\mathfrak{f}&=\mathbf{T}_{\s}^{\mathrm{symm}}\mathfrak{f}+\int_{0}^{\s}\mathbf{T}_{\s-\r}^{\mathrm{symm}}\Big\{(\mathscr{L}_{\mathbb{L}_{\mathrm{fat}}}-\mathscr{L}_{\mathbb{L}_{\mathrm{fat}},\mathrm{symm}})\mathbf{T}_{\r}\mathfrak{f}\Big\}\d\r.\label{eq:semigroupsI1}
\end{align}
To see this, consider the function {\small$\r\mapsto\mathbf{T}^{\mathrm{symm}}_{\s-\r}\mathbf{T}_{\r}\mathfrak{f}$}. We can use \eqref{eq:semigroups} to compute the {\small$\r$}-derivative of this quantity, and we note that this quantity equals {\small$\mathbf{T}_{\s}\mathfrak{f}$} at {\small$\r=\s$} and {\small$\mathbf{T}_{\s}^{\mathrm{symm}}\mathfrak{f}$} at time {\small$\r=0$}. Thus, the previous display follows by the fundamental theorem calculus; it is also a version of the Duhamel principle.

We now estimate the \abbr{RHS} of the previous display. To this end, we first claim that 
\begin{align*}
\|(\mathscr{L}_{\mathbb{L}_{\mathrm{fat}}}-\mathscr{L}_{\mathbb{L}_{\mathrm{fat}},\mathrm{symm}})\mathbf{T}_{\r}\mathfrak{f}\|_{\mathrm{L}^{\infty}(\{\pm1\}^{\mathbb{L}_{\mathrm{fat}}})}\lesssim N^{\frac32}|\mathbb{L}_{\mathrm{fat}}|\cdot\|\mathbf{T}_{\r}\mathfrak{f}\|_{\mathrm{L}^{\infty}(\{\pm1\}^{\mathbb{L}_{\mathrm{fat}}})}.
\end{align*}
Indeed, the difference {\small$\mathscr{L}_{\mathbb{L}_{\mathrm{fat}}}-\mathscr{L}_{\mathbb{L}_{\mathrm{fat}},\mathrm{symm}}$} is a sum of {\small$\mathrm{O}(|\mathbb{L}_{\mathrm{fat}}|)$}-many bounded operators that are each scaled by {\small$\lesssim N^{3/2}$}; see \eqref{eq:localizationIa} and \eqref{eq:generatorsymm}. Now, since {\small$\mathbf{T}^{\mathrm{symm}}$} is the semigroup associated to the operator {\small$\mathscr{L}_{\mathbb{L}_{\mathrm{fat}},\mathrm{symm}}$} which keeps {\small$\mathbb{P}^{0}$}-invariant, we know that {\small$\mathbf{T}^{\mathrm{symm}}$} operators are contractive with respect to the {\small$\mathrm{L}^{\infty}(\{\pm1\}^{\mathbb{L}_{\mathrm{fat}}})$}-norm. (This is a standard convexity for semigroups with respect to their invariant measures; see Appendix 1.4 in \cite{KL}.) So, when we combine the previous two displays, we obtain
\begin{align}
\|\mathbf{T}_{\s}\mathfrak{f}\|_{\mathrm{L}^{\infty}(\{\pm1\}^{\mathbb{L}_{\mathrm{fat}}})}&\lesssim\|\mathbf{T}_{\s}^{\mathrm{symm}}\mathfrak{f}\|_{\mathrm{L}^{\infty}(\{\pm1\}^{\mathbb{L}_{\mathrm{fat}}})}+\int_{0}^{\s}N^{\frac32}|\mathbb{L}_{\mathrm{fat}}|\cdot\|\mathbf{T}_{\r}\mathfrak{f}\|_{\mathrm{L}^{\infty}(\{\pm1\}^{\mathbb{L}_{\mathrm{fat}}})}\d\r\nonumber\\
&\lesssim\|\mathfrak{f}\|_{\mathrm{L}^{\infty}(\{\pm1\}^{\mathbb{L}_{\mathrm{fat}}})}+\int_{0}^{\s}N^{\frac32+\kappa}\cdot\|\mathbf{T}_{\r}\mathfrak{f}\|_{\mathrm{L}^{\infty}(\{\pm1\}^{\mathbb{L}_{\mathrm{fat}}})}\d\r,\nonumber
\end{align}
where the second line follows since {\small$|\mathbb{L}_{\mathrm{fat}}|\lesssim N^{\kappa}$} for some small {\small$\kappa$} (see Definition \ref{definition:localization}). We now use the Gronwall inequality to deduce that 
\begin{align}
\|\mathbf{T}_{\s}\mathfrak{f}\|_{\mathrm{L}^{\infty}(\{\pm1\}^{\mathbb{L}_{\mathrm{fat}}})}\lesssim\|\mathfrak{f}\|_{\mathrm{L}^{\infty}(\{\pm1\}^{\mathbb{L}_{\mathrm{fat}}})}\cdot\exp(\int_{0}^{\s}N^{\frac32+\kappa}\d\r)\lesssim\|\mathfrak{f}\|_{\mathrm{L}^{\infty}(\{\pm1\}^{\mathbb{L}_{\mathrm{fat}}})}.\nonumber
\end{align}
since {\small$\s\leq N^{-2+\rho}$} for {\small$\rho>0$} small. If we plug this into the \abbr{RHS} of \eqref{eq:semigroupsI1} and again use contractivity of {\small$\mathbf{T}^{\mathrm{symm}}$} operators with respect to {\small$\|\cdot\|_{\mathrm{L}^{\infty}(\{\pm1\}^{\mathbb{L}_{\mathrm{fat}}})}$}, we obtain
\begin{align*}
\|(\mathbf{T}_{\s}-\mathbf{T}_{\s}^{\mathrm{symm}})\mathfrak{f}\|_{\mathrm{L}^{\infty}(\{\pm1\}^{\mathbb{L}_{\mathrm{fat}}})}\lesssim\int_{0}^{\s}N^{\frac32+\kappa}\|\mathbf{T}_{\r}\mathfrak{f}\|_{\mathrm{L}^{\infty}(\{\pm1\}^{\mathbb{L}_{\mathrm{fat}}})}\d\r\lesssim N^{\frac32+\kappa}N^{-2+\rho}\|\mathfrak{f}\|_{\mathrm{L}^{\infty}(\{\pm1\}^{\mathbb{L}_{\mathrm{fat}}})},
\end{align*}
which yields \eqref{eq:semigroupsI} again because {\small$\kappa,\rho>0$} are small.
\end{proof}
We are now in a position to prove Proposition \ref{prop:kv}.
\begin{proof}[Proof of Proposition \ref{prop:kv}]
To avoid confusing notation (because we shortly consider expectations with respect to probability measures on {\small$\{\pm1\}^{\mathbb{L}_{\mathrm{fat}}}$}), let us use {\small$\mathsf{E}$} to denote expectation with respect to the law of {\small$\t\mapsto\eta^{\mathbb{L}_{\mathrm{fat}}}_{\t}$} with initial data {\small$\eta^{\mathbb{L}_{\mathrm{fat}}}_{0}\sim\mathbb{P}^{0}$}. (This is instead of {\small$\E^{0}$} as in the \eqref{eq:kvI}.) Now, we compute the second moment in \eqref{eq:kvI} as follows:
\begin{align}
\mathsf{E}\Big({\int_{0}^{\tau}}\mathfrak{d}[\eta^{\mathbb{L}_{\mathrm{fat}}}_{\s}]\d\s\Big)^{2}&=\int_{0}^{\tau}\int_{0}^{\tau}\mathsf{E}\mathfrak{d}[\eta^{\mathbb{L}_{\mathrm{fat}}}_{\s_{1}}]\mathfrak{d}[\eta^{\mathbb{L}_{\mathrm{fat}}}_{\s_{2}}]\d\s_{2}\d\s_{1}=2\int_{0}^{\tau}\int_{\s_{1}}^{\tau}\mathsf{E}\mathfrak{d}[\eta^{\mathbb{L}_{\mathrm{fat}}}_{\s_{1}}]\mathfrak{d}[\eta^{\mathbb{L}_{\mathrm{fat}}}_{\s_{2}}]\d\s_{2}\d\s_{1}\nonumber\\
&=2\int_{0}^{\tau}\int_{\s_{1}}^{\tau}\mathsf{E}\Big(\mathfrak{d}[\eta^{\mathbb{L}_{\mathrm{fat}}}_{\s_{1}}]\cdot\mathsf{E}_{\s_{1}}\mathfrak{d}[\eta^{\mathbb{L}_{\mathrm{fat}}}_{\s_{2}}]\Big)\d\s_{2}\d\s_{1}.\label{eq:kvI1}
\end{align}
In the second line, the notation {\small$\mathsf{E}_{\s_{1}}$} means expectation conditioning on the filtration at time {\small$\s_{1}$}. Because {\small$\s_{2}\geq\s_{1}$} in the above display, we know that {\small$\mathsf{E}_{\s_{1}}\mathfrak{d}[\eta^{\mathbb{L}_{\mathrm{fat}}}_{\s_{2}}]=(\mathbf{T}_{\s_{2}-\s_{1}}^{}\mathfrak{d})[\eta^{\mathbb{L}_{\mathrm{fat}}}_{\s_{1}}]$}; see \eqref{eq:semigroups} for the notation. Since {\small$\mathfrak{d}$} is uniformly bounded, we can use Lemma \ref{lemma:semigroups} to obtain that 
\begin{align}
\mathsf{E}\Big(\mathfrak{d}[\eta^{\mathbb{L}_{\mathrm{fat}}}_{\s_{1}}]\cdot\mathsf{E}_{\s_{1}}\mathfrak{d}[\eta^{\mathbb{L}_{\mathrm{fat}}}_{\s_{2}}]\Big)&=\mathsf{E}\Big(\mathfrak{d}[\eta^{\mathbb{L}_{\mathrm{fat}}}_{\s_{1}}]\cdot(\mathbf{T}^{\mathrm{symm}}_{\s_{2}-\s_{1}}\mathfrak{d})[\eta^{\mathbb{L}_{\mathrm{fat}}}_{\s_{1}}]\Big)+\mathrm{O}(N^{-\frac13}).\nonumber
\end{align}
By the same reasoning like in the previous paragraph, the expectation on the \abbr{RHS} of the previous display can be written in as the image of a {\small$\mathbf{T}_{\s_{1}}$}-semigroup operator, so that 
\begin{align}
\mathsf{E}\Big(\mathfrak{d}[\eta^{\mathbb{L}_{\mathrm{fat}}}_{\s_{1}}]\cdot(\mathbf{T}^{\mathrm{symm}}_{\s_{2}-\s_{1}}\mathfrak{d})[\eta^{\mathbb{L}_{\mathrm{fat}}}_{\s_{1}}]\Big)&=\Big\{\mathbf{T}_{\s_{1}}\Big[\mathfrak{d}\cdot(\mathbf{T}^{\mathrm{symm}}_{\s_{2}-\s_{1}}\mathfrak{d})\Big]\Big\}[\eta^{\mathbb{L}_{\mathrm{fat}}}_{0}]\nonumber\\
&=\Big\{\mathbf{T}^{\mathrm{symm}}_{\s_{1}}\Big[\mathfrak{d}\cdot(\mathbf{T}^{\mathrm{symm}}_{\s_{2}-\s_{1}}\mathfrak{d})\Big]\Big\}[\eta^{\mathbb{L}_{\mathrm{fat}}}_{0}]+\mathrm{O}(N^{-\frac13}),\nonumber
\end{align}
where the second line follows again by Lemma \ref{lemma:semigroups}. (For this, we need also that {\small$\mathbf{T}^{\mathrm{symm}}_{\s_{2}-\s_{1}}\mathfrak{d}$} is uniformly bounded. But this follows from contractivity of {\small$\mathbf{T}^{\mathrm{symm}}_{\s_{2}-\s_{1}}$} as in the proof of Lemma \ref{lemma:semigroups} and uniform boundedness of {\small$\mathfrak{d}$}.) We now combine the previous three displays. This gives us
\begin{align}
\mathsf{E}\Big(\int_{0}^{\tau}\mathfrak{d}[\eta^{\mathbb{L}_{\mathrm{fat}}}_{\s}]\d\s\Big)^{2}\lesssim2\int_{0}^{\tau}\int_{\s_{1}}^{\tau}\Big\{\mathbf{T}^{\mathrm{symm}}_{\s_{1}}\Big[\mathfrak{d}\cdot(\mathbf{T}^{\mathrm{symm}}_{\s_{2}-\s_{1}}\mathfrak{d})\Big]\Big\}[\eta^{\mathbb{L}_{\mathrm{fat}}}_{0}]\d\s_{2}\d\s_{1}+N^{-\frac13}\tau^{2}.\label{eq:kvI2}
\end{align}
Next, we let {\small$\mathsf{E}^{\mathrm{symm}}$} denote expectation with respect to the law of the process {\small$\s\mapsto\eta^{\mathbb{L}_{\mathrm{fat}},\mathrm{symm}}$} valued in {\small$\{\pm1\}^{\mathbb{L}_{\mathrm{fat}}}$} with generator {\small$\mathscr{L}_{\mathbb{L}_{\mathrm{fat}},\mathrm{symm}}$} from \eqref{eq:generatorsymm}. We assume that its initial data is the same as that of the process {\small$\t\mapsto\eta^{\mathbb{L}_{\mathrm{fat}}}_{\t}$}. By the same reasoning that gave us \eqref{eq:kvI1}, we have 
\begin{align}
2\int_{0}^{\tau}\int_{\s_{1}}^{\tau}\Big\{\mathbf{T}^{\mathrm{symm}}_{\s_{1}}\Big[\mathfrak{d}\cdot(\mathbf{T}^{\mathrm{symm}}_{\s_{2}-\s_{1}}\mathfrak{d})\Big]\Big\}[\eta^{\mathbb{L}_{\mathrm{fat}}}_{0}]\d\s_{2}\d\s_{1}=\mathsf{E}^{\mathrm{symm}}\Big(\int_{0}^{\tau}\mathfrak{d}[\eta^{\mathbb{L}_{\mathrm{fat}},\mathrm{symm}}_{\s}]\d\s\Big)^{2}.\label{eq:kvI3}
\end{align}
Thus, we are left to control the \abbr{RHS} of the previous display. The upshot here is that the process {\small$\s\mapsto\eta^{\mathbb{L}_{\mathrm{fat}},\mathrm{symm}}$} has {\small$\mathbb{P}^{0}$} as an invariant measure, so we can proceed with standard estimates, which we sketch below. In particular, by the classical Kipnis-Varadhan inequality that can be found Appendix 1.6 of \cite{KL}, we have the following general inequality:
\begin{align}
\mathsf{E}^{\mathrm{symm}}\Big(\int_{0}^{\tau}\mathfrak{d}[\eta^{\mathbb{L}_{\mathrm{fat}},\mathrm{symm}}_{\s}]\d\s\Big)^{2}\lesssim\tau\sup_{\mathfrak{f}:\{\pm1\}^{\mathbb{L}_{\mathrm{fat}}}\to\R}\Big\{2\E^{0}\mathfrak{d}\mathfrak{f}+\E^{0}\mathfrak{f}\cdot\mathscr{L}_{\mathbb{L}_{\mathrm{fat}},\mathrm{symm}}\mathfrak{f}\Big\}.\nonumber
\end{align}
Recall that {\small$\mathbb{P}^{0}$} is an invariant measure for {\small$\mathscr{L}_{\mathbb{L}_{\mathrm{fat}},\mathrm{symm}}$} from \eqref{eq:generatorsymm}. In particular, the last expectation in the above display is given by a Dirichlet form with coefficients of order {\small$N^{2}$} and with an energy coming from each swap of spins at neighboring points in {\small$\mathbb{L}_{\mathrm{fat}}$} and from a spin-flip at {\small$1$}. To be precise, by Proposition 10.1 in Appendix 1 of \cite{KL}, we have the following for some fixed constant {\small$\mathrm{c}>0$}:
\begin{align}
\E^{0}\mathfrak{f}\cdot\mathscr{L}_{\mathbb{L}_{\mathrm{fat}},\mathrm{symm}}\mathfrak{f}&\leq-\mathrm{c}N^{2}\sum_{\w,\w+1\in\mathbb{L}_{\mathrm{fat}}}\E^{0}|\mathscr{L}_{\w}\mathfrak{f}|^{2}-\mathrm{c}N^{2}\E^{0}|\mathscr{S}_{1}\mathfrak{f}|^{2}\nonumber\\
&\leq-\mathrm{c}N^{2}\sum_{\w,\w+1\in\mathbb{L}}\E^{0}|\mathscr{L}_{\w}\mathfrak{f}|^{2}-\mathrm{c}N^{2}\E^{0}|\mathscr{S}_{1}\mathfrak{f}|^{2}.\nonumber
\end{align}
The second line follows just by dropping some of the non-positive terms. If we set {\small$\overline{\mathfrak{f}}=\mathfrak{f}-\E^{0}\mathfrak{f}$}, then every term in the above display remain the same if we replace {\small$\mathfrak{f}$} by {\small$\overline{\mathfrak{f}}$}. Moreover, if we further replace {\small$\overline{\mathfrak{f}}$} by its expectation {\small$\E^{0}_{\mathbb{L}}\overline{\mathfrak{f}}$} conditioning on {\small$\eta^{\mathbb{L}_{\mathrm{fat}}}_{\w}$} for {\small$\w\in\mathbb{L}$}, then by convexity of the Dirichlet form (see Appendix 1.10 in \cite{KL}), the last line can only increase (i.e. become ``less negative"). Finally, after we make all these replacements, we can use a spectral gap (which is implied by the log-Sobolev inequality in \eqref{eq:localeqII1}). In particular, we ultimately deduce that 
\begin{align*}
-\mathrm{c}N^{2}\sum_{\w,\w+1\in\mathbb{L}}\E^{0}|\mathscr{L}_{\w}\mathfrak{f}|^{2}-\mathrm{c}N^{2}\E^{0}|\mathscr{S}_{1}\mathfrak{f}|^{2}\leq-\mathrm{c}_{2}N^{2}|\mathbb{L}|^{-2}\E^{0}|\E^{0}_{\mathbb{L}}\overline{\mathfrak{f}}|^{2}.
\end{align*}
Above, {\small$\mathrm{c}_{2}>0$} is a fixed constant. Now, since {\small$\mathfrak{d}$} depends only on spins at points in {\small$\mathbb{L}$} by assumption, and because {\small$\E^{0}\mathfrak{d}=0$}, we have {\small$2\E^{0}\mathfrak{d}\mathfrak{f}=2\E^{0}\mathfrak{d}\E^{0}_{\mathbb{L}}\overline{\mathfrak{f}}\leq \mathrm{c}_{2}^{-1}|\mathbb{L}|^{2}N^{-2}\E^{0}|\mathfrak{d}|^{2}+\mathrm{c}_{2}N^{2}|\mathbb{L}|^{-2}\E^{0}|\E^{0}_{\mathbb{L}}\mathfrak{f}|^{2}$}, where the last inequality follows by Schwarz. Thus, if we combine this estimate with the previous three displays, we arrive at 
\begin{align*}
\mathsf{E}^{\mathrm{symm}}\Big(\int_{0}^{\tau}\mathfrak{d}[\eta^{\mathbb{L}_{\mathrm{fat}},\mathrm{symm}}_{\s}]\d\s\Big)^{2}\lesssim\tau\sup_{\mathfrak{f}:\{\pm1\}^{\mathbb{L}_{\mathrm{fat}}}\to\R}\mathrm{c}_{2}^{-1}|\mathbb{L}|^{2}N^{-2}\E^{0}|\mathfrak{d}|^{2}\lesssim N^{-2}\tau,
\end{align*}
because {\small$\mathrm{c}_{2}>0$} is fixed, and {\small$|\mathbb{L}|\lesssim1$}, and {\small$\mathfrak{d}$} is uniformly bounded by assumption. Combining the previous display with \eqref{eq:kvI2}-\eqref{eq:kvI3} and dividing by {\small$\tau^{2}$} then yields the desired estimate \eqref{eq:kvI}.
\end{proof}
%
%
%
\section{Proof of Proposition \ref{prop:stoch}}\label{section:stochproof}
In this section, we will first prove \eqref{eq:stochI}, and then we will prove \eqref{eq:stochII} (the latter of which is a now-standard argument as in \cite{YPTRF}, for example).
\subsection{Proof of \eqref{eq:stochI}}
Recall the notation \eqref{eq:identify3b}-\eqref{eq:identify3e}. By Lemma \ref{lemma:moments} and uniform boundedness of each of the factors {\small$\varphi,\mathfrak{f}_{\mathrm{left}},\mathfrak{f}_{\mathrm{right}},\mathfrak{b}_{\mathrm{left}},\mathfrak{b}_{\mathrm{right}}$}, we get {\small$\E|\mathbf{R}^{N}_{\k,\t}[\varphi]-\mathbf{R}^{N}_{\k,\s}[\varphi]|^{p}\lesssim_{p}|\t-\s|^{p}$} for any {\small$p\geq1$} and {\small$\s,\t\in[0,1]$}. Therefore, by the Kolmogorov continuity criterion, we know that {\small$\t\mapsto\mathbf{R}^{N}_{\k,\t}[\varphi]$} is tight in the space {\small$\mathscr{C}([0,1],\R)$} of continuous real-valued paths. In particular, to show convergence to zero uniformly in {\small$\t\in[0,1]$} in probability, it suffices to show that {\small$\mathbf{R}^{N}_{\k,\t}[\varphi]\to0$} for each fixed {\small$\t\in[0,1]$}. 

We restrict to the following events (which we explain afterwards and which have probability {\small$1-\mathrm{o}(1)$} by Corollary \ref{corollary:tightness}):
\begin{align}
\mathscr{E}&:=\Big\{\sup_{(\t,\x)\in[0,1]\times\llbracket0,N\rrbracket}|\mathbf{Z}^{N}_{\t,\x}|\leq\boldsymbol{\Lambda}_{N}\Big\}\cap\Big\{\sup_{\t\in[0,1]}\sup_{\substack{\x,\y\in\llbracket0,N\rrbracket\\|\x-\y|\leq N^{\e}\log N}}|\mathbf{Z}^{N}_{\t,\x}-\mathbf{Z}^{N}_{\t,\y}|\leq N^{-\frac12+\frac23\e}\Big\}\label{eq:hpevent1}\\
&\cap\Big\{\sup_{\substack{\s,\t\in[0,1]\\|\t-\s|\leq N^{-\e}}}\sup_{\x\in\llbracket0,N\rrbracket}|\mathbf{Z}^{N}_{\t,\x}-\mathbf{Z}^{N}_{\s,\x}|\leq\boldsymbol{\Lambda}_{N}^{-1}\Big\}.\label{eq:hpevent2}
\end{align}
Above, {\small$\boldsymbol{\Lambda}_{N}\to\infty$} slowly, and {\small$\e>0$} is independent of {\small$N$}. Indeed, by Corollary \ref{corollary:tightness}, the function {\small$\mathbf{Z}^{N}$} is uniformly bounded in probability, and it has spatial H\"{o}lder-{\small$1/2-$} regularity on length-scales of order {\small$N$} and time-regularity on time-scales of order {\small$1$}, so {\small$\mathscr{E}$} holds with probability {\small$1-\mathrm{o}(1)$}. Now, on the event {\small$\mathscr{E}$}, we have immediately that {\small$\mathbf{R}_{\k,\t}^{N}[\varphi]\to0$} in probability for {\small$\k=3,4$} if {\small$\boldsymbol{\Lambda}_{N}\ll N^{1/2}$}. Thus, we focus on {\small$\k=1,2$}. Moreover, the arguments for {\small$\k=1,2$} are the same after reflection on {\small$\llbracket0,N\rrbracket$} that swaps {\small$0\leftrightarrow N$}, so let us focus on {\small$\k=1$}.

Recall {\small$\mathbf{R}^{N}_{1,\t}[\varphi]$} from \eqref{eq:identify3b}. We claim that for {\small$\tau=N^{-2+\rho}$} with {\small$\rho>0$} small but independent of {\small$N$}, we have
\begin{align}
\mathbf{R}^{N}_{1,\t}[\varphi]&=\int_{0}^{\t}\mathfrak{f}_{\mathrm{left}}[\eta_{\s}]\cdot\mathbf{Z}^{N}_{\s,0}\varphi_{0}\d\s=\int_{0}^{\t}\Big(\tau^{-1}{\textstyle\int_{0}^{\tau}}\mathfrak{f}_{\mathrm{left}}[\eta_{\s+\r}]\d\r\Big)\cdot\mathbf{Z}^{N}_{\s,0}\varphi_{0}\d\s+\mathrm{o}(1).\label{eq:stochI1}
\end{align}
Indeed, the error in replacing {\small$\mathfrak{f}_{\mathrm{left}}[\eta_{\s}]$} by its time-average on time-scale {\small$\tau$} is controlled by regularity of {\small$\mathbf{Z}^{N}$} in time on time-scales {\small$\r\in[0,\tau]$}. By \eqref{eq:hpevent2}, this is small, so \eqref{eq:stochI1} follows. Now, we bound the {\small$\d\s$}-integral on the far \abbr{RHS} of \eqref{eq:stochI1} using the triangle inequality. If we again use \eqref{eq:hpevent1}, then we have the inequality
\begin{align}
|\mathbf{R}^{N}_{1,\t}[\varphi]|&\lesssim\boldsymbol{\Lambda}_{N}\int_{0}^{\t}|\tau^{-1}{\textstyle\int_{0}^{\tau}}\mathfrak{f}_{\mathrm{left}}[\eta_{\s+\r}]\d\r|\d\s+\mathrm{o}(1).\label{eq:stochI2}
\end{align}
It is left to establish that the \abbr{RHS} of \eqref{eq:stochI2} vanishes in the large-{\small$N$} limit in probability. In particular, it suffices to control its expectation. The first step we take is a ``localization construction", which we explain below.
\begin{itemize}
\item Fix any {\small$\s\in[0,\t]$}. Let {\small$\mathbb{L}$} be an interval in {\small$\mathbb{K}_{N}$} containing {\small$1$} such that {\small$\mathfrak{f}_{\mathrm{left}}[\eta]$} depends only on {\small$\eta_{\w}$} for {\small$\w\in\mathbb{L}$}. Let us now recall {\small$\mathbb{L}_{\mathrm{fat}}$} from Definition \ref{definition:localization}, and let {\small$\r\mapsto\eta^{\mathbb{L}_{\mathrm{fat}}}_{\s+\r}$} be the process from Definition \ref{definition:localization} with initial data {\small$\eta^{\mathbb{L}_{\mathrm{fat}}}_{\s,\w}=\eta_{\s,\w}$} for all {\small$\w\in\mathbb{L}_{\mathrm{fat}}$}.
\item We will couple {\small$\r\mapsto\eta_{\s+\r}$} and {\small$\r\mapsto\eta^{\mathbb{L}_{\mathrm{fat}}}_{\s+\r}$} using the following basic coupling. Run {\small$\r\mapsto\eta_{\s+\r}$}, and in the {\small$\r\mapsto\eta^{\mathbb{L}_{\mathrm{fat}}}_{\s+\r}$} process, take a particle at point {\small$\x\in\mathbb{L}_{\mathrm{fat}}$}. Suppose that the speed that this particle goes from {\small$\x\mapsto\x-1$} is the same as in the {\small$\r\mapsto\eta_{\s+\r}$} process. Then, we let this particle move from {\small$\x\mapsto\x-1$} if and only if the same move happens in the {\small$\r\mapsto\eta_{\s+\r}$} process. Similarly, suppose that the speed of {\small$\x\mapsto\x+1$} in the {\small$\r\mapsto\eta^{\mathbb{L}_{\mathrm{fat}}}_{\s+\r}$} process is the same as in the {\small$\r\mapsto\eta_{\s+\r}$} process. Then, we move the particle {\small$\x\mapsto\x+1$} in the {\small$\r\mapsto\eta^{\mathbb{L}_{\mathrm{fat}}}_{\s+\r}$} process if and only if it does the same move in the {\small$\r\mapsto\eta_{\s+\r}$} process. 
\item Similarly, suppose that the speed of flipping the spin at {\small$1$} in the {\small$\r\mapsto\eta^{\mathbb{L}_{\mathrm{fat}}}_{\s+\r}$} process is the same as in the {\small$\r\mapsto\eta_{\s+\r}$} process. Then, we flip this spin if and only if we do it in the {\small$\r\mapsto\eta_{\s+\r}$} process.
\item If the speed of a move in the {\small$\r\mapsto\eta^{\mathbb{L}_{\mathrm{fat}}}_{\s+\r}$} process is \emph{not} the same as a move in the original {\small$\r\mapsto\eta_{\s+\r}$} process, then we perform this move in the {\small$\r\mapsto\eta^{\mathbb{L}_{\mathrm{fat}}}_{\s+\r}$} process according to an independent Poisson clock.
\item We let a discrepancy between the {\small$\r\mapsto\eta_{\s+\r}$} and {\small$\r\mapsto\eta^{\mathbb{L}_{\mathrm{fat}}}_{\s+\r}$} processes be a site {\small$\y\in\mathbb{L}_{\mathrm{fat}}$} where the configurations have different spin values. We now make the following observations.
\begin{enumerate}
\item The only way for moves of a particle at {\small$\x\in\mathbb{L}_{\mathrm{fat}}$} in the {\small$\r\mapsto\eta^{\mathbb{L}_{\mathrm{fat}}}_{\s+\r}$} process to have a different speed as in the {\small$\r\mapsto\eta_{\s+\r}$} process is if there is a discrepancy in a neighborhood of length {\small$\mathrm{O}(1)$} around {\small$\x$}. Indeed, the speed of jumping and flipping a particle or hole at {\small$\x$} depends only on the configuration around {\small$\x$}.
\item Discrepancies evolve according to the following dynamics. First, from the particle hopping, discrepancies travel according to simple random walks (not necessarily symmetric!) at speed {\small$\mathrm{O}(N^{2})$}. A discrepancy at {\small$\x$} can also give birth to or remove another discrepancy at a point {\small$\y$} satisfying {\small$|\x-\y|\lesssim1$}. This is because the discrepancy at {\small$\x$} can affect particle hopping or creation-removal speeds in a neighborhood around {\small$\x$} of length {\small$\mathrm{O}(1)$}.
\item Initially at time {\small$\r=0$}, by construction, there are no discrepancies in {\small$\mathbb{L}_{\mathrm{fat}}$}. Thus, the first discrepancy must appear within {\small$\mathrm{O}(1)$} of the boundary of {\small$\mathbb{L}_{\mathrm{fat}}:=\llbracket1,\max\mathbb{L}+N^{\kappa}\rrbracket$}.
\item Lastly, with probability {\small$\geq1-\mathrm{O}(\exp(-N))$}, the number of discrepancies created in {\small$\mathbb{L}_{\mathrm{fat}}$} is {\small$\lesssim N^{10}$}. This follows because discrepancies are only created by the ringing of a Poisson clock. There are {\small$\mathrm{O}(N)$} Poisson clocks with speed {\small$\mathrm{O}(N^{2})$}, so the bound on discrepancies follows from standard sub-exponential tail estimates for the Poisson distribution.
\end{enumerate}
\item We now claim that with respect to the above coupling, the probability that a discrepancy appears inside {\small$\mathbb{L}$} by time {\small$\r=\tau=N^{-2+\rho}$} is exponentially small in {\small$N$}, and thus {\small$\lesssim_{\mathrm{D}}N^{-\mathrm{D}}$} for any {\small$\mathrm{D}>0$}, if the parameter {\small$\rho$} is small enough depending on the parameter {\small$\kappa>0$} from Definition \ref{definition:localization}. Indeed, by points (2) and (3), a discrepancy in {\small$\mathbb{L}$} implies that a simple random walk of speed {\small$\mathrm{O}(N^{2})$} travels a distance of {\small$\gtrsim N^{\kappa}$} by time {\small$N^{-2+\rho}$}. In particular, this implies that Poisson random variable of intensity parameter {\small$\lesssim N^{\rho}$} exceeds {\small$\gtrsim N^{\kappa}$}. Using standard Poisson tail estimates, we deduce that this event is exponentially small in {\small$N$} if {\small$\kappa$} is a large multiple of {\small$\rho>0$}. Moreover, by point (4) above, we can assume that the total number of discrepancies is {\small$\lesssim N^{10}$}. Thus, we can use a union bound over all discrepancies to deduce that there is no discrepancy between {\small$\mathbb{L}$} by time {\small$\r=\tau$} with probability {\small$1-\mathrm{O}_{\mathrm{D}}(N^{-\mathrm{D}})$}.
\end{itemize}
Now, because {\small$\mathfrak{f}_{\mathrm{left}}[\eta]$} depends only on {\small$\eta_{\w}$} for {\small$\w\in\mathbb{L}$}, from the above localization argument, we deduce that 
\begin{align}
\int_{0}^{\t}\E\Big(|\tau^{-1}{\textstyle\int_{0}^{\tau}}\mathfrak{f}_{\mathrm{left}}[\eta_{\s+\r}]\d\r|\Big)\d\s=\int_{0}^{\t}\E\Big(|\tau^{-1}{\textstyle\int_{0}^{\tau}}\mathfrak{f}_{\mathrm{left}}[\eta^{\mathbb{L}_{\mathrm{fat}}}_{\s+\r}]\d\r|\Big)\d\s+\mathrm{O}_{\mathrm{D}}(N^{-\mathrm{D}})\label{eq:stochI3}
\end{align}
for any {\small$\mathrm{D}>0$}. Now, we rewrite the expectation on the \abbr{RHS} of \eqref{eq:stochI3} in the following fashion. This expectation is with respect to the law of the process {\small$\r\mapsto\eta^{\mathbb{L}_{\mathrm{fat}}}_{\s+\r}$}. It is the same as conditioning on {\small$\eta^{\mathbb{L}_{\mathrm{fat}}}_{\s}$}, taking expectation with respect to the law of the dynamic {\small$\r\mapsto\eta^{\mathbb{L}_{\mathrm{fat}}}_{\s+\r}$}, and then taking an expectation over the randomness of {\small$\eta^{\mathbb{L}_{\mathrm{fat}}}_{\s}$}. Thus, if we let {\small$\E^{\mathrm{dyn}}_{\eta^{\mathbb{L}_{\mathrm{fat}}}_{\s}}$} denote the expectation with respect to the law of the dynamics {\small$\r\mapsto\eta^{\mathbb{L}_{\mathrm{fat}}}_{\s+\r}$} conditioning on the initial data {\small$\eta^{\mathbb{L}_{\mathrm{fat}}}_{\s}$}, then we have 
\begin{align}
\E\Big(|\tau^{-1}{\textstyle\int_{0}^{\tau}}\mathfrak{f}_{\mathrm{left}}[\eta^{\mathbb{L}_{\mathrm{fat}}}_{\s+\r}]\d\r|\Big)&=\E\E^{\mathrm{dyn}}_{\eta^{\mathbb{L}_{\mathrm{fat}}}_{\s}}\Big(|\tau^{-1}{\textstyle\int_{0}^{\tau}}\mathfrak{f}_{\mathrm{left}}[\eta^{\mathbb{L}_{\mathrm{fat}}}_{\s+\r}]\d\r|\Big)=\E\E^{\mathrm{dyn}}_{\eta_{\s}}\Big(|\tau^{-1}{\textstyle\int_{0}^{\tau}}\mathfrak{f}_{\mathrm{left}}[\eta^{\mathbb{L}_{\mathrm{fat}}}_{\s+\r}]\d\r|\Big),\label{eq:stochI4}
\end{align}
where the last identity follows since the initial data {\small$\eta^{\mathbb{L}_{\mathrm{fat}}}_{\s}$} is the same as {\small$\eta_{\s}$} restricted to the subset {\small$\mathbb{L}_{\mathrm{fat}}$}. Now, we claim that the following estimate holds (with explanation given afterwards):
\begin{align}
\int_{0}^{\t}\E\E^{\mathrm{dyn}}_{\eta_{\s}}\Big(|\tau^{-1}{\textstyle\int_{0}^{\tau}}\mathfrak{f}_{\mathrm{left}}[\eta^{\mathbb{L}_{\mathrm{fat}}}_{\s+\r}]\d\r|\Big)\d\s&=\int_{0}^{\t}\E\E^{\mathrm{dyn}}_{\eta_{\s}}\Big(|\tau^{-1}{\textstyle\int_{0}^{\tau}}\mathfrak{f}_{\mathrm{left}}[\eta^{\mathbb{L}_{\mathrm{fat}}}_{\r}]\d\r|\Big)\d\s\nonumber\\
&\lesssim\E^{0}\Big(|\tau^{-1}{\textstyle\int_{0}^{\tau}}\mathfrak{f}_{\mathrm{left}}[\eta^{\mathbb{L}_{\mathrm{fat}}}_{\r}]\d\r|\Big)+N^{-\frac14}|\mathbb{L}_{\mathrm{fat}}|\nonumber\\
&\lesssim_{}N^{-\frac12\rho}+N^{-\frac14+\kappa}.\label{eq:stochI5}
\end{align}
The first identity follows since the law of the process {\small$\r\mapsto\eta^{\mathbb{L}_{\mathrm{fat}}}_{\s+\r}$} is time-homogeneous, so once we condition on the initial data as a function of {\small$\eta_{\s}$}, we can start the dynamics at time {\small$0$} instead of time {\small$\s$}. The second line holds by Lemma \ref{lemma:localeq}. Indeed, the inner expectation on the \abbr{RHS} of the first line depends on {\small$\eta_{\s}$} through its restriction to {\small$\mathbb{L}_{\mathrm{fat}}$}. We note that in the second line, the expectation is with respect to {\small$\r\mapsto\eta^{\mathbb{L}_{\mathrm{fat}}}_{\r}$} with initial data given by {\small$\eta^{\mathbb{L}_{\mathrm{fat}}}_{\w}$} being i.i.d. mean-zero random variables (as with Proposition \ref{prop:kv}). The last line follows by Proposition \ref{prop:kv} and {\small$|\mathbb{L}_{\mathrm{fat}}|=|\llbracket1,\max\mathbb{L}+N^{\kappa}\rrbracket|\lesssim N^{\kappa}$}. We now combine \eqref{eq:stochI2}, \eqref{eq:stochI3}, \eqref{eq:stochI4}, and \eqref{eq:stochI5} with {\small$\kappa>0$} small and {\small$\rho>0$} small (depending only on {\small$\kappa$}). This implies that {\small$\mathbf{R}^{N}_{1,\t}[\varphi]\to0$} in probability. \qed
\subsection{Proof of \eqref{eq:stochII}}
By using the Kolmogorov continuity argument as in the beginning of the proof of \eqref{eq:stochI}, we reduce to showing that the integral on the \abbr{LHS} of \eqref{eq:stochII} vanishes pointwise in {\small$\t$} in probability, i.e. that 
\begin{align}
\lim_{N\to\infty}\int_{0}^{\t}\tfrac{1}{N}\sum_{\x\in\llbracket0,N\rrbracket}\mathfrak{a}_{\x}[\eta_{\s}]|\mathbf{Z}^{N}_{\s,\x}|^{2}\cdot\wt{\varphi}_{N^{-1}\x}\d\s=0.\label{eq:stochII0}
\end{align}
Take the \abbr{LHS} of \eqref{eq:stochII0}. For any {\small$\x\in\llbracket0,N\rrbracket$}, consider a neighborhood {\small$\mathbb{I}[\x]$} of length {\small$N^{\e}$}, where {\small$\e>0$} is the parameter from \eqref{eq:hpevent1}. We now define the following ``canonical ensemble expectation":
\begin{align}
\boldsymbol{\Psi}_{\mathfrak{a}_{\x},\x}[\eta]:=\E^{0}\Big(\mathfrak{a}_{\x}[\wt{\eta}]\Big|\tfrac{1}{\mathbb{I}[\x]}\sum_{\y\in\mathbb{I}[\x]}\wt{\eta}_{\y}=\tfrac{1}{\mathbb{I}[\x]}\sum_{\y\in\mathbb{I}[\x]}{\eta}_{\y}\Big).\label{eq:stochII1}
\end{align}
Let us now clarify the notation above. The variable {\small$\wt{\eta}$} is the expectation variable (the dummy integration variable). We are conditioning on the {\small$\wt{\eta}$}-density on {\small$\mathbb{I}[\x]$} to equal the {\small$\eta$}-density on {\small$\mathbb{I}[\x]$}. In particular, the previous display is a function of {\small$\eta_{\w}$} for {\small$\w\in\mathbb{I}[\x]$}. Now, we write
\begin{align}
\int_{0}^{\t}\tfrac{1}{N}\sum_{\x\in\llbracket0,N\rrbracket}\mathfrak{a}_{\x}[\eta_{\s}]|\mathbf{Z}^{N}_{\s,\x}|^{2}\cdot\wt{\varphi}_{N^{-1}\x}\d\s&=\int_{0}^{\t}\tfrac{1}{N}\sum_{\x\in\llbracket0,N\rrbracket}\boldsymbol{\Psi}_{\mathfrak{a}_{\x},\x}[\eta_{\s}]|\mathbf{Z}^{N}_{\s,\x}|^{2}\cdot\wt{\varphi}_{N^{-1}\x}\d\s\label{eq:stochII2a}\\
&+\int_{0}^{\t}\tfrac{1}{N}\sum_{\x\in\llbracket0,N\rrbracket}\Big(\mathfrak{a}_{\x}[\eta_{\s}]-\boldsymbol{\Psi}_{\mathfrak{a}_{\x},\x}[\eta_{\s}]\Big)\cdot|\mathbf{Z}^{N}_{\s,\x}|^{2}\cdot\wt{\varphi}_{N^{-1}\x}\d\s.\label{eq:stochII2b}
\end{align}
We first estimate the \abbr{RHS} of \eqref{eq:stochII2a}. First, let {\small$\sigma_{\x}[\eta]:=|\mathbb{I}[\x]|^{-1}\sum_{\y\in\mathbb{I}[\x]}\eta_{\y}$}. By Proposition 8 in \cite{GJ15}, we know that {\small$|\boldsymbol{\Psi}_{\mathfrak{a}_{\x},\x}[\eta_{\s}]|\lesssim|\E^{\sigma_{\x}[\eta_{\s}]}\mathfrak{a}_{\x}|+|\mathbb{I}[\x]|^{-1}$}. Note that {\small$\mathfrak{a}_{\x}[\eta]$} is a polynomial in {\small$\eta_{\w}$} for {\small$\w$} in a set of size {\small$\mathrm{O}(1)$}; this follows by locality of {\small$\mathfrak{a}_{\x}$} as assumed in Proposition \ref{prop:stoch}. Thus, {\small$\E^{\sigma}\mathfrak{a}_{\x}$} is Lipschitz in {\small$\sigma\in[-1,1]$}, and thus {\small$|\boldsymbol{\Psi}_{\mathfrak{a}_{\x},\x}[\eta_{\s}]|\lesssim|\mathbb{I}[\x]|^{-1}+|\sigma_{\x}[\eta_{\s}]|\lesssim N^{-\e}+|\sigma_{\x}[\eta_{\s}]|$}.

Now, consider the following decomposition for the \abbr{RHS} of \eqref{eq:stochII2a}, where {\small$\boldsymbol{\Gamma}_{N}\gg1$} is another slowly diverging constant to be chosen shortly:
\begin{align}
\int_{0}^{\t}\tfrac{1}{N}\sum_{\x\in\llbracket0,N\rrbracket}\boldsymbol{\Psi}_{\mathfrak{a}_{\x},\x}[\eta_{\s}]|\mathbf{Z}^{N}_{\s,\x}|^{2}\cdot\wt{\varphi}_{N^{-1}\x}\d\s&=\int_{0}^{\t}\tfrac{1}{N}\sum_{\x\in\llbracket0,N\rrbracket}\boldsymbol{\Psi}_{\mathfrak{a}_{\x},\x}[\eta_{\s}]\mathbf{1}_{|\mathbf{Z}^{N}_{\s,\x}|\leq\boldsymbol{\Gamma}_{N}^{-1}}|\mathbf{Z}^{N}_{\s,\x}|^{2}\cdot\wt{\varphi}_{N^{-1}\x}\d\s\nonumber\\
&+\int_{0}^{\t}\tfrac{1}{N}\sum_{\x\in\llbracket0,N\rrbracket}\boldsymbol{\Psi}_{\mathfrak{a}_{\x},\x}[\eta_{\s}]\mathbf{1}_{|\mathbf{Z}^{N}_{\s,\x}|>\boldsymbol{\Gamma}_{N}^{-1}}|\mathbf{Z}^{N}_{\s,\x}|^{2}\cdot\wt{\varphi}_{N^{-1}\x}\d\s.\nonumber
\end{align}
By \eqref{eq:hpevent1}, if we choose {\small$\boldsymbol{\Gamma}_{N}\gg\boldsymbol{\Lambda}_{N}$}, then the first term on the \abbr{RHS} of the above vanishes in the large-{\small$N$} limit. For the second term, we claim that the following estimate holds if {\small$|\mathbf{Z}^{N}_{\s,\x}|>\boldsymbol{\Gamma}_{N}^{-1}$}:
\begin{align*}
|\sigma_{\x}[\eta_{\s}]|&=N^{\frac12}|\mathbb{I}[\x]|^{-1}|\mathbf{h}^{N}_{\s,\max\mathbb{I}[\x]}-\mathbf{h}^{N}_{\s,\min\mathbb{I}[\x]-1}|=N^{\frac12}|\mathbb{I}[\x]|^{-1}|\log\mathbf{Z}^{N}_{\s,\max\mathbb{I}[\x]}-\log\mathbf{Z}^{N}_{\s,\min\mathbb{I}[\x]-1}|\\
&\lesssim\boldsymbol{\Gamma}_{N}N^{\frac12}N^{-\e}|\mathbf{Z}^{N}_{\s,\max\mathbb{I}[\x]}-\mathbf{Z}^{N}_{\s,\min\mathbb{I}[\x]-1}|\lesssim N^{-\frac13\e}\boldsymbol{\Gamma}_{N}.
\end{align*}
The first line follows by definitions \eqref{eq:hf} and \eqref{eq:ch}. For the second line, we note that if {\small$|\mathbf{Z}^{N}_{\s,\x}|>\boldsymbol{\Gamma}_{N}^{-1}$}, then by the spatial regularity in \eqref{eq:hpevent1}, we know that {\small$|\mathbf{Z}^{N}_{\s,\y}|\gtrsim\boldsymbol{\Gamma}_{N}^{-1}$} for any {\small$|\y-\x|\lesssim N^{\e}$}. So, the first bound follows because {\small$|\mathbb{I}[\x]|\lesssim N^{\e}$} and by a Lipschitz bound on the {\small$\log$} restricted to {\small$[\mathrm{c}\boldsymbol{\Gamma}_{N}^{-1},\infty)$} for some {\small$\mathrm{c}>0$} fixed. The last bound in the second line follows by the spatial regularity in \eqref{eq:hpevent1} once more. Therefore, we ultimately deduce that 
\begin{align}
\lim_{N\to\infty}\int_{0}^{\t}\tfrac{1}{N}\sum_{\x\in\llbracket0,N\rrbracket}\boldsymbol{\Psi}_{\mathfrak{a}_{\x},\x}[\eta_{\s}]|\mathbf{Z}^{N}_{\s,\x}|^{2}\cdot\wt{\varphi}_{N^{-1}\x}\d\s=0\label{eq:stochII3}
\end{align}
in probability. Let us now treat the term \eqref{eq:stochII2b}. To this end, we first write the summation over {\small$\x\in\llbracket0,N\rrbracket$} as a summation over {\small$\x\in\mathbb{I}_{N}$} and {\small$\x\not\in\mathbb{I}_{N}$}, where {\small$\mathbb{I}_{N}:=\llbracket N^{1/2},N-N^{1/2}\rrbracket$} is a suitable ``interior". Because {\small$|\mathbb{I}_{N}^{\mathrm{C}}|\lesssim N^{1/2}$}, we obtain
\begin{align}
\int_{0}^{\t}\tfrac{1}{N}\sum_{\x\not\in\mathbb{I}_{N}}\Big(\mathfrak{a}_{\x}[\eta_{\s}]-\boldsymbol{\Psi}_{\mathfrak{a}_{\x},\x}[\eta_{\s}]\Big)\cdot|\mathbf{Z}^{N}_{\s,\x}|^{2}\cdot\wt{\varphi}_{N^{-1}\x}\d\s\lesssim N^{-\frac12}\boldsymbol{\Lambda}_{N}^{2}\ll1\label{eq:stochII4}
\end{align}
with high probability. Now, on {\small$\mathbb{I}[\x]$}, we can use spatial regularity of {\small$\mathbf{Z}^{N}$} (from \eqref{eq:hpevent1}) and of {\small$\wt{\varphi}$} to replace the difference {\small$\mathfrak{a}_{\x}[\eta_{\s}]-\boldsymbol{\Psi}_{\mathfrak{a}_{\x},\x}[\eta_{\s}]$} by its spatial-average on a block of length {\small$\delta N^{\e}\log N$} for some small {\small$\delta>0$} independent of {\small$N$}. In particular, with high probability, we have the following spatial analogue to \eqref{eq:stochI1} for {\small$\mathfrak{l}=\lfloor\delta N^{\e}\log N\rfloor$}:
\begin{align}
&\Big|\int_{0}^{\t}\tfrac{1}{N}\sum_{\x\in\mathbb{I}_{N}}\Big(\mathfrak{a}_{\x}[\eta_{\s}]-\boldsymbol{\Psi}_{\mathfrak{a}_{\x},\x}[\eta_{\s}]\Big)\cdot|\mathbf{Z}^{N}_{\s,\x}|^{2}\wt{\varphi}_{N^{-1}\x}\Big|\nonumber\\
&\lesssim\boldsymbol{\Lambda}_{N}^{2}\int_{0}^{\t}\tfrac{1}{N}\sum_{\x\in\mathbb{I}_{N}}\Big|\mathfrak{l}^{-1}\sum_{\w=1,\ldots,\mathfrak{l}}\Big(\mathfrak{a}_{\x+\w}[\eta_{\s}]-\boldsymbol{\Psi}_{\mathfrak{a}_{\x+\w},\x+\w}[\eta_{\s}]\Big)\Big|\d\s+\mathrm{o}(1).\label{eq:stochII5}
\end{align}
We clarify that without the restriction to {\small$\mathbb{I}_{N}$}, the block averages in the second line are not well-defined. We now control the expectation of the first term in the second line of \eqref{eq:stochII5}. Denote the term in absolute values therein by {\small$\mathfrak{f}_{\x}[\eta_{\s}]$}. Note that it depends only on {\small$\eta_{\s,\w}$} for {\small$\w\in\mathbb{L}_{\x}$}, where {\small$\mathbb{L}_{\x}$} is an interval of length {\small$\lesssim N^{2\e}$}. We now apply Lemma \ref{lemma:localeq}. This shows that 
\begin{align}
&\int_{0}^{\t}\tfrac{1}{N}\sum_{\x\in\mathbb{I}_{N}}\Big|\mathfrak{l}^{-1}\sum_{\w=1,\ldots,\mathfrak{l}}\Big(\mathfrak{a}_{\x+\w}[\eta_{\s}]-\boldsymbol{\Psi}_{\mathfrak{a}_{\x+\w},\x+\w}[\eta_{\s}]\Big)\Big|\d\s\nonumber\\
&\lesssim\sup_{\x\in\mathbb{I}_{N}}\sup_{\sigma\in[-1,1]}\E^{\sigma,\mathbb{L}_{\x}}\Big|\mathfrak{l}^{-1}\sum_{\w=1,\ldots,\mathfrak{l}}\Big(\mathfrak{a}_{\x+\w}[\eta]-\boldsymbol{\Psi}_{\mathfrak{a}_{\x+\w},\x+\w}[\eta]\Big)\Big|+N^{-\frac32}N^{6\e}.\nonumber
\end{align}
Using (4.40) in \cite{DT} (see (4.24) therein for the relevant notation), we deduce that the first term from the second line vanishes in the large-{\small$N$} limit. (Although (4.40) in \cite{DT} takes {\small$\mathfrak{l}\to\infty$} at no prescribed rate, the proof gives a quantitative estimate on the expectations in the previous display in terms of a quantity that decays as long as {\small$\mathfrak{l}\gg|\mathbb{I}[\x]|$}, which is the case here. We clarify that the condition {\small$\mathfrak{l}\gg|\mathbb{I}[\x]|$} means that the averaging scale is much longer than the support length of {\small$\mathfrak{a}_{\x}-\boldsymbol{\Psi}_{\mathfrak{a}_{\x},\x}$}, so that averaging on scale {\small$\mathfrak{l}$} introduces cancellations.) Therefore, the previous display is {\small$\mathrm{o}(1)$}. Plugging this into \eqref{eq:stochII5}, choosing {\small$\boldsymbol{\Lambda}_{N}\to\infty$} sufficiently slowly, and combining this with \eqref{eq:stochII4} implies that the term in \eqref{eq:stochII2b} vanishes in probability in the large-{\small$N$} limit. Combining this further with \eqref{eq:stochII3} and \eqref{eq:stochII2a}-\eqref{eq:stochII2b} implies the desired claim \eqref{eq:stochII0}. \qed
%
%
%
\section{Proof of Corollary \ref{corollary:stat}}\label{section:corollary}
We choose the initial data {\small$\mathbf{h}^{N}_{0,\cdot}$} by using {\small$\mathbf{h}^{N}_{0,0}=0$} and then by sampling the increment process {\small$\x\mapsto\mathbf{h}^{N}_{0,\x}-\mathbf{h}^{N}_{0,0}=\mathbf{h}^{N}_{0,\x}$} via {\small$\boldsymbol{\pi}^{N,\alpha,\beta,\gamma,\delta}$}. We claim that for this choice of initial data, the following properties hold.
\begin{enumerate}
\item The spatially-rescaled initial data {\small$\X\mapsto\mathbf{h}^{N}_{0,N\X}$} is tight in the large-{\small$N$} limit in {\small$\mathscr{C}([0,1])$}, where {\small$\mathbf{h}^{N}_{0,N\X}$} extends from {\small$\X\in\{0,\frac{1}{N},\ldots,1\}$} to {\small$\X\in[0,1]$} by linear interpolation.
\item For this choice of initial data, the estimates in \eqref{eq:mainap} hold (where {\small$\mathbf{Z}^{N}_{0,\cdot}$} therein is defined via \eqref{eq:ch}).
\end{enumerate}
Assuming that these two claims hold, Corollary \ref{corollary:stat} follows readily. Indeed, Theorem \ref{theorem:main} gives the convergence of {\small$\mathbf{h}^{N}$} (with this stationary initial data) to the solution of \eqref{eq:openkpz}-\eqref{eq:openkpzII}, possibly along a subsequence as {\small$N\to\infty$}. Moreover, by construction, the increment process of the corresponding sub-sequential limit of {\small$\mathbf{h}^{N}_{0,\cdot}$} has law given by an invariant measure of the increment process of \eqref{eq:openkpz}-\eqref{eq:openkpzII}. But the increment process of \eqref{eq:openkpz}-\eqref{eq:openkpzII} has a \emph{unique} invariant measure \cite{KM,P23}. Thus, any subsequence of {\small$\Gamma^{N}_{\ast}\boldsymbol{\pi}^{N,\alpha,\beta,\gamma,\delta}$} has a subsequence that converges weakly to the unique invariant measure for the increment process of \eqref{eq:openkpz}-\eqref{eq:openkpzII}. This completes the proof.

We are left to show points (1) and (2) above. By assumption, there exists a constant {\small$\Lambda>0$} such that {\small$|\alpha[\eta]|+|\beta[\eta]|+|\gamma[\eta]|+|\delta[\eta]|\leq\Lambda$} for all {\small$\eta$}. Consider the probability measures {\small$\boldsymbol{\pi}^{N,\mathrm{upper}}:=\boldsymbol{\pi}^{N,\Lambda,-\Lambda,-\Lambda,\Lambda}$} and {\small$\boldsymbol{\pi}^{N,\mathrm{lower}}:=\boldsymbol{\pi}^{N,-\Lambda,\Lambda,\Lambda,-\Lambda}$}; these are invariant measures for the increment process of {\small$\mathbf{h}^{N}$} upon replacing {\small$(\alpha[\eta],\beta[\eta],\gamma[\eta],\delta[\eta])$} in \eqref{eq:generatorIc}-\eqref{eq:generatorId} by {\small$(\Lambda,-\Lambda,-\Lambda,\Lambda)$} and {\small$(-\Lambda,\Lambda,\Lambda,-\Lambda)$}, respectively. In words, {\small$\boldsymbol{\pi}^{N,\mathrm{upper}}$} is the invariant measure for the increment process after we increase the speed of inserting particles and decrease the speed of removing particles, and {\small$\boldsymbol{\pi}^{N,\mathrm{lower}}$} is the same but decreasing the insertion-speed and increasing the removal-speed, with all changes in speeds being {\small$\mathrm{O}(N^{3/2})$}. 

We now claim that there exists a coupling between {\small$\boldsymbol{\pi}^{N,\mathrm{lower}},\boldsymbol{\pi}^{N,\alpha,\beta,\gamma,\delta},\boldsymbol{\pi}^{N,\mathrm{upper}}$} under which
\begin{align}
\mathbf{h}^{N,\mathrm{lower}}_{\x}-\mathbf{h}^{N,\mathrm{lower}}_{0} \leq \mathbf{h}^{N}_{\x}-\mathbf{h}^{N}_{0} \leq \mathbf{h}^{N,\mathrm{upper}}_{\x}-\mathbf{h}^{N,\mathrm{upper}}_{0}\label{eq:stat0}
\end{align}
for all {\small$\x\in\mathbb{K}_{N}$} with probability {\small$1$}, where {\small$\mathbf{h}^{N,\mathrm{lower}}_{\cdot}-\mathbf{h}^{N,\mathrm{lower}}_{0}\sim\boldsymbol{\pi}^{N,\mathrm{lower}}$}, where {\small$\mathbf{h}^{N}_{\cdot}-\mathbf{h}^{N}_{0}\sim\boldsymbol{\pi}^{N,\alpha,\beta,\gamma,\delta}$}, and where {\small$\mathbf{h}^{N,\mathrm{upper}}_{\cdot}-\mathbf{h}^{N,\mathrm{upper}}_{0}\sim\boldsymbol{\pi}^{N,\mathrm{upper}}$}. Equivalently, if {\small$\eta^{\mathrm{lower}}$} is the configuration associated to {\small$\mathbf{h}^{N,\mathrm{lower}}_{\cdot}-\mathbf{h}^{N,\mathrm{lower}}_{0}\sim\boldsymbol{\pi}^{N,\mathrm{lower}}$}, if {\small$\eta$} is the configuration associated to {\small$\mathbf{h}^{N}_{\cdot}-\mathbf{h}^{N,\mathrm{lower}}_{0}\sim\boldsymbol{\pi}^{N,\alpha,\beta,\gamma,\delta}$}, and if {\small$\eta^{\mathrm{upper}}$} is the configuration associated to {\small$\mathbf{h}^{N,\mathrm{upper}}_{\cdot}-\mathbf{h}^{N,\mathrm{upper}}_{0}\sim\boldsymbol{\pi}^{N,\mathrm{upper}}$}, then we must find a coupling under which
\begin{align}
\eta^{\mathrm{lower}}_{\x}\leq\eta_{\x}\leq\eta^{\mathrm{upper}}_{\x} \quad\text{for all } \x\in\mathbb{K}_{N}.\label{eq:stat1}
\end{align}
Constructing a coupling under which \eqref{eq:stat1} holds follows the same argument as in Lemma 4.1 of \cite{CK}. In particular, we will fix any arbitrary common configuration {\small$\eta^{\mathrm{common}}_{0,\cdot}$} from which we run three different particle dynamics, whose invariant measures are given by {\small$\boldsymbol{\pi}^{N,\mathrm{lower}},\boldsymbol{\pi}^{N,\alpha,\beta,\gamma,\delta},\boldsymbol{\pi}^{N,\mathrm{upper}}$}, that are denoted by {\small$\eta^{\mathrm{lower}}_{\t,\cdot},\eta_{\t,\cdot},\eta^{\mathrm{upper}}_{\t,\cdot}$}, and that are coupled in the following way.
\begin{itemize}
\item The particle jumps are given the standard basic coupling, so the jump clocks are the same in all three ``species" (see Section 4 of \cite{CK}).
\item If a particle is created in the {\small$\eta^{\mathrm{lower}}$} system, then it is also created in the {\small$\eta,\eta^{\mathrm{upper}}$} systems. If a particle is created in the {\small$\eta$} system, then it is also created in the {\small$\eta^{\mathrm{upper}}$} system. We can do this because for any configuration, the speeds of inserting particles in the three different species are ordered appropriately.
\item If a particle is removed in the {\small$\eta^{\mathrm{upper}}$} system, then it is also removed in the {\small$\eta,\eta^{\mathrm{lower}}$} systems. If a particle is removed in the {\small$\eta$} system, then it is also created in the {\small$\eta^{\mathrm{lower}}$} system. This is achievable, again, because the removal-speeds are ordered appropriately.
\end{itemize}
By construction, we know that 
\begin{align}
\eta^{\mathrm{lower}}_{\t,\x}\leq\eta_{\t,\x}\leq\eta^{\mathrm{upper}}_{\t,\x} \quad\text{for } \t=0 \ \text{and} \ \text{for all } \x\in\mathbb{K}_{N}.\label{eq:stat2}
\end{align}
It is standard that the basic coupling of particle jumps propgates the ordering \eqref{eq:stat2} in time (again, see Section 4 of \cite{CK}). It is also by construction that the coupling of inserting and removing particles given above preserves the ordering \eqref{eq:stat2} for later times. (To break the first inequality at some time {\small$\t$} and {\small$\x=1$}, we need to create a particle at {\small$\x=1$} in the {\small$\eta^{\mathrm{lower}}$} system without doing so in the {\small$\eta$} system; this cannot happen by construction.) We now deduce \eqref{eq:stat1} (and thus \eqref{eq:stat0}) by taking \eqref{eq:stat2}, sending {\small$\t\to\infty$}, and using ergodicity of all particle dynamics at hand (the state space is finite and irreducible for each dynamic).

To conclude the proof of points (1) and (2) from the beginning of this argument, and thus the proof of Corollary \ref{corollary:stat} entirely, we note that by \eqref{eq:stat0}, it suffices to show points (1) and (2) hold for {\small$\boldsymbol{\pi}^{N,\mathrm{lower}},\boldsymbol{\pi}^{N,\mathrm{upper}}$}. This follows from Proposition 4.2 in \cite{CK}. \qed
%
%
%
\section{The half-space case -- proof of Theorem \ref{theorem:mainhalf}}\label{section:half}
We would like to essentially follow the proof of Theorem \ref{theorem:main} verbatim. While this is almost doable, there is one preliminary reduction that we must make, which is to ``localize" the initial configuration {\small$\eta^{\half}_{0}$} in some sense. Otherwise, without this step, the relative entropy of the initial data with respect to product measures on {\small$\{\pm1\}^{\Z_{>0}}$} can be infinite. (This would obstruct a lot of the analysis in Section \ref{section:stochestimates}, for example.) But, we only care about sets of size {\small$\lesssim N$} near the origin, and because the dynamics of the particle system are local, we can essentially ``randomize" the initial configuration far from the origin without changing the behavior of the process near the origin.

Otherwise, the strategy is the same as in the proof of Theorem \ref{theorem:main}. We first derive an evolution equation for {\small$\mathbf{Z}^{N,\half}$}. Then, we get moment estimates to deduce tightness, and we reduce the identification of limit points to an analogue of Proposition \ref{prop:stoch}. We conclude by proving said analogue.

Because the strategy (and essentially all the estimates) is the exact same, we only sketch the argument (with details for the estimates that \emph{are} different). First, however, let us declare that {\small$\mathbb{P}^{0}$} will now denote a product measure on {\small$\{\pm1\}^{\Z_{>0}}$}, such that under this measure, the {\small$\eta^{\half}_{\w}$} are i.i.d. mean-zero random variables.
\subsection{A preliminary reduction}
We now introduce the following ``cutoff" of the distribution of the initial configuration {\small$\eta^{\half}_{0,\cdot}$}.
\begin{itemize}
\item Let {\small$\mathbb{P}_{0}$} denote the probability measure on {\small$\{\pm1\}^{\Z_{>0}}$} for the distribution of {\small$\eta^{\half}_{0,\cdot}$}. 
\item Fix an interval {\small$\mathbb{J}_{N}:=\llbracket1,N^{3/2+\delta}\rrbracket$} with {\small$\delta>0$} small but fixed. 
\item Let {\small$\Pi^{\mathbb{J}_{N}}\mathbb{P}_{0}$} be the pushforward of the initial measure {\small$\mathbb{P}_{0}$} under the projection map {\small$\{\pm1\}^{\Z_{>0}}\to\{\pm1\}^{\mathbb{J}_{N}}$}. We then let {\small$\mathbb{P}_{0}^{\mathrm{cut}}:=\Pi^{\mathbb{J}_{N}}\mathbb{P}_{0}\otimes\Pi^{\mathbb{J}_{N}^{\mathrm{C}}}\mathbb{P}^{0}$}, where {\small$\Pi^{\mathbb{J}_{N}^{\mathrm{C}}}\mathbb{P}^{0}$} is the restriction of the product measure {\small$\mathbb{P}^{0}$} to {\small$\{\pm1\}^{\mathbb{J}_{N}^{\mathrm{C}}}$} with {\small$\mathbb{J}_{N}:=\Z_{>0}\setminus\mathbb{J}_{N}$}.
\item Next, for any {\small$\t\geq0$}, we let {\small$\mathbb{P}_{\t}^{\mathrm{cut}}$} denote the law of {\small$\eta_{\t}^{\half}$} assuming its initial data is sampled via {\small$\mathbb{P}_{0}^{\mathrm{cut}}$}. Finally, we let {\small$\mathfrak{P}_{\t}^{\mathrm{cut}}$} be the density of {\small$\mathbb{P}_{\t}^{\mathrm{cut}}$} with respect to the product measure {\small$\mathbb{P}^{0}$}. It will also be convenient to let {\small$\mathfrak{P}_{\t}$} denote the density of {\small$\mathbb{P}_{\t}$} with respect to {\small$\mathbb{P}^{0}$}.
\end{itemize}

Put into words, {\small$\mathfrak{P}_{\t}^{\mathrm{cut}}\d\mathbb{P}^{0}$} denotes the law of {\small$\eta^{\half}_{\t}$} with an initial data which is given by restricting the ``original" initial data sampled via {\small$\mathfrak{P}_{0}\d\mathbb{P}^{0}$} to {\small$\mathbb{J}_{N}$}, and then independently sampling the rest of the configuration on {\small$\Z_{>0}\setminus\mathbb{J}_{N}$} by sampling the spin at each point in {\small$\Z_{>0}\setminus\mathbb{J}_{N}$} independently and mean-zero with respect to {\small$\mathbb{P}^{0}$}.

The goal of this step is to prove the following, which, as a consequence, would let us work with the process whose initial configuration is sampled via the cutoff initial measure {\small$\mathfrak{P}^{\mathrm{cut}}_{0}\d\mathbb{P}^{0}$}.
\begin{lemma}\label{lemma:cutoff}
\fsp There exists a Markov process {\small$\t\mapsto\eta^{\half,\mathrm{cut}}_{\t}$} with state space {\small$\{\pm1\}^{\Z_{>0}}$} such that the following are satisfied:
\begin{enumerate}
\item The process {\small$\t\mapsto\eta^{\half,\mathrm{cut}}_{\t}$} has the same generator \eqref{eq:generatorhalf} as {\small$\eta^{\half}_{\t}$}, and its initial data is {\small$\eta^{\half,\mathrm{cut}}_{0,\cdot}\sim\mathfrak{P}^{\mathrm{cut}}_{0}\d\mathbb{P}^{0}$}.
\item The processes {\small$\t\mapsto\eta^{\half,\mathrm{cut}}_{\t}$} and {\small$\t\mapsto\eta^{\half}_{\t}$} are coupled so that for any {\small$\mathrm{D}>0$}, we have
\begin{align}
\mathbb{P}\Big(\sup_{\t\in[0,1]}\sup_{|\w|\lesssim N\log N}|\eta^{\half}_{\t,\w}-\eta^{\half,\mathrm{cut}}_{\t,\w}|\neq0\Big)\lesssim_{\mathrm{D},\delta}N^{-\mathrm{D}}.\label{eq:cutoffI}
\end{align}
\end{enumerate}
\end{lemma}
\begin{remark}\label{remark:cut}
\fsp Given Lemma \ref{lemma:cutoff}, we will, for the rest of this section, assume the initial data {\small$\eta^{\half}_{0,\cdot}$} is sampled via {\small$\mathfrak{P}^{\mathrm{cut}}_{0}\d\mathbb{P}^{0}$}. In particular, we will not explicitly write {\small$\eta^{\half,\mathrm{cut}}$} to avoid cluttering notation, and we will explicitly mention when this choice of initial data is important (by referring to this remark).
\end{remark}
\begin{proof}
The argument is essentially a refined version of the coupling from the proof of \eqref{eq:stochI}. We describe it as follows.
\begin{itemize}
\item The dynamics of {\small$\t\mapsto\eta^{\half,\mathrm{cut}}_{\t}$} can be described as follows. First, for every {\small$\x,\x+1$} bond in {\small$\Z_{>0}$}, there exists a clock of speed {\small$N^{2}/2+\mathrm{O}(N^{3/2})$} which swaps the spin values {\small$\eta_{\x},\eta_{\x+1}$}, where the {\small$\mathrm{O}(N^{3/2})$} depends on the configuration in a neighborhood (of size {\small$\mathrm{O}(1)$}) around {\small$\x\in\Z$}. (This comes from the asymmetry in the particle random walks.) At the site {\small$1$}, there exists a Poisson clock of speed {\small$N^{2}/4+\mathrm{O}(N^{3/2})$} which flips the spin-value at {\small$1$}. Again, the {\small$\mathrm{O}(N^{3/2})$} depends on the spin-value at {\small$1$}.
\item We now construct {\small$\t\mapsto\eta^{\half,\mathrm{cut}}_{\t}$} as follows. Take any {\small$\x,\x+1$} bond in {\small$\Z_{>0}$}. We couple the Poisson clock at this bond to the corresponding clock in {\small$\t\mapsto\eta^{\half}_{\t}$} so that the probability that one clock rings and the other does not is {\small$\mathrm{O}(N^{-1/2})$}. We can do this because the order {\small$N^{2}$} part of the clock speed is independent of the particle configuration. At the boundary point {\small$1$}, we couple the spin-flip clocks in both processes to ring at the same time if the spin-values are equal at the point {\small$1$}. Otherwise, if said spin-values are distinct, we run the clocks independently.
\item We define a discrepancy between these two processes to be a point in {\small$\Z_{>0}$} such that the spin values in the two processes at the point are different. By construction, there exists a coupling between initial configurations such that all initial discrepancies are to the right of {\small$N^{3/2+\delta}$}. Every discrepancy evolves as a simple random walk, where the symmetric part of the step distribution has speed {\small$\frac12N^{2}$}, and the asymmetric part has speed {\small$\mathrm{O}(N^{3/2})$}. Therefore, for a discrepancy to ever reach the domain {\small$\{\w:|\w|\lesssim N\log N\}$}, we must have the aforementioned simple random walk travel a distance of {\small$\gtrsim N^{3/2+\delta}$}. By tail bounds for the Poisson distribution and sub-Gaussian bounds for the symmetric part of the simple random walk, this happens with probability {\small$\lesssim\exp(-N^{\e})$} for some {\small$\e>0$} depending on {\small$\delta$}.
\item We also note that a discrepancy, if created, must happen next to an existing discrepancy since the speed of particle hopping depends only on the configuration at neighboring points. The symmetric part of the speed of the particle hops is independent of the particle configuration, so the speed of creating a discrepancy next to an existing one is bounded by the speed of the asymmetric part of the discrepancy random walk, i.e. {\small$\lesssim N^{3/2}$}. Thus, if one thinks of creating a new discrepancy as moving the existing discrepancy according to a simple random walk at speed {\small$N^{3/2}$} (and possibly creating one at the point that the random walk has left), then for \emph{any} discrepancy to travel from {\small$\llbracket\mathrm{C}N^{3/2+\delta},\infty\rrbracket$} to {\small$\llbracket1,\mathrm{C}N\log N\rrbracket$} (for large {\small$\mathrm{C}>0$}), we only need \emph{one} of the random walk events from the previous bullet point.
\end{itemize}
Ultimately, we deduce that the probability of having a discrepancy by time {\small$\mathrm{O}(1)$} that appears at a point {\small$|\w|\lesssim N\log N$} is exponentially small in {\small$N$} as long as {\small$\delta>0$} is fixed. This completes the proof.
\end{proof}
\subsection{Evolution equation for {\small$\mathbf{Z}^{N,\half}$}}
In order to state an equation for the dynamics of {\small$\mathbf{Z}^{N,\half}$}, we require some notation. We start with the following discrete approximation to the Robin Laplacian on {\small$[0,\infty)$} and its associated heat kernel. We invite the reader to compare Definitions \ref{definition:robinheat} and \ref{definition:robinhalf}.
\begin{definition}\label{definition:robinhalf}
\fsp Consider any function {\small$\varphi:\Z_{\geq0}\to\R$}. We define the discrete Robin Laplacian {\small$\Delta_{\mathbf{A}}$} via
\begin{align}
\Delta_{\mathbf{A}}\varphi_{\x}&:=\begin{cases}\varphi_{\x+1}+\varphi_{\x-1}-2\varphi_{\x}&\x\geq1\\\varphi_{1}-\varphi_{0}+\tfrac{\lambda \mathbf{A}}{N}\varphi_{0}&\x=0\end{cases}\label{eq:robinhalfI}
\end{align}
We let {\small$\mathbf{H}^{N,\half}_{\s,\t,\x,\y}$} be a function of {\small$(\s,\t,\x,\y)\in[0,\infty)^{2}\times\Z_{\geq0}^{2}$} with {\small$\s\leq\t$} such that 
\begin{align}
\partial_{\t}\mathbf{H}^{N}_{\s,\t,\x,\y}=\tfrac12N^{2}\Delta_{\mathbf{A}}\mathbf{H}^{N}_{\s,\t,\x,\y}\quad\text{and}\quad\mathbf{H}^{N}_{\s,\s,\x,\y}=\mathbf{1}_{\x=\y}.\label{eq:robinhalfII}
\end{align}
We clarify that the Robin Laplacian acts on {\small$\x$}, and the differential equation is for {\small$\s<\t$}.
\end{definition}
\begin{lemma}\label{lemma:mshehalf}
\fsp With notation to be explained afterwards, for any {\small$\t\in[0,\infty)$} and {\small$\x\in\Z_{\geq0}$}, we have 
\begin{align}
\d\mathbf{Z}^{N,\half}_{\t,\x}=\tfrac12N^{2}\Delta_{\mathbf{A}}\mathbf{Z}^{N,\half}_{\t,\x}\d\t+\mathbf{Z}^{N,\half}_{\t,\x}\d\mathscr{Q}_{\t,\x}+\lambda \mathbf{1}_{\x=0}N\mathfrak{f}_{\mathrm{left}}[\eta^{\half}_{\t}]\mathbf{Z}^{N,\half}_{\t,\x}\d\t+\mathbf{1}_{\x=0}N^{\frac12}\mathfrak{b}_{\mathrm{left}}[\eta^{\half}_{\t}]\mathbf{Z}^{N,\half}_{\t,\x}\d\t.\label{eq:mshehalfI}
\end{align}
%
\begin{itemize}
\item The functions {\small$\mathfrak{f}_{\mathrm{left}}$} and {\small$\mathfrak{b}_{\mathrm{left}}$} are from Lemma \ref{lemma:msheleft}.
\item The process {\small$\mathscr{Q}$} is given as in Lemmas \ref{lemma:mshebulk} and \ref{lemma:msheleft} (but replacing {\small$\eta$} therein by {\small$\eta^{\half}$}).
\end{itemize}
Thus, by the Duhamel formula, we have 
\begin{align}
\mathbf{Z}^{N,\half}_{\t,\x}&=\sum_{\y\in\Z_{\geq0}}\mathbf{H}^{N,\half}_{0,\t,\x,\y}\mathbf{Z}^{N,\half}_{0,\y}+\int_{0}^{\t}\sum_{\y\in\Z_{\geq0}}\mathbf{H}^{N,\half}_{\s,\t,\x,\y}\mathbf{Z}^{N,\half}_{\s,\y}\d\mathscr{Q}_{\s,\y}\label{eq:mshehalfIIa}\\
&+\lambda \int_{0}^{\t}\mathbf{H}^{N}_{\s,\t,\x,0}\cdot N\mathfrak{f}_{\mathrm{left}}[\eta^{\half}_{\s}]\mathbf{Z}^{N,\half}_{\s,0}\d\s+\int_{0}^{\t}\mathbf{H}^{N}_{\s,\t,\x,0}\cdot N^{\frac12}\mathfrak{b}_{\mathrm{left}}[\eta^{\half}_{\s}]\mathbf{Z}^{N,\half}_{\s,0}\d\s.\label{eq:mshehalfIIb}
\end{align}
\end{lemma}
\begin{proof}
Since the dynamics of {\small$\mathbf{Z}^{N,\half}$} are local in space, it suffices to directly use the calculations in Lemmas \ref{lemma:mshebulk} and \ref{lemma:msheleft}. (The point is that there is no right boundary in the half-space, so we can disregard Lemma \ref{lemma:msheright}.)
\end{proof}
\subsection{Tightness for {\small$\mathbf{Z}^{N,\half}$}}
Following the proof of Corollary \ref{corollary:tightness}, it suffices to prove the following collection of moment estimates (which are analogues of Lemma \ref{lemma:moments}) in order to establish tightness.
\begin{lemma}\label{lemma:momentshalf}
\fsp Fix any {\small$p\geq1$}. For any {\small$\s,\t\in[0,1]$} and {\small$\x,\y\in\Z_{\geq0}$} with {\small$\x\neq\y$} and {\small$|\x-\y|\lesssim N$}, we have the estimates below, in which {\small$\kappa_{p}>0$} is large and depends only on {\small$p\geq1$}:
\begin{align}
\exp(-\tfrac{\kappa_{p}|\x|}{N})\E|\mathbf{Z}^{N,\half}_{\t,\x}|^{2p}&\lesssim_{p}1,\label{eq:momentshalfI}\\
\exp(-\tfrac{\kappa_{p}|\x|}{N})\E|\mathbf{Z}^{N,\half}_{\t,\x}-\mathbf{Z}^{N,\half}_{\t,\y}|^{2p}&\lesssim_{p}N^{-p}|\x-\y|^{p},\label{eq:momentshalfII}\\
\exp(-\tfrac{\kappa_{p}|\x|}{N})\E|\mathbf{Z}^{N,\half}_{\t,\x}-\mathbf{Z}^{N,\half}_{\s,\x}|^{2p}&\lesssim_{p}|\t-\s|^{\frac12p}.\label{eq:momentshalfIII}
\end{align}
In particular, as with Corollary \ref{corollary:tightness}, we deduce that the process {\small$(\t,\X)\mapsto\mathbf{Z}^{N,\half}_{\t,N\X}$} from Theorem \ref{theorem:mainhalf} is tight in the large-{\small$N$} limit in the Skorokhod space {\small$\mathscr{D}([0,1],\mathscr{C}([0,\infty)))$} with continuous limit points.
\end{lemma}
\begin{proof}
For any {\small$p\geq1$}, we again use the notation {\small$\|\cdot\|_{p}:=(\E|\cdot|^{p})^{1/p}$}. We also borrow notation from \eqref{eq:momentsI0}. We claim that the estimates 
\begin{align}
\|\int_{0}^{\t}\mathbf{H}^{N,\half}_{\s,\t,\x,0}\cdot N\mathfrak{m}_{\mathrm{left}}[\eta^{\half}_{\s}]\mathbf{Z}^{N,\half}_{\s,0}\d\s\|_{2p}^{2}&\lesssim_{p}\int_{0}^{\t}|\t-\s|^{-\frac12}\|\mathbf{Z}^{N,\half}_{\s,0}\|_{2p}^{2}\d\s,\label{eq:momentshalf0a}\\
\|\int_{0}^{\t}(\mathbf{H}^{N,\half}_{\s,\t,\x,0}-\mathbf{H}^{N,\half}_{\s,\t,\y,0})\cdot N\mathfrak{m}_{\mathrm{left}}[\eta^{\half}_{\s}]\mathbf{Z}^{N,\half}_{\s,0}\|_{2p}^{2}&\lesssim_{p}\int_{0}^{\t}|\t-\s|^{-\frac34}\|\mathbf{Z}^{N,\half}_{\s,0}\|_{2p}^{2}\d\s\cdot N^{-1}|\x-\y|,\label{eq:momentshalf0b}
\end{align}
and 
\begin{align}
&\|\int_{0}^{\t_{2}}\mathbf{H}^{N,\half}_{\s,\t_{2},\x,0}\cdot N\mathfrak{m}_{\mathrm{left}}[\eta_{\s}]\mathbf{Z}^{N,\half}_{\s,0}\d\s-\int_{0}^{\t_{1}}\mathbf{H}^{N,\half}_{\s,\t_{1},\x,0}\cdot N\mathfrak{m}_{\mathrm{left}}[\eta_{\s}]\mathbf{Z}^{N,\half}_{\s,0}\d\s\|_{2p}^{2}\nonumber\\
&\lesssim_{p}|\t_{2}-\t_{1}|^{\frac12}\cdot\sup_{\t\in[0,1]}\|\mathbf{Z}^{N,\half}_{\t,0}\|_{2p}^{2}.\label{eq:momentshalf0c}
\end{align}
These follow from the same argument that gave \eqref{eq:moments0}, \eqref{eq:moments0c}, and \eqref{eq:moments0e}. (Indeed, these arguments use only heat kernel estimates for the interval heat kernel {\small$\mathbf{H}^{N}$}, which are also available for {\small$\mathbf{H}^{N,\half}$}; see Propositions \ref{prop:hke} and \ref{prop:hkehalf}.) Now, if we combine \eqref{eq:momentshalf0a}-\eqref{eq:momentshalf0c} with (32) in \cite{P}, we obtain
\begin{align}
\exp(-\tfrac{\kappa_{p}|\x|}{N})\|\mathbf{Z}^{N,\half}_{\t,\x}\|_{2p}^{2}&\lesssim_{p}1+\int_{0}^{\t}|\t-\s|^{-\frac12}\sup_{\x\in\Z_{\geq0}}\exp(-\tfrac{\kappa_{p}|\x|}{N})\|\mathbf{Z}^{N,\half}_{\s,\x}\|_{2p}^{2}\d\s.\nonumber
\end{align}
At this point, we can use Gronwall as in the proof of Lemma \ref{lemma:moments} to deduce \eqref{eq:momentshalfI} (with a possibly different {\small$\kappa_{p}$}). To obtain the regularity estimates \eqref{eq:momentshalfII}-\eqref{eq:momentshalfIII}, we again follow the proof of Proposition 5.4 in \cite{P} as well as \eqref{eq:momentshalf0b}-\eqref{eq:momentshalf0c} combined with \eqref{eq:momentshalfI} (as in the proof of Lemma \ref{lemma:momentshalf}).
\end{proof}
\subsection{Identification of limit points}
Consider the following space of Robin test functions:
\begin{align*}
\mathscr{C}^{\infty}_{\mathbf{A}}:=\Big\{\varphi\in\mathscr{C}^{\infty}_{\mathrm{c}}(\R): \varphi'_{0}=-\lambda \mathbf{A}\varphi_{0}\Big\}.
\end{align*}
Next, we define the following pairing between any {\small$\varphi\in\mathscr{C}^{\infty}_{\mathbf{A}}$} and any function {\small$\psi:\Z_{\geq0}\to\R$}:
\begin{align}
(\psi,\varphi)_{N,\half}:=\tfrac{1}{N}\sum_{\x\in\Z_{\geq0}}\psi_{\x}\varphi_{N^{-1}\x}.
\end{align}
We claim that the following two processes are martingales if {\small$\varphi\in\mathscr{C}^{\infty}_{\mathbf{A}}$}:
\begin{align}
\mathscr{N}^{N,\half}_{\t}&:=(\mathbf{Z}^{N,\half}_{\t,\cdot},\varphi)_{N,\half}-(\mathbf{Z}^{N,\half}_{0,\cdot},\varphi)_{N,\half}-\int_{0}^{\t}(\mathbf{Z}^{N,\half}_{\s,\cdot},\tfrac12\varphi'')_{N,\half}\d\s-\sum_{\k=0,1,2}\mathbf{R}^{N,\half}_{\k,\t}[\varphi],\nonumber\\
\mathscr{Q}^{N,\half}_{\t}&:=(\mathscr{N}^{N,\half}_{\t})^{2}-\lambda^{2} \int_{0}^{\t}(|\mathbf{Z}^{N,\half}_{\s,\cdot}|^{2},|\varphi|^{2})_{N,\half}\d\s-\mathbf{R}^{N,\half}_{3,\t}[\varphi]-\mathbf{R}^{N,\half}_{4,\t}[\varphi].\nonumber
\end{align}
Above, {\small$\mathbf{R}^{N,\half}_{\k,\t}[\varphi]$} vanish locally uniformly in {\small$\t\geq0$} in probability as {\small$N\to\infty$} for {\small$\k=0,4$}, and 
\begin{align}
\mathbf{R}_{1,\t}^{N,\half}[\varphi]&:=\lambda \int_{0}^{\t}\mathfrak{f}_{\mathrm{left}}[\eta^{\half}_{\s}]\mathbf{Z}^{N,\half}_{\s,0}\cdot\varphi_{0}\d\s,\label{eq:identifyhalf1a}\\
\mathbf{R}_{2,\t}^{N,\half}[\varphi]&:=\int_{0}^{\t}N^{-\frac12}\mathfrak{b}_{\mathrm{left}}[\eta^{\half}_{\s}]\mathbf{Z}^{N,\half}_{\s,0}\cdot\varphi_{0}\d\s,\label{eq:identifyhalf1b}\\
\mathbf{R}_{3,\t}^{N,\half}[\varphi]&:=-\tfrac{\lambda^{2}}{2}\int_{0}^{\t}\tfrac{1}{N}\sum_{\x\geq0}\eta_{\s,\x}\eta_{\s,\x+1}|\mathbf{Z}^{N}_{\s,\x}|^{2}|\varphi_{N^{-1}\x}|^{2}\d\s.\label{eq:identifyhalf1c}
\end{align}
The proof of the martingale property for {\small$\mathscr{N}^{N,\half}$} follows by the same argument as what gave us the martingale property for \eqref{eq:identify3a}, except we drop the terms supported at the right boundary {\small$N$} therein (since they do not appear in \eqref{eq:mshehalfI}). The martingale property for {\small$\mathscr{Q}^{N,\half}$} follows by the same argument as what gave the martingale property for \eqref{eq:identity5} (except, again, we drop terms supported at the right boundary {\small$N$} therein, though these only contribute to an error term anyway). Thus, uniqueness (in law) of solutions to the martingale problem for \eqref{eq:openshehalf}-\eqref{eq:openshehalfII}, which is an analogue of Proposition \ref{prop:mgproblem} (see also Proposition 5.9 in \cite{CS}), implies that to complete the proof of Theorem \ref{theorem:mainhalf}, it suffices to show the following analogue of Proposition \ref{prop:stoch}.
\begin{prop}\label{prop:stochhalf}
\fsp We have the following vanishes in probability in the large-{\small$N$} limit for {\small$\k=1,2$}:
\begin{align}
\lim_{N\to\infty}\sup_{\t\in[0,1]}|\mathbf{R}^{N,\half}_{\k,\t}[\varphi]|=0.\label{eq:stochhalfI}
\end{align}
Second, for any {\small$\x\geq0$}, let {\small$\mathfrak{a}_{\x}:\{\pm1\}^{\mathbb{K}_{N}}\to\R$} be a function satisfying the properties (1)-(3) from Proposition \ref{prop:stoch} (but replacing {\small$\llbracket0,N\rrbracket$} and {\small$\mathbb{K}_{N}$} by {\small$\Z_{\geq0}$} and {\small$\Z_{>0}$}, respectively). Then, we have the following, in which {\small$\wt{\varphi}\in\mathscr{C}^{\infty}_{\mathrm{c}}(\R)$} is arbitrary:
\begin{align}
\lim_{N\to\infty}\sup_{\t\in[0,1]}\Big|\int_{0}^{\t}\tfrac{1}{N}\sum_{\x\in\Z_{\geq0}}\mathfrak{a}_{\x}[\eta_{\s}]|\mathbf{Z}^{N,\half}_{\s,\x}|^{2}\cdot\wt{\varphi}_{N^{-1}\x}\d\s\Big|=0.\label{eq:stochhalfII}
\end{align}
\end{prop}
\subsection{Proof of Proposition \ref{prop:stochhalf}}
This follows from the same argument for the proof of Proposition \ref{prop:stoch}. In particular, we require only Proposition \ref{prop:kv} (which requires no adaptations since it is about a local process and thus is independent of the geometry on which the particle system of interest lives on), as well as the following analogue of Lemma \ref{lemma:localeq}.
\begin{lemma}\label{lemma:localeqhalf}
\fsp Recall that {\small$\delta>0$} is a small parameter. For {\small$\x\in\llbracket0,N\rrbracket$}, let {\small$\mathfrak{f}_{\x}:\{\pm1\}^{\Z_{>0}}\to\R$} satisfy:
\begin{enumerate}
\item We have that {\small$\mathfrak{f}_{\x}$} is uniformly bounded.
\item There exists {\small$\mathfrak{m}_{N}>0$} depending possibly on {\small$N$} so that for any {\small$\x\in\llbracket0,N\rrbracket$} and {\small$\eta\in\{\pm1\}^{\Z_{>0}}$}, the quantity {\small$\mathfrak{f}_{\x}[\eta]$} depends only on {\small$\eta_{\w}$} for {\small$\w\in\mathbb{L}_{\x}$}, where {\small$\mathbb{L}_{\x}$} is an interval containing {\small$\x$} of size at most {\small$\mathfrak{m}_{N}$}.
\end{enumerate}
Then, we have the following estimate for any deterministic {\small$\mathfrak{t}>0$}:
\begin{align}
\int_{0}^{\mathfrak{t}}\tfrac{1}{N}\sum_{|\x|\lesssim N}\E|\mathfrak{f}_{\x}[\eta^{\half}_{\s}]|\d\s\lesssim_{\mathfrak{t},\delta}\sup_{|\x|\lesssim N}\sup_{\sigma\in[-1,1]}\E^{\sigma,\mathbb{L}_{\x}}|\mathfrak{f}_{\x}|+N^{-\frac32+\delta}\mathfrak{m}_{N}^{3}.\label{eq:localeqhalfI}
\end{align}
We now assume that {\small$\mathfrak{d}:\{\pm1\}^{\Z_{>0}}\to\R$} is uniformly bounded, and that {\small$\mathfrak{d}[\eta]$} depends only on {\small$\eta_{\w}$} for {\small$\w\in\mathbb{L}_{\mathfrak{d}}$}, where {\small$\mathbb{L}_{\mathfrak{d}}\subseteq\Z_{>0}$} is an interval that contains {\small$1$}. Then, for any deterministic {\small$\mathfrak{t}>0$}, we have 
\begin{align}
\int_{0}^{\mathfrak{t}}\E|\mathfrak{d}[\eta^{\half}_{\s}]|\d\s\lesssim_{\delta}\E^{0}|\mathfrak{d}|+N^{-\frac12+\delta}|\mathbb{L}_{\mathfrak{d}}|^{2}.\label{eq:localeqhalfII}
\end{align}
\end{lemma}
Before we prove this lemma (and therefore complete the proof of Theorem \ref{theorem:mainhalf}), we note the extra factor of {\small$N^{\delta}$} in the last terms in \eqref{eq:localeqhalfI}-\eqref{eq:localeqhalfII} are harmless. Indeed, the proof of Proposition \ref{prop:stochhalf} will anyways estimate the contribution of these error terms by a negative power of {\small$N$}, so we can afford to give up a factor of {\small$N^{\delta}$} if {\small$\delta>0$} is as small as we like (see the proof of Proposition \ref{prop:stoch} for details).
\begin{proof}
If we use the argument in the proof of Lemma \ref{lemma:localeq}, then to prove Lemma \ref{lemma:localeqhalf}, it suffices to get an analogue of \eqref{eq:entropyproductionI}. Indeed, the proof of Lemma \ref{lemma:localeqhalf} is a calculation with product measures on {\small$\{\pm1\}^{\mathbb{K}_{N}}$} (which carry over from {\small$\mathbb{K}_{N}$} to any discrete interval), and the only external input is the Dirichlet form estimate \eqref{eq:entropyproductionI}. 

To be precise, we recall that {\small$\mathfrak{P}_{\t}$} is the density with respect to {\small$\mathbb{P}^{0}$} for the law of {\small$\eta^{\half}_{\t}\in\{\pm1\}^{\Z_{>0}}$}. Now, consider the following Fisher information for {\small$\mathfrak{P}_{\t}$} (which is analogous to \eqref{eq:fi}):
\begin{align}
\mathfrak{D}^{0,\half}[\mathfrak{P}_{\t}]:=\Big\{\sum_{\x\in\llbracket1,\infty\rrbracket}\E^{0}[(\mathscr{L}_{\x}\sqrt{\mathfrak{P}_{\t}})^{2}]\Big\}+\E^{0}|\mathscr{S}_{1}\sqrt{\mathfrak{P}_{\t}}|^{2}.\nonumber
\end{align}
We now claim that the following holds for small {\small$\delta>0$} and any deterministic $\t\geq0$:
\begin{align}
\int_{0}^{\t}\mathfrak{D}^{0,\half}[\mathfrak{P}_{\s}^{}]\d\s\lesssim_{\t} N^{-2}\E^{0}\mathfrak{P}_{0}^{}\log\mathfrak{P}_{0}^{}+N^{-\frac12}\lesssim_{} N^{-\frac12+\delta}.\label{eq:entropyhalfI}
\end{align}
The first bound in the second estimate \eqref{eq:entropyhalfI} holds by following the same argument that gave Lemma \ref{lemma:entropyproduction}. The second bound in \eqref{eq:entropyhalfI} holds since {\small$\mathfrak{P}^{}$} is actually a density with respect to the marginal of {\small$\mathbb{P}^{0}$} on {\small$\{\pm1\}^{\mathbb{J}_{N}}$}, because {\small$\mathfrak{P}^{\mathrm{cut}}_{0}$} is defined as a measure on {\small$\{\pm1\}^{\mathbb{J}_{N}}$} tensored with the marginal of {\small$\mathbb{P}^{0}$} on {\small$\{\pm1\}^{\Z_{>0}\setminus \mathbb{J}_{N}}$}. See Remark \ref{remark:cut}. Now, use that any probability density with respect to {\small$\mathbb{P}^{0}$} on {\small$\{\pm1\}^{\mathbb{J}_{N}}$} must be deterministically {\small$\lesssim\log|\{\pm1\}^{\mathbb{J}_{N}}|\lesssim |\mathbb{J}_{N}|\lesssim N^{3/2+\delta}$}. This implies the last estimate in \eqref{eq:entropyhalfI}. As we mentioned at the beginning of this argument, Lemma \ref{lemma:localeqhalf} (and thus Theorem \ref{theorem:mainhalf}) follows from \eqref{eq:entropyhalfI}, so we are done.
\end{proof}
\appendix
\section{Heat kernel estimates}
We record below heat kernel estimates that come from Propositions 3.15, 3.16, and 3.21 in \cite{P}. (We will not include their proofs, since these estimates are taken directly from \cite{P}. However, we note that the heat kernel in \cite{P} does not have time scaled by {\small$N^{2}$}, and {\small$N$} is denoted by {\small$\e^{-1}$} therein. In particular, all the times in the estimates in Propositions 3.15, 3.16, and 3.21 in \cite{P} must be scaled by {\small$N^{2}$} to agree with the convention in this paper.)
\begin{prop}\label{prop:hke}
\fsp Fix any {\small$\s,\t,\r\in[0,\infty)$} satisfying {\small$\s\leq\t\leq\r$}, and fix any {\small$\x,\y,\w\in\llbracket0,N\rrbracket$}. We have 
\begin{align}
\mathbf{H}^{N}_{\s,\t,\x,\y}&\lesssim N^{-1}|\t-\s|^{-\frac12},\nonumber\\
|\mathbf{H}^{N}_{\s,\t,\x,\y}-\mathbf{H}^{N}_{\s,\t,\w,\y}|&\lesssim N^{-\frac32}|\t-\s|^{-\frac34}|\x-\w|^{\frac12},\nonumber\\
|\mathbf{H}^{N}_{\s,\t,\x,\y}-\mathbf{H}^{N}_{\s,\r,\x,\y}|&\lesssim N^{-1}|\t-\s|^{-\frac34}|\t-\r|^{\frac14}.\nonumber
\end{align}
\end{prop}
We now record estimates for the half-space heat kernel {\small$\mathbf{H}^{N,\half}$}. These come from Propositions 3.1, 3.2 and 3.5 in \cite{P} in the same way as how Proposition \ref{prop:hke} follows from Propositions 3.15, 3.16, and 3.21 in \cite{P}.
\begin{prop}\label{prop:hkehalf}
\fsp Fix any {\small$\s,\t,\r\in[0,\infty)$} satisfying {\small$\s\leq\t\leq\r$}, and fix any {\small$\x,\y,\w\in\Z_{\geq0}$}. We have 
\begin{align}
\mathbf{H}^{N,\half}_{\s,\t,\x,\y}&\lesssim N^{-1}|\t-\s|^{-\frac12},\nonumber\\
|\mathbf{H}^{N,\half}_{\s,\t,\x,\y}-\mathbf{H}^{N,\half}_{\s,\t,\w,\y}|&\lesssim N^{-\frac32}|\t-\s|^{-\frac34}|\x-\w|^{\frac12},\nonumber\\
|\mathbf{H}^{N,\half}_{\s,\t,\x,\y}-\mathbf{H}^{N,\half}_{\s,\r,\x,\y}|&\lesssim N^{-1}|\t-\s|^{-\frac34}|\t-\r|^{\frac14}.\nonumber
\end{align}
\end{prop}
%
%
%


\end{document}